\documentclass[reqno,a4paper]{amsart}

\usepackage{epstopdf}
\usepackage{color} 
\usepackage[usenames,dvipsnames]{xcolor}
\usepackage{transparent}
\usepackage[latin1]{inputenc}  
\usepackage{dsfont}
\usepackage{gensymb}
\usepackage[active]{srcltx}\usepackage[colorlinks,pdfpagelabels,pdfstartview=FitH,bookmarksopen=true,bookmarksnumbered=true,linkcolor=blue,plainpages=false,hypertexnames=false,citecolor=red]{hyperref}
\usepackage{mathrsfs}
\usepackage{verbatim, amssymb,amsmath, amsfonts, amsthm}
\usepackage{latexsym}
\usepackage{float}

\usepackage{multicol}
\usepackage{subcaption}
\usepackage{graphicx}
\usepackage{caption}

\usepackage{a4wide} 

\allowdisplaybreaks 

\newtheorem{theorem}{Theorem}[section]
\newtheorem{proposition}[theorem]{Proposition}
\newtheorem{lemma}[theorem]{Lemma}
\newtheorem{corollary}[theorem]{Corollary}
\newtheorem{remark}[theorem]{Remark}
\newtheorem{definition}[theorem]{Definition}

\newtheorem{conjecture}[theorem]{Conjecture}
\newtheorem*{theorem*}{Theorem}



\newcommand{\N}{\mathbb{N}}

\newcommand{\R}{\mathbb{R}}
\renewcommand{\to}{\rightarrow}

\def\sideremark#1{\ifvmode\leavevmode\fi\vadjust{\vbox to0pt{\vss
 \hbox to 0pt{\hskip\hsize\hskip1em
 \vbox{\hsize2.1cm\tiny\raggedright\pretolerance10000
  \noindent #1\hfill}\hss}\vbox to15pt{\vfil}\vss}}}%

\newcommand{\apice}[2]{\overset{\scriptscriptstyle{#2}}{{#1}}}

\begin{document}
\numberwithin{equation}{section}
\parindent=0pt
\hfuzz=2pt
\frenchspacing

\title[Radial solutions]{Sharp asymptotic behavior of radial solutions of some planar semilinear elliptic problems}

\author[]{Isabella Ianni}

\address{Isabella Ianni, Universita degli studi della Campania \emph{Luigi Vanvitelli}, V.le Lincoln 5, 81100 Caserta, Italy.
\texttt{isabella.ianni@unicampania.it}}

\author[]{Alberto Salda\~{n}a}
\address{Alberto Salda\~{n}a, Instituto de Matem\'{a}ticas, Universidad Nacional Aut\'{o}noma de M\'{e}xico,Circuito Exterior, Ciudad Universitaria, 04510 Coyoac\'{a}n, Ciudad de M\'{e}xico, M\'{e}xico. \texttt{alberto.saldana@im.unam.mx}}

\thanks{2010 \textit{Mathematics Subject classification:} 35B05, 35B06, 35J91. }

\thanks{ \textit{Keywords}: H\'enon equation, Lane-Emden equation, sign-changing radial solutions, asymptotic analysis, a priori bounds, Morse index.}

\thanks{Research partially supported by:  PRIN  $2017$JPCAPN$\_003$ grant and INDAM - GNAMPA}

\begin{abstract} We consider the equation $-\Delta u= |x|^{\alpha}|u|^{p-1}u$ for any $\alpha\geq 0$, 
	either in $\R^2$ or in the unit ball $B$ of $\R^2$ centered at the origin with Dirichlet or Neumann boundary conditions. We give a sharp description of the asymptotic behavior as $p\rightarrow +\infty$ of all the radial solutions to these problems and we show that there is no uniform a priori bound (in $p$) for nodal solutions under Neumann or Dirichlet boundary conditions. This contrasts with the existence of uniform bounds for positive solutions, as recently shown in \cite{KamburovSirakov} for $\alpha=0$ and Dirichlet boundary conditions.\end{abstract}

\maketitle

\section{Introduction}
We consider  the equation
\begin{equation}\label{problemHenonWhole}
-\Delta u=|x|^{\alpha} |u|^{p-1}u\qquad  \mbox{ in }\R^2,
\end{equation}
where  $\alpha \geq 0$ and $p>1$, we also consider the problem in the unit ball $B$ of $\R^2$ centered at the origin
\begin{equation}\label{equationHenon}-\Delta u= |x|^{\alpha}|u|^{p-1}u\qquad  \mbox{ in }B,
\end{equation}  together with either  Dirichlet boundary conditions
\begin{equation}\label{DirichletBC}
u=0\qquad\text{ on }\partial B
\end{equation}
or  Neumann boundary conditions
\begin{equation}\label{NeumannBC}
\partial_{\nu}u=0\qquad\text{ on }\partial B.
\end{equation}
When $	\alpha=0$ these are commonly known as \emph{Lane-Emden} problems, while when $\alpha>0$ they 
were introduced first by M. H\'enon in \cite{Henon} and so are usually referred to as  \emph{H\'enon} problems. As we will see the existence of radial solutions for all these planar problems can be easily proved for any $p>1$.
\\\\
In this work we give a sharp description of the asymptotic behavior of all the radial solutions to the problems \eqref{problemHenonWhole}, \eqref{equationHenon}-\eqref{DirichletBC} and \eqref{equationHenon}-\eqref{NeumannBC} as the exponent  $p\rightarrow +\infty$, for any fixed $\alpha\geq 0$.
\\
\\
The interest in the study of the asymptotic behavior as $p\rightarrow +\infty$ for  these $2$-dimensional problems started from the seminal works by Ren and Wei (\cite{RenWeiTAMS1994,RenWeiPAMS1996}), which concern the study of the least energy solutions  for the Lane-Emden equation ($\alpha=0$) in general smooth bounded domains, under  Dirichlet boundary conditions. In this case, the least energy solutions $u_p$ (which are \emph{positive}) exhibit a single point concentration  and $u_p\rightarrow 0$ locally uniformly outside the concentration point as $p\rightarrow +\infty$. Moreover, differently from the almost critical higher dimensional case (which is much more studied, see for instance  \cite{BahriLiRey,BrezisPeletier,Han,Rey,Struwe}), these solutions do not blow-up but  stay uniformly  bounded (in $p$) in $L^{\infty}$ and  (see \cite{AdimurthiGrossi})   $\|u_p\|_{L^{\infty}}\rightarrow \sqrt{e}$ as $p\rightarrow +\infty$.\\\\  Recently, an \emph{a priori uniform bound} (in $p$) has been proven in  \cite{KamburovSirakov}  \emph{for any positive solution} of the Dirichlet Lane-Emden problem, in any smooth bounded planar domain. More precisely it has been proved that, for any $p_0>1$, there exists a constant $C>0$, which depends only on the domain and $p_0$, such that, for all $p\geq p_0$, any solution $u_p$ of the  Lane-Emden problem with Dirichlet condition on the boundary of the domain satisfies that
\begin{equation}\label{aPrioriBoundPositive}
\|u_p\|_{L^{\infty}}\leq C.
\end{equation}
Moreover in \cite{DIPpositive} a complete description of the asymptotic behavior of any positive solution of the Dirichlet Lane-Emden problem has been given, in any smooth bounded planar domain and under a uniform energy bound assumption, showing simple concentration at a finite number of distinct points, convergence to $0$ locally uniformly outside the concentration set, convergence of the $L^{\infty}$-norm to $\sqrt e$ and energy quantization to integer multiples of the value $8\pi e$ in the limit as $p\rightarrow +\infty$ (for the sharp quantization result and $L^{\infty}$-norm limit behavior see \cite{DeMarchisGrossiIanniPacellaImprovement} and also \cite{Thizy}).
\\
In particular, when the domain is the ball $B$ centered at $0$, then  the Dirichlet Lane-Emden problem (i.e. \eqref{equationHenon}-\eqref{DirichletBC} with $\alpha=0$) admits a unique positive solution  (which is the least energy) which is radial and strictly decreasing in the radial variable (by the  symmetry result in \cite{GNN}), so the unique concentration point is necessarily the origin, which is the maximum point, and one has that
\begin{equation}\label{conv11}u_p(0)=\|u_p\|_{L^{\infty}}\rightarrow \sqrt e\qquad \text{as $p\rightarrow +\infty$},
\end{equation}
\begin{equation}\label{conv111}pu_p(x)\rightarrow 4\sqrt{e}\log\frac{1}{|x|} \qquad \mbox{ in }C^1_{loc}(\bar B\setminus\{0\})\text{ as $p\rightarrow +\infty$},
\end{equation}
from which the following convergence of the derivative on the boundary directly follows
\begin{align*}
p|(u_p)'(1)|\rightarrow 4\sqrt e\qquad \text{as $p\rightarrow +\infty$};
\end{align*}
furthermore, the energy satisfies that
\begin{equation}\label{energi11}p\int_0^1|(u_p)'(r)|^2r\, dr\rightarrow 4 e\qquad \text{as $p\rightarrow +\infty$}.
\end{equation}  
From \cite{AdimurthiGrossi} we also know that a suitable scaling of $u_p$ converges to a radial solution $Z_0$ of the Liouville equation in the whole $\mathbb R^2$

\begin{equation}	\label{intoLiou}
\left\{
\begin{array}{lr}
	\ \ \ \ -\Delta Z_0=e^{Z_0}\quad\mbox{ in }\R^2,\\
\ \ \ \ \ Z_0(0)=0,\\
\int_{\R^2}e^{Z_0}dx= 8\pi.
\end{array}
\right.
\end{equation}

\begin{center}

\setlength{\unitlength}{1cm}
\thicklines
\begin{picture}(5,4.5)
\put(5.8,3.7){$r$}
\put(-1,0.3){$Z_0(r)$}
\includegraphics[width=.4\textwidth]{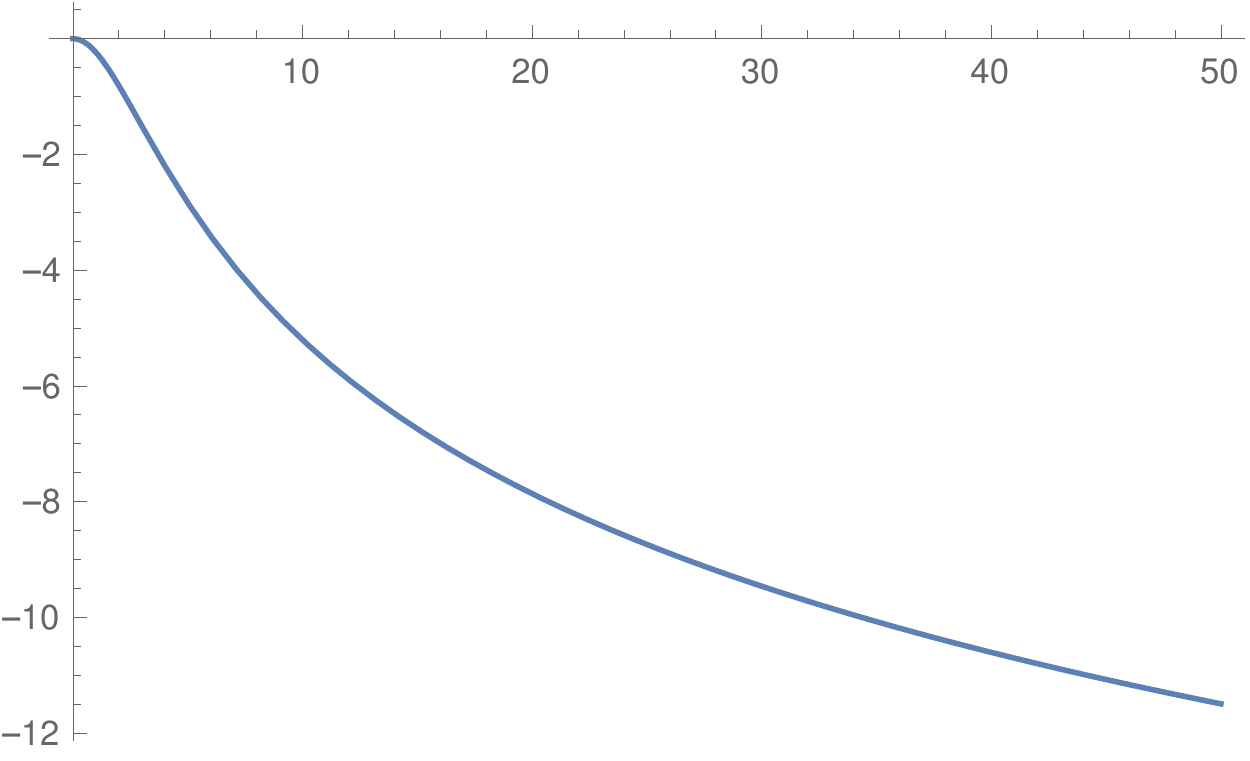}
\end{picture}

\captionof{figure}{The solution of \eqref{intoLiou}.}
\label{f1}	
\end{center}

\

\

Concerning \emph{sign-changing} solutions of the planar Dirichlet Lane-Emden problem, from \cite{DeMarchisIanniPacellaJEMS}  it is known that, under energy uniform bounds in $p$, nodal solutions are uniformly bounded in $p$,  concentrate at a finite number of points and converge to zero locally uniformly outside the concentration set as $p\rightarrow +\infty$ (see also \cite{GrossiGrumiauPacella1}, where low energy sign-changing solutions have been studied).
\\ 
\\
In this paper we show that an a prior bound as in  \eqref{aPrioriBoundPositive}  does \emph{not} hold in general for sign-changing solutions (see Theorem \ref{theorem:NObounds}).  Moreover, differently from the positive case,   as $p\rightarrow +\infty$ the concentration may not be simple and a tower of bubbles may appear as shown in \cite{GrossiGrumiauPacella2}  in the ball $B$ and later generalized to other symmetric domains in \cite{DeMarchisIanniPacellaJEMS}. \\
In particular, when the domain is the unit ball $B$ centered at $0$, it is known that the Dirichlet Lane-Emden problem \eqref{equationHenon}-\eqref{DirichletBC} (with $\alpha=0$) admits infinitely many radial solutions, one (up to sign) for each fixed number $m$ of nodal regions, and they all  concentrate only at the origin as $p\rightarrow +\infty$, where their absolute maximum (up to sign) is attained. 
\\
The asymptotic behavior of  the \emph{nodal radial  solutions with $m=2$ nodal regions} has been analyzed in full detail in \cite{GrossiGrumiauPacella2}. 
The nodal radius $r_p$ of these solutions shrinks to the origin as $p\rightarrow +\infty$ and there exists an explicit  constant $\theta\sim 10.374$ (so that $t:=\frac{4\sqrt e}{\theta-2}$ is the unique solution of the equation $2\sqrt e\log t+t=0$) such that 
\begin{equation}\label{conv22}|u_p(0)|=\|u_p\|_{L^{\infty}}\rightarrow \frac{\theta-2}{4}e^{\frac{2}{\theta+2}} (>\sqrt e)\qquad \text{as $p\rightarrow +\infty$},
\end{equation}
\begin{equation}
\label{conv222}
p|u_p(x)|\rightarrow (\theta+2) e^{\frac{2}{\theta+2}} \log\frac{1}{|x|} \qquad \mbox{ in }C^1_{loc}(\bar B\setminus\{0\})\text{ as $p\rightarrow +\infty$}.
\end{equation} 
From this convergence the following convergence of the derivative on the boundary directly follows
\[|p(u_p)'(1)|\rightarrow (\theta+2) e^{\frac{2}{\theta+2}}\qquad \text{as $p\rightarrow +\infty$}. \]
Moreover, in \cite{GrossiGrumiauPacella2} it is  proved that the nodal radius $r_p$ satisfies
\begin{equation}\label{r22}(r_p)^{\frac{2}{p-1}}\rightarrow \frac{4\sqrt e}{(\theta-2)}e^{-\frac{2}{\theta +2}}\qquad \text{as $p\rightarrow +\infty$}
\end{equation}
and 
\begin{equation}\label{r222}|p(u_p)'(r_p)|r_p\rightarrow  (\theta-2)e^{\frac{2}{\theta+2}}\qquad \text{as $p\rightarrow +\infty$}.
\end{equation}
Furthermore, denoting by $s_p$ the unique minimum (up to sign) of $u_p$, one has that it shrinks to $0$, 
\begin{equation}\label{s22}(s_p)^{\frac{2}{p-1}}\rightarrow e^{-\frac{2}{\theta +2}}\qquad \text{as $p\rightarrow +\infty$}
\end{equation}
and \begin{equation}\label{s222}|u_p(s_p)|\rightarrow e^{\frac{2}{\theta+2}}(<\sqrt e)\qquad \text{as $p\rightarrow +\infty$}.
\end{equation}
Finally, the asymptotic value of the  energy is explicitly determined
\begin{equation}\label{energi22}p\int_0^1|(u_p)'(r)|^2r\, dr\rightarrow \frac{(\theta+2)^2}{4} e^{\frac{4}{\theta+2}}\qquad \text{as $p\rightarrow +\infty$}.
\end{equation}  
The following is also known: a suitable scaling of $u_p^+$ converges to the solution $Z_0$ of the same limit problem  \eqref{intoLiou} already involved in the asymptotic of the positive solution, but a suitable rescaling of $u_p^-$ converges to a radial solution of a \emph{different} limit problem, that is, the singular Liouville equation
\begin{align}\label{intoLiousin}
\left\{
\begin{array}{lr}
\ \ \ \ -\Delta Z=e^Z+ 2\pi(2-\theta)\delta_{0}\quad\mbox{ in }\R^2,\\
\int_{\R^2}e^Zdx=4\pi\theta,
\end{array}
\right.
\end{align}
where  $\delta_{0}$ is the Dirac measure centered at $0$.

\begin{center}

\setlength{\unitlength}{1cm}
\thicklines
\begin{picture}(5,4.5)
\put(5.8,3.7){$r$}
\put(-0.9,0.3){$Z(r)$}
\includegraphics[width=.4\textwidth]{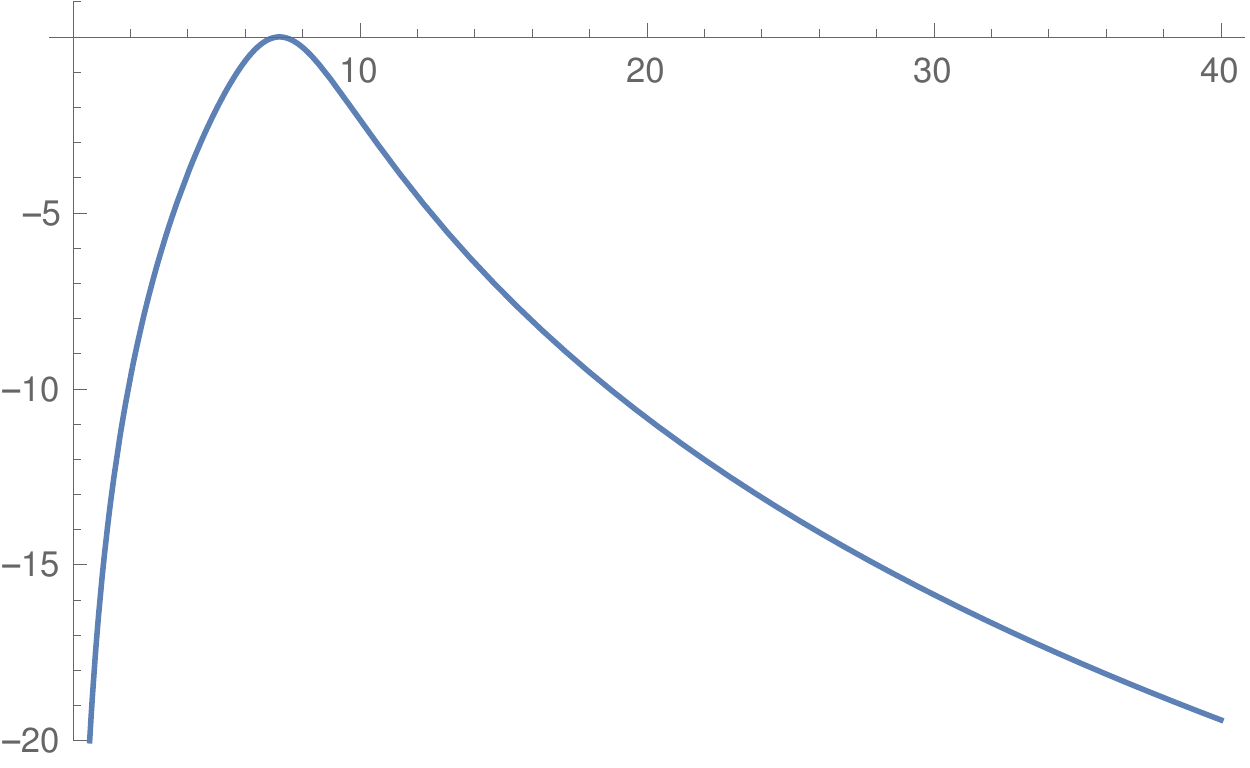}
\end{picture}

\captionof{figure}{The solution of \eqref{intoLiousin} with $\theta=10.374$.}
\label{fig2}
\end{center}

Observe that,  in \cite{DeMarchisIanniPacellaMathAnn}, this asymptotic information is the starting point to show that the Morse index of $u_p$ is $12$ when $p$ is sufficiently large (see also \cite{DeMarchisIanniPacellaAdvMath} for the higher dimensional case).\\\\
In this paper, we extend and generalize  to \emph{any} radial solution $u_p$  the asymptotic results described above.
In particular (see the case $\alpha =0$ in Theorems \ref{theorem:mainDirichletG} and \ref{theorem:analisiAsintoticaRiscalateAlpha}  below) we  show similar convergence as in \eqref{conv11}-\eqref{conv111} and \eqref{conv22}-\eqref{conv222} and describe in a precise way the asymptotic behavior of all the critical points, all the roots and the values of $u_p$ and $u_p'$ at these points respectively as in \eqref{r22}-\eqref{r222}-\eqref{s22}-\eqref{s222}, for \emph{any} radial solution $u_p$. We also study the convergence of a suitable rescaling of the solution in each nodal region and explicitly determine the asymptotic value of the energy, generalizing \eqref{energi11} and \eqref{energi22}.\\
\\
Our asymptotic analysis is carried out also for the \emph{Neumann Lane-Emden problem} (i.e. \eqref{equationHenon}-\eqref{NeumannBC} with $\alpha=0$), which is  much less studied than the Dirichlet problem.  These seem to be the first results concerning the asymptotic behavior of Neumann solutions in the $2$-dimensional case. Observe that an easy application of the divergence theorem implies that nontrivial solutions $u_p$ of the Neumann problem are necessarily sign-changing, since
 \begin{equation}\label{Neumansign}0=\int_{\partial B}\frac{\partial u_p}{\partial\nu}=-\int_B \Delta u_p=\int_B|u_p|^{p-1}u_p.
 \end{equation}
 Moreover, when the domain is the ball $B$, radial solutions cannot be ground states (see \cite[Corollary 1.4]{SaldanaTavares}, see also \cite{GiraoWeth}).
Our Theorem \ref{theoremNeumannMain}  below (case $\alpha=0$) provides a first  description in dimension $2$ of the asymptotic behavior  of \emph{all the radial solutions of the Neumann Lane-Emden problem in the ball $B$} as $p \rightarrow +\infty$. 
In particular, we obtain sharp constants for the asymptotic behavior of all the critical points, all the roots and the values of $u_p$ and $u_p'$ at these points respectively and explicitly determine the asymptotic value of the energy.
\\\\	
Our results, both for the Dirichlet and the Neumann Lane-Emden problems, complement the known results for the radial solutions of the  same problems in  higher dimension $N\geq 3$, as $p$ approaches the critical Sobolev exponent $\frac{N+2}{N-2}$ from the left (see \cite{AtkinsonPeletier, DeMarchisIanniPacellaAdvMath} and in particular the recent paper \cite{GrossiSaldanaTavares} where explicit rates of blow-up and sharp constants are obtained, similarly to what we obtain in our results).
\\
\\
In this work we also consider the case $\alpha>0$, namely the \emph{H\'enon problems}.
\\\\
The first results on the asymptotic behavior of solutions to the Dirichlet H\'enon problem are due to Cao e Peng (\cite{CaoPeng}) who study, in higher dimension $N\geq 3$ and as $p\rightarrow \frac{N+2}{N-2}$ from below, the ground states solutions (which are positive) showing a pointwise blow-up at a point approaching the boundary. As observed in \cite{CaoPeng} when the domain is a ball of $\mathbb R^N$, $N\geq 3$, this implies in particular that, differently with the case $\alpha=0$,  the least energy Dirichlet solution for the H\'enon problem for any fixed  $\alpha>0$ is not radial for almost critical values  $p$. This was already known from \cite[Theorem 6.1]{SmetsSuWillem}, where it was also proved that, in any dimension $N\geq 2$, the breaking of symmetry also happens at any subcritical value of the exponent $p$, as soon as $\alpha$ is large enough. \\\\
When the domain is the ball $B$,  radial  solutions for the Dirichlet H\'enon problem (i.e. \eqref{equationHenon}-\eqref{DirichletBC} with $\alpha>0$) exist for any $p\in (1, p_{\alpha})$, where $p_{\alpha}=\frac{N+2+2\alpha}{N-2}$ if $N\geq 3$, $p_{\alpha}=+\infty$ if $N=2$ (see \cite{NiHenon}).\\
A partial study of their asymptotic behavior, as $p\rightarrow p_{\alpha}$, may be found in  \cite{AmadoriGladiali_Henon_N>=3} in dimension $N\geq 3$, and in \cite{AmadoriGladiali_Henon_2} in dimension $N=2$, where this analysis is used as the starting point to compute the Morse index of the solutions as in \cite{DeMarchisIanniPacellaMathAnn, DeMarchisIanniPacellaAdvMath}. In particular, \cite{AmadoriGladiali_Henon_2} contains the sharp description,  in dimension $N=2$, of the asymptotic behavior, as $p\rightarrow +\infty$, of the nodal radii, critical points and values of the \emph{least energy radial}  solution (which is positive) and of the \emph{least energy sign-changing radial} solution (which has $2$ nodal regions).  Also, the behavior of suitable rescalings of these  solutions are studied.
\\
In this paper, we derive the existence of a unique (up to a sign) radial sign-changing solution  with  $m$ nodal regions for any $m\geq 2$, $p>1$ and for $N=2$, for the \emph{H\'enon problems} with both \emph{Dirichlet and Neumann boundary conditions} and we perform a sharp asymptotic analysis of solutions as $p\rightarrow +\infty$ (see Theorems \ref{theorem:mainDirichletG}, \ref{theorem:analisiAsintoticaRiscalateAlpha} and \ref{theoremNeumannMain}).
\\ 
In the Dirichlet case, our results complement  the ones about the asymptotic behavior of the solutions contained in \cite{AmadoriGladiali_Henon_2, AmadoriGladiali_Henon_N>=3}. In particular, we extend the results in \cite{AmadoriGladiali_Henon_2}  to \emph{any radial solution} and improve them, showing among other things that, similarly  as for the Lane-Emden case, the solutions concentrate at the origin. Following \cite{AmadoriGladiali_Henon_2, AmadoriGladiali_Henon_N>=3} and using our asymptotic analysis, one could then compute the Morse index for all the radial solutions also in dimension $N=2$, this will be the object of future investigation (see the Appendix for some discussion and a conjecture).
\\\\
The case of the \emph{Neumann H\'enon problem} is particularly interesting, since no results at all are available in the literature for this equation. The only papers considering Neumann H\'enon type problems (where a linear term is added into the equation) are \cite{GazziniSerra}, where the existence of ground state solutions and the breaking of symmetry phenomenon is investigated, and \cite{ByeonWang2018} where the asymptotic behavior as $\alpha\rightarrow +\infty$ and for fixed $p$ is studied for this solution. Our case is different not only because we are interested in radial solutions but above all because while  \cite{GazziniSerra, ByeonWang2018} deal with positive solutions, it is not difficult to see that any nontrivial solution of \eqref{equationHenon}-\eqref{NeumannBC} is necessarily sign-changing (just apply the divergence theorem similarly as in \eqref{Neumansign}).
\\ 
\\
As a consequence of our results, we also obtain information on the asymptotic behavior of radial solutions for the equation in the whole plane \eqref{problemHenonWhole}, for any fixed $\alpha\geq 0$ (see Theorem \ref{theoremMainWhole}).
\\
\\
Furthermore, from our sharp asymptotic analysis, we deduce the \emph{lack of uniform (in $p$) a priori bounds for nodal solutions of the $2$-dimensional Lane-Emden/H\'enon equation} with either Dirichlet or Neumann boundary conditions (see Theorem \ref{theorem:NObounds}), thus showing, in the Dirichlet Lane-Emden case, a difference with respect to the case of positive solutions, for which the uniform a priori bound \eqref{aPrioriBoundPositive} holds, as  recently proved in \cite{KamburovSirakov}.

\

We point out that  we are interested in the asymptotic analysis as $p\rightarrow +\infty$ and for  $\alpha\geq 0$ fixed, we refer to 
\cite{ByeonWangI, ByeonWangII, ByeonWang2018} for different results concerning the description of the  asymptotic behavior as $\alpha\rightarrow +\infty$ and $p$ is fixed (\cite{ByeonWangI, ByeonWangII}  for   Dirichlet ground states positive solutions  and \cite{ByeonWang2018}  for  Neumann ground states positive solutions of an H\'enon-type problem).

\section{Main results and ideas}
For a radial function $u :B\rightarrow\mathbb R$  defined on the unit ball $B$ we freely vary between the notations $u(x), x\in B$
and $u(r)$, $0\leq r \leq 1$.
\\
\\
In order to state our results, first we need to introduce the following definition:
\begin{definition} \label{definition:Constants}
	Let $(\theta_{k})_{k\geq 0}$ be the  sequence of real numbers uniquely defined by the following iteration:
	\begin{equation}\label{iterateAlphaM}
	\left\{
	\begin{array}{lr}
	\theta_0=2,\\
	&\\
	\theta_{k}=\displaystyle\frac{2}{\mathcal L\left[\displaystyle \frac{2}{ \ 2+\theta_{k-1}}e^{ -\frac{2}{2+\theta_{k-1}}}\right]}+2\ (>2), &\qquad \mbox{for }k\geq 1,
	\end{array}
	\right.
	\end{equation}
	where $\mathcal L$ is the Lambert function, namely the inverse function of $f(L)=Le^L$.
\end{definition}

The numbers $\theta_k$ are the building blocks to express the sharp constants involved in the asymptotic analysis of the solutions as $p\to\infty$. Below we define these constants (Definition \ref{definition:Constants2}) and explain their relationship with the concentration rates (Theorems \ref{theorem:mainDirichletG}, \ref{theoremNeumannMain}, and \ref{theoremMainWhole}). Once these relationships have been established, a careful study of the behavior of these constants (Section \ref{Constants:sec}) is used to show the lack of a  universal uniform bound on nodal solutions for the Dirichlet and Neumann problems (see Theorem \ref{theorem:NObounds}).
\begin{definition}\label{definition:Constants2}
	Let
	\begin{eqnarray}
	&& \label{Mk}	\apice{M}{k}_{k-1}:= e^{{2}/(2+\theta_{k-1})} \qquad\qquad \forall k\geq 1,\\
	&& \label{Rk} 
	\apice{R}{k}_{k-1}:=\frac{\apice{M}{k-1}_{\!\!k-2}\ (\theta_{k-2}+2)}{\apice{M}{k}_{k-1}(\theta_{k-1}-2)} \qquad\qquad \forall k\geq 2,
	\\ \label{Dk}
	&& \apice{D}{k}_k 
	:=\apice{M}{k}_{k-1} (\theta_{k-1}+2)\qquad\qquad \forall k\geq 1,
	\\\label{Sk}
	&&\apice{S}{k}_{k-1}:=\left\{\begin{array}{lr}
	0\qquad& k=1,\\
	(\apice{M}{k}_{k-1})^{-1}  \qquad&  k\geq 2,
	\end{array}
	\right.
	\end{eqnarray}
	\begin{equation}\label{SoloUltimiReD}
	\apice{R}{m}_{i}:=
	\prod_{k=i+1}^m\!\!\apice{R}{k}_{k-1},
	\qquad 
	\apice{D}{m}_{i}:=
	\frac{\apice{D}{i}_{i}}{ \displaystyle \prod_{ k=i+1}^m\!\!\apice{R}{k}_{k-1}}, \qquad i=1,\ldots, m-1,
	\end{equation}
 
	\begin{equation}
	\label{SoloUltimiSeM}
	\apice{S}{m}_{i}:=\apice{S}{i+1}_{\!\!i}\prod_{k=i+2}^{m}\!\!\apice{R}{k}_{k-1},
	\qquad
	\apice{M}{m}_{i}:=\frac{\apice{M}{i+1}_{\! i}}{\displaystyle\prod_{k=i+2}^{m}\!\!\apice{R}{k}_{k-1}},\qquad i=0,\ldots, m-2.
	\end{equation}
\end{definition}

We have the following monotonicity properties. 
\begin{lemma}\label{lemma:PropertiesConstants}
It holds that \begin{equation}\label{catenateointro}0=\apice{S}{m}_{0}<\ \apice{R}{m}_{1}<\apice{S}{m}_{1}< \apice{R}{m}_{2}<\apice{S}{m}_{2}<\cdots <\apice{S}{m}_{m-2}<\apice{R}{m}_{m-1}<\apice{S}{m}_{m-1}\, < 1,\end{equation} 
\begin{equation}\label{catenaMintro}\apice{M}{m}_0> \apice{M}{m}_1> \ldots> \apice{M}{m}_{m-1}>1.\end{equation}
The sequences $(\apice{M}{m}_j)_m$, $(\apice{D}{m}_j)_m$ are strictly increasing, while 	the sequences $(\apice{S}{m}_j)_m$, $(\apice{R}{m}_j)_m$  are strictly decreasing.
\end{lemma}

Next we state our asymptotic results for the Dirichlet problems \eqref{equationHenon}-\eqref{DirichletBC}.  Recall the definition of $\apice{M}{m}_{m-1}$, $\theta_{m-1}$, $\apice{R}{m}_i,$ $\apice{S}{m}_i,$ $\apice{M}{m}_{i},$ $\apice{D}{m}_i>0$ given in Definitions \ref{definition:Constants}-\ref{definition:Constants2} and let 
\begin{align*}
G(x,0)=-\frac{1}{2\pi}\log|x|,\qquad x\in B,
\end{align*}
denote the Green's function of $-\Delta$ in $B$ with Dirichlet boundary condition, computed at the origin.

\begin{theorem}[Dirichlet problem]\label{theorem:mainDirichletG}
Let $\alpha\geq 0$, $m\in\mathbb N$, $m\geq 1$ and $p>1$. Then there exists a unique (up to a sign) radial solution of  \eqref{equationHenon}-\eqref{DirichletBC} with $m-1$ interior zeros. This solution does not vanish in the origin, between any two consecutive zeros it has exactly one critical point, which is either a minimum or a maximum.  Let us denote by $\apice{u}{m}_{\alpha,p}$  the radial solution of \eqref{equationHenon}-\eqref{DirichletBC} with $m-1$ interior zeros and such that  $\apice{u}{m}_{\alpha,p}(0)>0$. Then $\|\apice{u}{m}_{\alpha,p}\|_{L^{\infty}}=|\apice{u}{m}_{\alpha,p}(0)|$ and 
\begin{equation}\label{pointwiseDirichletG}
p\apice{u}{m}_{\alpha,p}(x)=2\pi\gamma_{\alpha,m} G(x,0) +o(1), \mbox{ in } C^1_{loc }(B\setminus\{0\})\qquad \text{as $p\rightarrow +\infty$,}\end{equation}
where \begin{equation}
\label{gammaalm}
 \gamma_{\alpha,m}:=(-1)^{m-1}\frac{\alpha+2}{2}\apice{M}{m}_{m-1} (\theta_{m-1}+2).
 \end{equation}
Moreover,
\begin{eqnarray}
\label{EnergiaTotaleG}
p\int_0^1|(\apice{u}{m}_{\alpha,p})'(r)|^2rdr&=&
p\int_0^1|\apice{u}{m}_{\alpha,p}(r)|^{p+1}r^{\alpha+1}dr
=\frac{\alpha+2}{8}(\apice{M}{m}_{m-1})^2 (\theta_{m-1}+2)^2 +o(1)\qquad 
\end{eqnarray}
as $p\rightarrow +\infty$. If $\apice{r}{m}_{i,\alpha,p}$ and  $\apice{s}{m}_{i,\alpha,p}$ denote the  zeros  and the critical points of $\apice{u}{m}_{\alpha,p}$ respectively, so that
\[ 
0=\apice{s}{m}_{0,\alpha,p}<\apice{r}{m}_{1,\alpha,p}<\apice{s}{m}_{1,\alpha,p}<\apice{r}{m}_{2,\alpha,p}<\ldots <\apice{r}{m}_{m-1,\alpha,p}<\apice{s}{m}_{m-1,\alpha,p}<\apice{r}{m}_{m,\alpha,p}=1,
\] then, as $p\rightarrow +\infty$,
\begin{eqnarray}
&(\apice{r}{m}_{i,\alpha,p})^{\frac{2}{p-1}}=\left(\apice{R}{m}_{i}\right)^{\frac{2}{\alpha+2}}+o(1), & \ i=1,\ldots, m-1,
\label{raggioG}
\\
&p|(\apice{u}{m}_{\alpha,p})'(\apice{r}{m}_{i,\alpha,p})|(\apice{r}{m}_{i,\alpha,p})=\frac{\alpha+2}{2}\apice{D}{m}_i +o(1), &\  i=1,\ldots, m,
\label{derivataInRaggioG}
\\
&(\apice{s}{m}_{i,\alpha,p})^{\frac{2}{p-1}}=\left(\apice{S}{m}_i\right)^{\frac{2}{\alpha+2}}+o(1),& \ i=1,\ldots, m-1,
\label{puntoCriticoG}
\\
&|\apice{u}{m}_{\alpha,p}(\apice{s}{m}_{i,\alpha,p})|=\apice{M}{m}_{i}+o(1),&\  i=0,\ldots, m-1.
\label{valoreInPuntoCriticoG}
\end{eqnarray}
\end{theorem}
Observe that, from  \eqref{pointwiseDirichletG}, it follows that the solution $\apice{u}{m}_{\alpha,p}$ \emph{concentrates at the origin as $p\rightarrow +\infty$} and that $\apice{u}{m}_{\alpha,p}\rightarrow 0$ uniformly on compact subsets of $B\setminus\{0\}$. Moreover from \eqref{valoreInPuntoCriticoG} we also know that $\apice{u}{m}_{\alpha,p}(0)\rightarrow \apice{M}{m}_0$ ($>1$ from \eqref{catenaMintro}). Notice also that 
Lemma \ref{lemma:PropertiesConstants} and  \eqref{raggioG}, \eqref{puntoCriticoG},\eqref{derivataInRaggioG} and \eqref{catenateointro} imply respectively that, as $p\to\infty$,
\[\begin{array}{ll}
\apice{r}{m}_{i,\alpha,p}\sim (\apice{R}{m}_i)^{\frac{p-1}{\alpha+2}}\rightarrow 0, \qquad i=1,\ldots, m-1, \\
\apice{s}{m}_{i,\alpha,p}\sim (\apice{S}{m}_i)^{\frac{p-1}{\alpha+2}}\rightarrow 0, \qquad i=1,\ldots, m-1,
\\
|(\apice{u}{m}_{\alpha,p})'(\apice{r}{m}_{i,\alpha,p})|\sim\left\{\begin{array}{lr}
\frac{\frac{\alpha+2}{2}\apice{D}{m}_i}{p(\apice{R}{m}_i)^{\frac{p-1}{\alpha+2}}}\rightarrow +\infty, \quad i=1,\ldots, m-1,\\\\
\frac{\frac{\alpha+2}{2}\apice{D}{m}_m}{p}\rightarrow 0, \qquad \ \ \ i=m.
\end{array}\right. 
\end{array}
\]

\

Next we investigate the asymptotic behavior as $p\rightarrow +\infty$ of suitable rescalings of the solution $\apice{u}{m}_{\alpha,p}$ in each nodal region, showing a \emph{tower of bubbles} phenomenon.

To this aim, let us define the $m$ 
parameters
\begin{equation}\label{scpar}\apice{\varepsilon}{m}_{i,\alpha,p}:= \left(\frac{\alpha+2}{2}\right)^{\frac{2}{\alpha+2}}\left[p |\apice{u}{m}_{\alpha,p}(\apice{s}{m}_{i,\alpha,p})|^{p-1}\right]^{-\frac{1}{2+\alpha}},\quad i=0,\ldots, m-1,
\end{equation}
and the $m$ rescaled functions 
	\begin{equation}\label{scsol}\apice{\xi}{m}_{i,\alpha,p}(r):=
p	\frac{(-1)^i\  \apice{u}{m}_{\alpha,p}( \apice{\varepsilon}{m}_{i,\alpha,p} r)  -    |\apice{u}{m}_{\alpha,p}(\apice{s}{m}_{i,\alpha,p})| }{|\apice{u}{m}_{\alpha,p}(\apice{s}{m}_{i,\alpha,p})|}, \quad i=0,\ldots, m-1,
\end{equation}
for 
\[r\in\left\{\begin{array}{lr}
[0,\frac{\apice{r}{m}_{1,\alpha,p}}{\apice{\varepsilon}{m}_{0,\alpha,p}}], \qquad \quad \qquad \ \mbox{ if }i=0,
\\
\left[\frac{\apice{r}{m}_{i,\alpha,p}}{\apice{\varepsilon}{m}_{i,\alpha,p}},\frac{\apice{r}{m}_{i+1,\alpha,p}}{\apice{\varepsilon}{m}_{i,\alpha,p}}
\right],\qquad \mbox{ if }i\geq 1.
\end{array}
\right.\]

Let
		\begin{equation}\label{bubblecomplete}
		Z_{i,\alpha}(r)=Z_{i,\alpha}(|x|):=\log\frac{2\theta_i^2\beta_i^{\theta_i}|x|^{\frac{(\alpha+2)}{2}(\theta_i-2)}}{(\beta_i^{\theta_i}+|x|^{\frac{(\alpha+2)}{2}\theta_i})^2},
		\end{equation}
		with $\beta_0:=2\sqrt{2} $, $\beta_i:=	\frac{1}{ \sqrt2}(\theta_i+2)^{\frac{\theta_i+2}{2\theta_i}}(\theta_i-2)^{\frac{\theta_i-2}{2\theta_i}} $ for $ i=1,\ldots,m-1$, then $Z_{i,\alpha}$ is a radial solution of
		\begin{equation}
		\label{LiouvilleSingularEquationAlpha}
		\left\{
		\begin{array}{lr}
		-\Delta Z=\left(\frac{\alpha +2 }{2}\right)^2|x|^{\alpha}e^Z+ (\alpha+2)\pi(2-\theta_i)\delta_{0}\quad\mbox{ in }\R^2,\\
		Z(\sigma_{i,\alpha})=0,\\
		\int_{\R^2}e^Z|x|^{\alpha}dx=\frac{8\pi\theta_i}{\alpha+2},
		\end{array}
		\right.
		\end{equation}
		where  $\delta_{0}$ is the Dirac measure centered at $0$ and $\theta_i \ (\geq 2)$ is given in \eqref{iterateAlphaM}. The limit profiles $Z_{i,\alpha}$'s are usually named \emph{bubbles}. 
\begin{theorem}[Tower of bubbles for the Dirichlet problem]
		\label{theorem:analisiAsintoticaRiscalateAlpha}
		Let $\alpha\geq 0$,  $m\in\mathbb N$, $m\geq 1$, then, as  $p\rightarrow +\infty$,  $\apice{\varepsilon}{m}_{i,\alpha,p}=o(1)$  and
	\begin{equation}
	\label{relaziParametri}
	\frac{\apice{r}{m}_{i,\alpha,p}}{\apice{\varepsilon}{m}_{i,\alpha,p}}=o(1)\ (i\neq 0),
	 \qquad\quad \frac{\apice{s}{m}_{i,\alpha,p}}{\apice{\varepsilon}{m}_{i,\alpha,p}}= \sigma_{i,\alpha}+ o(1), \qquad\quad \frac{\apice{\varepsilon}{m}_{i,\alpha,p}}{\apice{r}{m}_{i+1,\alpha,p}}=o(1),
	\end{equation}
for all $i=0,\ldots, m-1$ with $\sigma_{i,\alpha}:=\left(\frac{\theta_i^2-4}{2}\right)^{\frac{1}{2+\alpha}}$. Moreover,
		\begin{equation}
\label{convBubblesG}
	 \apice{\xi}{\emph m}_{i,\alpha,p} = Z_{i,\alpha}+o(1)\quad\mbox{in $C^1_{loc}(0, +\infty)  \qquad  \text{ for all }\, i=0,\ldots, m-1$ as $p\rightarrow +\infty$.}
		\end{equation}
	\end{theorem}	
Observe that $\frac{\apice{r}{m}_{i,\alpha,p}}{\apice{s}{m}_{i,\alpha,p}}=o(1)$ and  $ \frac{\apice{s}{m}_{i,\alpha,p}}{\apice{r}{m}_{i+1,\alpha,p}}=o(1)$
so that 
\[\frac{\apice{\varepsilon}{m}_{i,\alpha,p}}{\apice{\varepsilon}{m}_{i+1,\alpha,p}}=o(1)\quad \mbox{ as }p\rightarrow +\infty,\]
namely, in  each nodal region the \emph{bubble} appears at a \emph{different scale},  for this reason we say that the solution has a \emph{tower of bubbles} behavior.   
\\
Recall  that $\theta_0=2$, hence $\sigma_0=0$, as a consequence the \emph{first bubble $Z_{0,\alpha}$} is a regular solution of the Liouville equation	
	$-\Delta Z=\left(\frac{\alpha +2 }{2}\right)^2|x|^{\alpha}e^Z$ in $\R^2$ with $Z(0)=0$ (see Figure \ref{f1}).
	
	\
 
Using a suitable change of variables, we can obtain from Theorem \ref{theorem:mainDirichletG} a description of the radial solutions of the Neumann problem \eqref{equationHenon}-\eqref{NeumannBC}.

\begin{theorem}[Neumann problem]
\label{theoremNeumannMain} 
Let $\alpha\geq 0$,  $m\in\mathbb N$, $m\geq 2$ and $p>1$.
Then there exists a unique (up to a sign) radial solution of  \eqref{equationHenon}-\eqref{NeumannBC} with $m-1$ interior zeros. This solution does not vanish in the origin and between any two consecutive zeros it has exactly one critical point, which is either a minimum or a maximum.\\
Let us denote by $\apice{\bar u}{m}_{\alpha,p}$  the radial solution of \eqref{equationHenon}-\eqref{NeumannBC} with $m-1$ interior zeros and such that  $\apice{\bar u}{m}_{\alpha,p}(0)>0$.
Then 
\begin{equation}\label{changeuubar}
\apice{\bar u}{m}_{\alpha,p}(r)=(\apice{s}{m}_{m-1,\alpha,p})^{\frac{\alpha+2}{p-1}}\apice{u}{m}_{\alpha,p}(\apice{s}{m}_{m-1,\alpha,p} r), \qquad r \in [0,1],\end{equation}
where $\apice{ u}{m}_{\alpha,p}$ and $\apice{s}{m}_{m-1,\alpha,p}$ are as in Theorem \ref{theorem:mainDirichletG}.

Moreover,
\begin{equation}
\label{energiaNeumann}
p\int_0^1|(\apice{\bar u}{m}_{\alpha,p})'(r)|^2rdr=
p\int_0^1|\apice{\bar u}{m}_{\alpha,p}(r)|^{p+1}r^{1+\alpha}dr=\frac{\alpha+2}{8} (\theta_{m-1}+2)(\theta_{m-1}-2) +o(1)
\end{equation}
as $p\rightarrow +\infty$ and, letting $\apice{\bar r}{m}_{i,\alpha,p}$ and $\apice{\bar s}{m}_{i,\alpha,p}$ be the zeros and the  critical points of $\apice{\bar u}{m}_{\alpha,p}$ respectively, so that
 \[ 
 0=\apice{\bar s}{m}_{0,\alpha,p}<\apice{\bar r}{m}_{1,\alpha,p}<\apice{\bar s}{m}_{1,\alpha,p}<\apice{\bar r}{m}_{2,\alpha,p}<\ldots <\apice{\bar r}{m}_{m-1,\alpha,p}<\apice{\bar s}{m}_{m-1,\alpha,p}=1,
 \]
then 
	\begin{eqnarray}\label{raggioGN}
	&(\apice{\bar r}{m}_{i,\alpha,p})^{\frac{2}{p-1}}=\left(\apice{\bar R}{m}_{i}\right)^{\frac{2}{\alpha+2}} +o(1),\quad &  i=1,\ldots, m-1,
	\\ \label{derivataInRaggioGN}
	&p|(\apice{\bar u}{m}_{\alpha,p})'(\apice{\bar r}{m}_{i,\alpha,p})|(\apice{\bar r}{m}_{i,\alpha,p})=\frac{\alpha+2}{2}\apice{\bar D}{m}_i +o(1), \quad &  i=1,\ldots, m-1,
	\\\label{puntoCriticoGN}
	&(\apice{\bar s}{m}_{i,\alpha,p})^{\frac{2}{p-1}}=\left(\apice{\bar S}{m}_i\right)^{\frac{2}{\alpha+2}} +o(1), \quad& i=1,\ldots, m-2,
	\\\label{valoreInPuntoCriticoGN}
	&|\apice{\bar u}{m}_{\alpha,p}(\apice{\bar s}{m}_{i,\alpha,p})|=\apice{\bar M}{m}_{i} +o(1),\quad & i=0,\ldots, m-1,
	\end{eqnarray}
as $p\rightarrow +\infty$, where
\begin{equation}\label{raggioNeumann}
\apice{\bar R}{m}_{i}	:= \frac{\apice{R}{m}_{i} }{\apice{S}{m}_{m-1}},
\qquad
\apice{\bar D}{m}_i: =	\apice{S}{m}_{m-1} \apice{ D}{m}_i,
\qquad	
\apice{\bar S}{m}_i:=\frac{\apice{S}{m}_{i} }{\apice{S}{m}_{m-1}},
\qquad\apice{\bar M}{m}_{i}	:= 	\apice{S}{m}_{m-1} \apice{ M}{m}_i,
\end{equation}
and  $\apice{R}{m}_i$, $\apice{S}{m}_i$, $\apice{M}{m}_{i}$ and $\apice{D}{m}_i$ are the constants  in Definitions \ref{definition:Constants}-\ref{definition:Constants2}.
\end{theorem}

\

From Theorem \ref{theoremNeumannMain} and Lemma \ref{lemma:PropertiesConstants} we deduce that 
\[
\begin{array}{lr}
\apice{\bar r}{m}_{i,\alpha,p}\sim (\apice{\bar R}{m}_i)^{\frac{p-1}{\alpha+2}}\rightarrow 0, &\qquad i=1,\ldots, m-1,
\\
\apice{\bar s}{m}_{i,\alpha,p}\sim (\apice{\bar S}{m}_i)^{\frac{p-1}{\alpha+2}}\rightarrow 0, &\qquad i=1,\ldots, m-2, 
\\
|(\apice{\bar u}{m}_{\alpha,p})'(\apice{\bar r}{m}_{i,\alpha,p})|\sim
\frac{\frac{\alpha+2}{2}\apice{\bar D}{m}_i}{p(\apice{\bar R}{m}_i)^{\frac{p-1}{\alpha+2}}}\rightarrow +\infty,  & \qquad i=1,\ldots, m-1,
\\
|\apice{\bar u}{m}_{\alpha,p}(\apice{\bar s}{m}_{i,\alpha,p})|\sim\apice{\bar M}{m}_{i}>1, & \qquad i=0,\ldots, m-2,
\end{array}
\]
and that $|\apice{\bar u}{m}_{\alpha,p}(1)|\rightarrow 1$  as $p\rightarrow +\infty$. We expect  that  $\apice{\bar u}{m}_{\alpha,p}(x)\rightarrow 1$ uniformly on compact subsets of $B\setminus\{0\}$, but this does not follow directly from Theorem \ref{theorem:mainDirichletG} and would require further analysis.

\

Next we consider the radial solutions to \eqref{problemHenonWhole}. They oscillate infinitely many times and have a unique local maximum or minimum  between any two consecutive zeros. It is not difficult to see that they are linked by a suitable change of variables to the radial solutions of the Dirichlet and Neumann problems, hence as a consequence of our previous results we deduce the following characterization.
\begin{theorem}[Problem in  $\mathbb R^2$]\label{theoremMainWhole} Let $\alpha\geq 0$ and $w_{\alpha,p}$ be the radial solution of \eqref{problemHenonWhole} such that \[w_{\alpha,p}(0)=1.\] Then $w_{\alpha,p}$ has a sequence $(\rho_{m,\alpha,p})_{m\in\mathbb N}$ of zeros and a sequence 
$(\delta_{m,\alpha,p})_{m\in\mathbb N}$ of critical points such that
\[0=\delta_{0,\alpha,p}<\rho_{1,\alpha,p}<\delta_{1,\alpha,p}<\rho_{2,\alpha,p}<\ldots <\delta_{m-1,\alpha,p}<\rho_{m,\alpha,p}<\delta_{m,\alpha,p}<\ldots;\]
and
\begin{equation}\label{changeuw}
\apice{u}{m}_{\alpha,p}(r)=(\rho_{m,\alpha,p})^{\frac{\alpha+2}{p-1}}w_{\alpha,p}(\rho_{m,\alpha,p}r), \qquad r\in[0,1],\end{equation}
where $\apice{u}{m}_{\alpha,p}$ is as in Theorem \ref{theorem:mainDirichletG}.
Moreover,
\[\begin{array}{lr}
	(\rho_{m,\alpha,p})^{\frac{2}{p-1}}=
	\left(\apice{M}{m}_{0}\right)^{\frac{2}{\alpha +2}}+o(1),
	\\
	p|(w_{\alpha,p})'(\rho_{m,\alpha,p})|(\rho_{m,\alpha,p})=\frac{\alpha +2}{2}\frac{\apice{D}{m}_{m}}{\apice{M}{m}_0} +o(1),
	\\
(\delta_{m,\alpha,p})^{\frac{2}{p-1}}=\left(\apice{M}{m+1}_{\!\!0}\apice{S}{m+1}_{\!\!\!m}\right)^{\frac{2}{\alpha +2}} +o(1),
	\\
	|w_{\alpha,p}(\delta_{m,\alpha,p})|=\frac{\apice{M}{m+1}_{\!\!\!m}}{\apice{M}{m+1}_{\!\!0}}+o(1),
\end{array}\]
as $p\rightarrow +\infty$, where $\apice{M}{m}_0$,  $\apice{M}{m+1}_{\!\!0}$, $\apice{M}{m+1}_{\!\!\! m}$, $\apice{D}{m}_{m}$, $\apice{S}{m+1}_{\!\!\!m}$ are the constants in Definitions \ref{definition:Constants}-\ref{definition:Constants2}.
\end{theorem}

From the study of the constants performed in Section \ref{Constants:sec} (see Theorem \ref{a:t}), we deduce the following result which implies that there is no a priori bound for nodal solutions.
\begin{theorem}
	\label{theorem:NObounds}
Let $\apice{u}{m}_{\alpha,p}$ and $\apice{\bar u}{m}_{\alpha,p}$ be as in Theorem \ref{theorem:mainDirichletG} and Theorem \ref{theoremNeumannMain} respectively and let $c_1:=\sqrt{\pi}\approx 1.77$ and $c_2:= 6\frac{\Gamma(\frac{3}{4})}{\Gamma(\frac{1}{4})}\approx 2.02$. Then, for every $m\in \N$,
\begin{align}
c_1\ \frac{\Gamma(m)}{\Gamma(m-\frac{1}{2})}e^{\frac{1}{4m-1}} &< \lim_{p\to\infty}\|\apice{u}{m}_{\alpha,p}\|_{L^{\infty}}< c_2 \ \frac{\Gamma(m+\frac{1}{4})}{\Gamma(m-\frac{1}{4})}e^{\frac{1}{4m-2}},\label{D:bd}\\
c_1\ \frac{\Gamma(m)}{\Gamma(m-\frac{1}{2})}e^{\frac{1}{4m-1}-\frac{1}{4m-2}} &< \lim_{p\to\infty}\|\apice{\bar u}{m}_{\alpha,p}\|_{L^{\infty}}< c_2 \ \frac{\Gamma(m+\frac{1}{4})}{\Gamma(m-\frac{1}{4})}e^{\frac{1}{4m-2}-\frac{1}{4m-1}}\label{N:bd}.
\end{align}
Furthermore, 
\begin{align}
c_1\leq \liminf_{m\to\infty}\lim_{p\to\infty}\frac{\|\apice{u}{m}_{\alpha,p}\|_{L^{\infty}}}{\sqrt{m}}&\leq \limsup_{m\to\infty}\lim_{p\to\infty}\frac{\|\apice{u}{m}_{\alpha,p}\|_{L^{\infty}}}{\sqrt{m}}\leq c_2,\label{D:asym}\\
c_1\leq \liminf_{m\to\infty}\lim_{p\to\infty}\frac{\|\apice{\bar u}{m}_{\alpha,p}\|_{L^{\infty}}}{\sqrt{m}}&\leq \limsup_{m\to\infty}\lim_{p\to\infty}\frac{\|\apice{\bar u}{m}_{\alpha,p}\|_{L^{\infty}}}{\sqrt{m}}\leq c_2.\label{N:asym}
\end{align}
\end{theorem}

An immediate consequence is that the $i$-th local maximum or minimum is also unbounded for any $i\in\N$ as $m,p\to\infty$ and with the same growth rate.
\begin{corollary}\label{coro} Using the notation of  Theorems \ref{theorem:mainDirichletG} and \ref{theoremNeumannMain}. Then
\begin{align*}
0<\liminf_{m\to\infty}\lim_{p\to\infty}\frac{\apice{u}{m}_{\alpha,p}(\apice{s}{m}_{i,\alpha,p})}{\sqrt{m}}&\leq \limsup_{m\to\infty}\lim_{p\to\infty}\frac{\apice{u}{m}_{\alpha,p}(\apice{s}{m}_{i,\alpha,p})}{\sqrt{m}}<\infty\qquad \text{ for every }i\in\N,
\end{align*}
and the same result holds with $\apice{\bar u}{m}_{\alpha,p}(\apice{\bar s}{m}_{i,\alpha,p})$ instead of $\apice{u}{m}_{\alpha,p}(\apice{s}{m}_{i,\alpha,p})$.
\end{corollary}

Theorem \ref{theorem:NObounds} implies that, for the $2$-dimensional Lane-Emden and  H\'enon problems with either Dirichlet or Neumann boundary conditions,  \emph{a priori uniform bounds do not hold true for  nodal solutions in general}, and in fact the supremum norm grows as $\sqrt{m}$, where $m-1$ is the number of interior zeros of the solution (see Figure \ref{f3}). We recall that positive solutions of Dirichlet Lane-Emden problems do satisfy an a priori uniform bound, see \cite{KamburovSirakov}.

\

To close this section, we describe the organization and main ideas of the paper. Theorems \ref{theorem:mainDirichletG} and \ref{theorem:analisiAsintoticaRiscalateAlpha} (Dirichlet Lane-Emden) are shown in Section \ref{section:Lane} for $\alpha=0$. Here, an inductive strategy is used together with the fact that $\apice{u}{m-1}_{\!\!\!0,p}$ and $\apice{u}{m}_{0,p}$ are related via a suitable \emph{change of variables}. 
More details on the structure of our proof is in Section \ref{section:strategiaRIcorsiva}. In Section \ref{section:Heno} we study the case $\alpha>0$ in Theorems \ref{theorem:mainDirichletG} and \ref{theorem:analisiAsintoticaRiscalateAlpha} (Dirichlet H\'enon), which is deduced from the case $\alpha=0$ using another change of variables which \emph{works only in dimension $2$}. Using  \eqref{changeuubar} and \eqref{changeuw}, we show Theorems \ref{theoremNeumannMain} (Neumann) and \ref{theoremMainWhole} (whole $\mathbb R^2$) in Section \ref{section:Neumann&R2}. In Section \ref{Constants:sec} we do a careful analysis of the constants and use this to show Theorem \ref{theorem:NObounds}, Corollary \ref{coro}, and Lemma \ref{lemma:PropertiesConstants}.  Finally, in the Appendix \ref{App:Morse} we discuss a conjecture on a Morse index formula for the solutions.

\section{The Dirichlet Lane-Emden problem in $B$}\label{section:Lane}

This section is devoted to the proof of the case $\alpha=0$ in Theorem \ref{theorem:mainDirichletG} and in Theorem \ref{theorem:analisiAsintoticaRiscalateAlpha}. Hence, we consider the Dirichlet Lane-Emden problem
\begin{equation}\label{problem}\left\{\begin{array}{lr}-\Delta u= |u|^{p-1}u\qquad  \mbox{ in }B,\\
	u=0\qquad\qquad\qquad\mbox{ on }\partial B.
	\end{array}\right.
	\end{equation}
It is well known that, for any  $p>1$ and any $m\in\N$, $m\geq 1$, there exists a unique (up to a sign) radial solution   to \eqref{problem}  with exactly $m-1$ interior zeros  (see for instance \cite[p. 263]{K}).
\\
This solution does not vanish in the origin so let us simply denote by $\apice{u}{m}_p$ the unique nodal radial solution of \eqref{problem} having $m-1$ interior zeros and satisfying
\begin{equation}
\label{MaxInZero}
\apice{u}{m}_p(0)>0
\end{equation}
(namely $\apice{u}{m}_p=\apice{u}{m}_{p,0}$).
With a slight abuse of notation, we often  write  $\apice{u}{m}_p(r)=\apice{u}{m}_p(|x|)$.

\subsection{Preliminary properties}\label{sectin:LabePreliminaries}
Let us denote by $\apice{r}{m}_{i,p}$, $i=1,\ldots m$, the nodal radii of $\apice{u}{m}_p$, \emph{i.e.,}
\begin{equation} \label{rp}
\begin{array}{lr}
0<\apice{r}{m}_{1,p}<\apice{r}{m}_{2,p}<\cdots <\apice{r}{m}_{m-1,p}<\apice{r}{m}_{m,p}=1,\\
\apice{u}{m}_p(\apice{r}{m}_{i,p})=0\qquad i=1,\ldots,m.
\end{array}
\end{equation}

\begin{proposition}\label{PropositionUnicoMaxeMin}
	Let $p>1$, then
	\begin{itemize}
		\item[(i)] $\apice{u}{m}_p(0)=\|\apice{u}{m}_p\|_{\infty}$.
		\item[(ii)] $0$ is the unique critical point of the map $r\mapsto \apice{u}{m}_p(r)$ in the interval $[0, \apice{r}{m}_{1,p}]$; in each interval $[\apice{r}{m}_{i,p}, \apice{r}{m}_{i+1,p}]$ the map $r\mapsto \apice{u}{m}_p(r)$ has exactly one critical point $\apice{s}{m}_{i,p}$, so
		\begin{equation} \label{sp}
		\begin{array}{lr}
		\apice{s}{m}_{0,p}:=0,\\
		\apice{s}{m}_{i,p} \in (\apice{r}{m}_{i,p}, \apice{r}{m}_{i+1,p}),\quad i=1,\ldots, m-1  \ (\mbox{if }m\geq 2),
		\end{array}
		\end{equation} 
		\[(\apice{u}{m}_p)'(\apice{s}{m}_{i,p})=0, \quad i=0,\ldots, m-1,\]
		$\apice{s}{m}_{2j,p}$ is  a local maximum and $\apice{s}{m}_{2j+1,p}$ is a local minimum point for $j=0,1,2,\ldots$.
		\item[(iii)]
		\begin{equation}\label{Mmonotoni}
		|\apice{u}{m}_{p}(\apice{s}{m}_{0,p})|>|\apice{u}{m}_{p}(\apice{s}{m}_{1,p})|>\ldots >|\apice{u}{m}_{p}(\apice{s}{m}_{m-1,p})|.
	\end{equation}
	\end{itemize}
\end{proposition}
\begin{proof}
	The one variable function $r\mapsto \apice{u}{m}_{p}$  satisfies
	\begin{equation}\label{ode}
	(\apice{u}{m}_{p})''+\frac{1}{r}(\apice{u}{m}_{p})'+|\apice{u}{m}_{p}|^{p-1}
	\apice{u}{m}_{p}=0,
	\end{equation}
	(ii) then follows immediately observing that for  critical points $(\apice{u}{m}_{p})''=-|\apice{u}{m}_{p}|^{p-1}\apice{u}{m}_{p}$.	
	Moreover, multiplying the equation \eqref{ode} by $(\apice{u}{m}_{p})'$ we have that $F'(r)=-\frac{1}{r}\left[(\apice{u}{m}_p)'(r) \right]^2\leq 0$, where \[F(r)=\frac{1}{2}|(\apice{u}{m}_{p})'|^2+\frac{1}{p+1}|\apice{u}{m}_{p}|^{p+1}.\]
	Thus $F$ is non increasing, in particular $F(0)\geq F(r)$ for all $r\in (0,1)$, which implies (i), and $F(\apice{s}{m}_{i,p})\geq F(\apice{s}{m}_{i+i,p})$, which implies (iii). 
\end{proof}

\begin{center}
\setlength{\unitlength}{1cm}
\thicklines
\begin{picture}(13,8.5)
\put(13.1,2.35){$\mid$}
\put(13,1.9){$\apice{r}{m}_{m,p}$}
\put(13.1,2.7){$1$}
\put(11.65,2.35){$\mid$}
\put(11.5,1.9){$\apice{s}{m}_{m-1,p}$}
\put(10.34,2.35){$\mid$}
\put(10.24,1.9){$\apice{r}{m}_{m-1,p}$}
\put(9.1,2.35){$\mid$}
\put(8.7,1.9){$\apice{s}{m}_{m-2,p}$}
\put(7.85,2.35){$\mid$}
\put(7.4,1.9){$\apice{r}{m}_{m-2,p}$}
\put(6.7,2.35){$\mid$}
\put(5.55,2.35){$\mid$}
\put(4.55,2.35){$\mid$}
\put(3.6,2.35){$\mid$}
\put(2.7,2.35){$\mid$}
\put(2.6,1.9){$\apice{s}{m}_{2,p}$}
\put(1.94,2.35){$\mid$}
\put(1.9,1.9){$\apice{r}{m}_{2,p}$}
\put(1.35,2.35){$\mid$}
\put(1.1,1.9){$\apice{s}{m}_{1,p}$}
\put(.75,2.35){$\mid$}
\put(.35,1.9){$\apice{r}{m}_{1,p}$}
\put(.5,7.8){$\apice{u}{m}_{p}(r)$}
\put(-.7,7.8){$\apice{M}{m}_{0,p}$}
\put(-.7,4.4){$\apice{M}{m}_{2,p}$}
\put(-.9,0){$-\apice{M}{m}_{1,p}$}
\put(-1.3,1.3){$-\apice{M}{m}_{m-2,p}$}
\put(-1.1,3.4){$\apice{M}{m}_{m-1,p}$}
\put(-1.15,2.1){$\apice{s}{m}_{0,p}=0$}
\includegraphics[width=.9\textwidth]{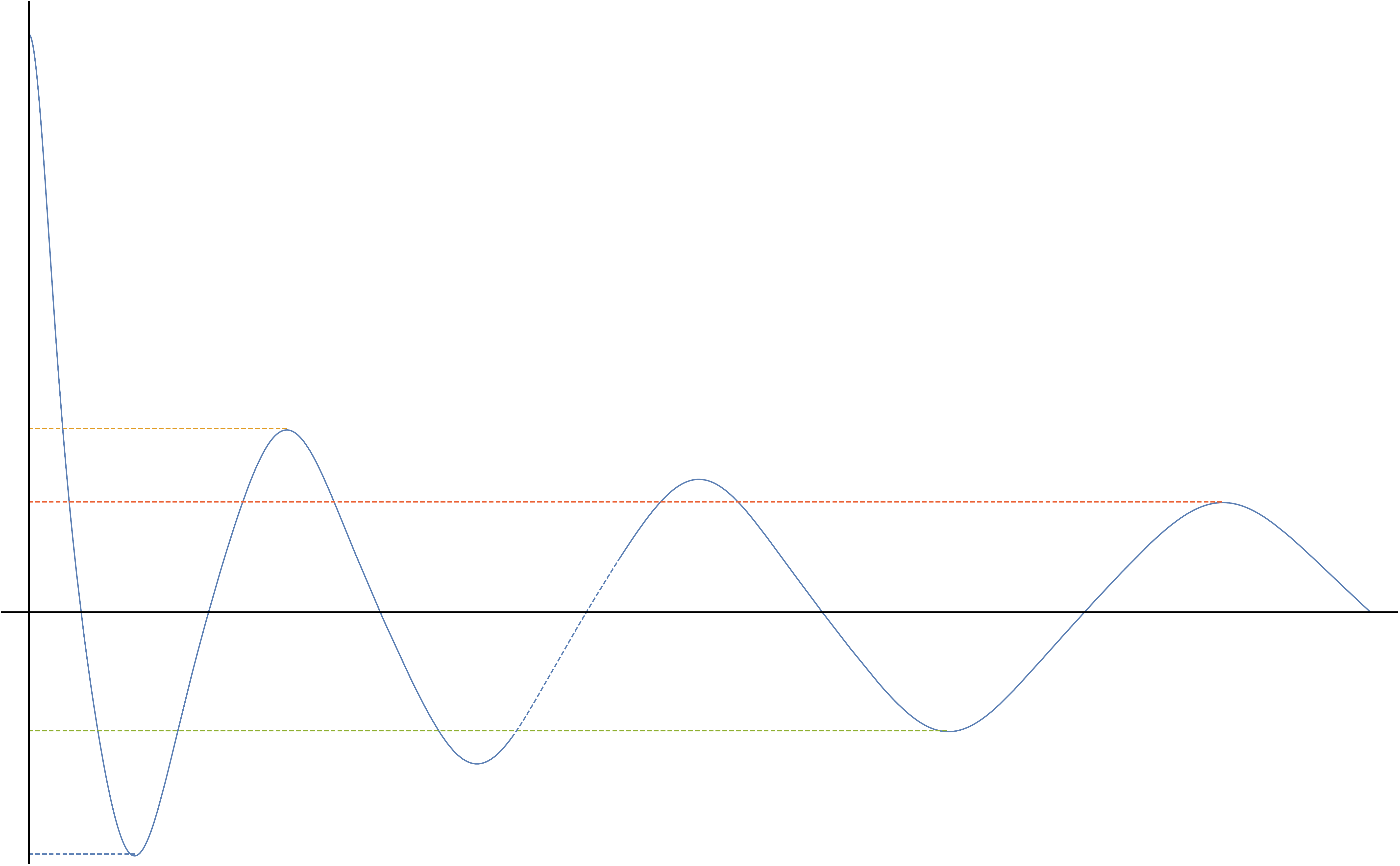}
\end{picture}
\captionof{figure}{In our notation, $\apice{s}{m}_{i,p}$ denote the critical points, $\apice{r}{m}_{i,p}$ are roots, and $\apice{M}{m}_{i,p}$ are the alternating local minima and maxima in absolute value.}
\label{im}
\end{center}

\medskip

From 	\cite[Proposition 2.1]{DicksteinPacellaSciunzi}  we have a bound of the total energy.
\begin{lemma}
	There exist $p_m>1$ and $E_m>0$ such that
	\begin{equation}\label{stimaenergiaupm}
	p\int_0^1|(\apice{u}{\emph m}_p)'(r)|^{2}rdr=p\int_0^1|\apice{u}{m}_p(r)|^{p+1}rdr\leq E_m \quad\mbox{ for }p\geq p_m.
	\end{equation}
\end{lemma}

By the classical Strauss inequality for radial functions in $H^1(\mathbb R^2)$ (\cite{Strauss}) and the energy estimate in \eqref{stimaenergiaupm} we deduce the following pointwise bound.
\begin{lemma}\label{lemma:Strauss}
	There exists $C_m>0$ such that 
	\[|\apice{u}{m}_p(r)|\leq \frac{C_m}{\sqrt{r}}\qquad\mbox{ for any }r\neq 0 \text{ and }\ p\geq p_m.
	\]
\end{lemma}

Moreover, integrating the  equation \eqref{problem}  written in polar coordinates, we have that
\[-\big((\apice{u}{m}_p)'(r)r \big)'=|\apice{u}{m}_p(r)|^{p-1}\apice{u}{m}_p(r)r\qquad \text{ in }(0,1).\]
As a consequence, we deduce the next identity.
\begin{lemma}\label{lemmaN}
Let $s,t\in(0,1)$, then
	\begin{equation}\label{N}
	(\apice{u}{m}_p)'(s)s-(\apice{u}{m}_p)'(t)t=\int_s^t|\apice{u}{m}_p(r)|^{p-1}\apice{u}{m}_p(r)r\,dr.
	\end{equation}
\end{lemma}

\

We then obtain the following estimate for the derivative.
\begin{lemma} \label{lemma:stimaDerivataInrp}
	There exists $C_m>0$  and $p_m>1$ such that
	\[
	p\big|(\apice{u}{m}_p)'(r)\big|\leq \frac{C_m}{r} \quad\  \text{ for all }\, r\in (0,1] \text{ and } p\geq p_m.
	\]
\end{lemma}
\begin{proof} Let $r\in (0,1)$. Choosing $s=0$ and $t=r$ in the identity \eqref{N} (recall that $(\apice{u}{m}_p)'(0)=0$), by H\"older's inequality,
	\[p\big|(\apice{u}{m}_p)'(r)\big|\,r\leq p\int_{0}^r|\apice{u}{m}_p(s)|^ps \, ds\leq  p\left[\int_{0}^1|\apice{u}{m}_p(s)|^{p+1}s \, ds\right]^{\frac{p}{p+1}}
	\]
	and the conclusion follows from \eqref{stimaenergiaupm}. The case $r=1$ is obtained   by continuity.
\end{proof}
Moreover, 	choosing $s=\apice{s}{m}_{m-1,p}$ and $t=1$ into \eqref{N}, we also have the following.
\begin{lemma}
	\label{lemmaC}	
	It holds that
	\begin{equation} 	\label{LemmaD}
	-p\,(\apice{u}{m}_p)'(1)=p\int_{\apice{s}{m}_{m-1,p}}^1 |\apice{u}{m}_p|^{p-1}\apice{u}{m}_pr\,dr
	=(-1)^{m-1}p\int_{\apice{s}{m}_{m-1,p}}^1 |\apice{u}{m}_p|^{p}r\,dr.
	\end{equation}	
\end{lemma}

\begin{lemma}[Pohozaev identity]\label{lemmaP}It holds that
	\[p\int_0^1 |\apice{u}{m}_p|^{p+1}r\,dr=\left(1+\frac{1}{p}\right)\frac{1}{4}\left[p(\apice{u}{m}_p)'(1) \right]^2.\]
\end{lemma}
\begin{proof}
	The claim follows from the Pohozaev identity \[\frac{2}{p+1}\int_B|\apice{u}{m}_p(x)|^{p+1}dx=\frac{1}{2}\int_{\partial B}(x\cdot \nu(x))\left(\frac{\partial \apice{u}{m}_p(x)}{\partial \nu}  \right)^2d\sigma_x,\]
	where $\nu$ denotes the exterior unit normal vector on $\partial B$.
\end{proof}

We conclude this paragraph recalling some known asymptotic results. Thanks to the energy bound \eqref{stimaenergiaupm}, the general asymptotic analysis  in \cite{DeMarchisIanniPacellaJEMS} apply to the solution $\apice{u}{m}_p$ (see \cite[Proposition 2.2]{DeMarchisIanniPacellaJEMS} or also   \cite[Proposition 2.4, Corollary 2.6]{DIPpositive}). In particular, since the domain is now a ball and $\apice{u}{m}_p$ is  radial, 
it follows that it concentrates at only $1$ point, which is the origin (i.e. $k=1$  and $\mathcal S=\{0\}$ in  the notation of \cite{DeMarchisIanniPacellaJEMS,DIPpositive}). In our notations:
\begin{lemma}\label{theorem: k=1EConvDeboleAZeroENEICOMPATTI}
	There exists $C_m, \widetilde C_m>0$ such that, for $p$ large,
	\[1\leq\|\apice{u}{m}_p\|_{\infty}=\apice{u}{m}_p(0)\leq C_m,\]
\begin{equation}\label{P3estimate}
pr^2|\apice{u}{m}_p(r)|^{p-1}\leq \widetilde C_m, \ \text{ for all } r\in [0,1].
\end{equation}	
Moreover, there exists a function $v\in C^1(B\setminus\{0\})$ such that
\[p\apice{u}{m}_p=v+o(1)\ \mbox{ in }C^1_{loc}(B\setminus\{0\}), \mbox{ as }p\rightarrow +\infty.\]
In particular,
\begin{equation}\label{ZeroSuCompatti}
p(\apice{u}{m}_p)^p =o(1) \text{ uniformly in }K \text{ for all compact sets } K\subset \mathcal B\setminus\{0\}.
\end{equation}
\end{lemma}
Observe that from \cite{DeMarchisIanniPacellaJEMS}  we also know that a \emph{first bubble} always appears. Indeed, \cite[Proposition 2.2]{DeMarchisIanniPacellaJEMS} also says that, when $\alpha=0$,   $\apice{\varepsilon}{m}_{0,\alpha,p}=o(1)$, where $\apice{\varepsilon}{m}_{0,\alpha,p}$ is the first scaling parameter as defined in \eqref{scpar}, 
and that $\apice{\eta}{m}_{0,\alpha,p}=Z_{0,\alpha} +o(1)$ in  $C^1_{loc}(0,+\infty)$
as $p\rightarrow +\infty$,  where   $\apice{\eta}{m}_{0,\alpha,p}$ is a first rescaling of $\apice{u}{m}_p$ defined formally as  the first rescaling $\apice{\xi}{m}_{0,\alpha,p}$  in \eqref{scsol}, but in a larger domain, namely for any  $r\in [0,\frac{1}{\apice{\varepsilon}{m}_{0,\alpha,p}}]$. Next, we restrict to the first nodal region to prove the convergence  of $\apice{\xi}{m}_{0,\alpha,p}$ stated in \eqref{convBubblesG}.

\subsection{Main partial results and an iterative strategy}\label{section:strategiaRIcorsiva}
The proof for the case $\alpha=0$  in Theorem \ref{theorem:mainDirichletG} and in Theorem \ref{theorem:analisiAsintoticaRiscalateAlpha} (Dirichlet Lane-Emden) is postponed to the end of Section \ref{section:Lane} (see Sections \ref{section:conclusionDirichletLane_emden} and \ref{section:conclusionDirichletLane_emden2}, respectively). It is a consequence of some preliminary results from Section \ref{sectin:LabePreliminaries} and of the following two partial results.
\begin{theorem}[Sharp asymptotic constants]\label{theorem:mainDirichlet}
	Let $m\in\mathbb N$, $m\geq 1$ and let $\apice{u}{m}_{p}$ be the radial solution of \eqref{problem}-\eqref{MaxInZero} with $m-1$ interior zeros and let $\apice{r}{m}_{i,p}$ and $\apice{s}{m}_{i,p}$ be its zeros and critical points (see \eqref{rp},\eqref{sp}). Then 
	\begin{eqnarray}
	&(\apice{r}{m}_{i,p})^{\frac{2}{p-1}}=\apice{R}{m}_{i}  +o(1), & \ i=1,\ldots, m-1,
	\label{raggio}
	\\
	&p|(\apice{u}{m}_{p})'(\apice{r}{m}_{i,p})|(\apice{r}{m}_{i,p})=\apice{D}{m}_i  +o(1), &\  i=1,\ldots, m,
	\label{derivataInRaggio}
	\\
	&(\apice{s}{m}_{i,p})^{\frac{2}{p-1}}=\apice{S}{m}_i +o(1),& \ i=1,\ldots, m-1,
	\label{puntoCritico}
	\\
	&|\apice{u}{m}_{p}(\apice{s}{m}_{i,p})|=\apice{M}{m}_{i} +o(1),&\  i=0,\ldots, m-1,
	\label{valoreInPuntoCritico}
	\end{eqnarray}
	as $p\rightarrow +\infty$, where  $\apice{R}{m}_i, \apice{S}{m}_i, \apice{M}{m}_{i}, \apice{D}{m}_i>0$ are the constants in Definition \ref{definition:Constants2}.  
\end{theorem}

\begin{theorem}[Last bubble]
	\label{th:ultimaBubble}
	Let $m\in\mathbb N$, $m\geq 1$ and set
\begin{align*}
\apice{\varepsilon}{m}_{m-1,p}&:= \left[p |\apice{u}{m}_{p}(\apice{s}{m}_{m-1,p})|^{p-1}\right]^{-\frac{1}{2}},\\
\apice{z}{\emph m}_{m-1,p}(r)&:=
\frac{p}{|\apice{u}{m}_{p}(\apice{s}{m}_{m-1,p})|}\left[(-1)^{m-1} \apice{u}{m}_p(\apice{s}{m}_{m-1,p}+ \apice{\varepsilon}{m}_{m-1,p} r)  -    |\apice{u}{m}_{p}(\apice{s}{m}_{m-1,p})| \right], 
\end{align*}
for  $r\in (\apice{a}{m}_{m-1,p},\apice{b}{m}_{m-1,p})$	
where  
\[\apice{a}{m}_{m-1,p}:=\left\{\begin{array}{lr}
0, &\qquad \mbox{ if }m=1, \\
\frac{\apice{r}{m}_{m-1,p}-\apice{s}{m}_{m-1,p}}{\apice{\varepsilon}{m}_{m-1,p}}\; (<0), &\qquad \mbox{ if }   m\geq 2,
\end{array}
\right.\]	and
\[
\apice{b}{m}_{m-1,p}:=\frac{1-\apice{s}{m}_{m-1,p}}{ \apice{\varepsilon}{m}_{m-1,p}}\; (>0).\]
Then
\begin{equation}
\label{epsilon0Ultimo}
\apice{\varepsilon}{m}_{m-1,p}=o(1),
\end{equation}
\begin{eqnarray}
\label{repsilon0Ultimo}\label{Cmi}
&&\frac{\apice{r}{m}_{m-1,p}}{\apice{\varepsilon}{m}_{m-1,p}}=o(1)\quad (m\neq 1),
\\
\label{sepsilon0Ultimo}\label{Bmi}
&&\frac{\apice{s}{m}_{m-1,p}}{\apice{\varepsilon}{m}_{m-1,p}}= \sigma_{m-1} +o(1), 
\end{eqnarray}
and
\begin{equation}	\label{zpmenoNgUltimo}
\apice{z}{\emph m}_{m-1,p} = Z_{m-1}(\cdot +\sigma_{m-1})+o(1)\quad\mbox{in }C^1_{loc}(-\sigma_{m-1}, +\infty),
\end{equation}
as $p\rightarrow +\infty$, where \begin{equation}\label{sigmaEta}\sigma_{m-1}:=\sigma_{m-1,0}=\sqrt{\frac{\theta_{m-1}^2-4}{2}}\end{equation} and \begin{equation}\label{espressioneZ}Z_{m-1}:=Z_{m-1,0}=
\log\frac{2(\theta_{m-1})^2(\beta_{m-1})^{\theta_{m-1}}|x|^{(\theta_{m-1}-2)}}{((\beta_{m-1})^{\theta_{m-1}}+|x|^{\theta_{m-1}})^2},
\end{equation}
(with $\beta_0:=2\sqrt{2} $, $\beta_{m-1}:=	\frac{1}{ \sqrt2}(\theta_{m-1}+2)^{\frac{\theta_{m-1}+2}{2\theta_{m-1}}}(\theta_{m-1}-2)^{\frac{\theta_{m-1}-2}{2\theta_{m-1}}} $, $m\geq 2$)
  are as in Theorem \ref{theorem:analisiAsintoticaRiscalateAlpha} with $\alpha=0$.
\end{theorem}

The proof of Theorem \ref{theorem:mainDirichlet} is done by an iterative argument on the number $m$. More precisely, it  is  obtained  once the following \emph{inductive basis} and \emph{inductive step} are shown.
\begin{proposition}[Inductive basis]\label{theorem:mainDirichletm=1}
	Theorem \ref{theorem:mainDirichlet} holds for $m=1$.
\end{proposition}

\begin{proposition}[Inductive step]\label{theorem:mainDirichlet_PASSO_INDUTTIVO}
	Let $m\geq 2$. If Theorem \ref{theorem:mainDirichlet} holds for the solution $\apice{u}{m-1}_{\!\!\!p}$, then it holds for the solution  $\apice{u}{m}_p$.
\end{proposition}

Proposition \ref{theorem:mainDirichletm=1} (\emph{inductive basis}) was essentially already known from the general works  \cite{RenWeiTAMS1994, RenWeiPAMS1996, AdimurthiGrossi} on the asymptotic behavior of the least energy solutions of the Dirichlet Lane-Emden problem. In Section \ref{section:base} we will sketch a proof for completeness. As we will see it exploits the following:
\begin{proposition}
	\label{theorem:analisiAsintoticaRiscalatem=1}
	Theorem \ref{th:ultimaBubble} holds for $m=1$.
\end{proposition}

The proof of Proposition \ref{theorem:mainDirichlet_PASSO_INDUTTIVO} (\emph{inductive step}) is based on the key observation that the radial solution $ \apice{u}{m}_{p}$ with $m-1$ interior zeros and the radial solution $ \apice{u}{m-1}_{\!\!\!p}$ with $m-2$ interior zeros are related by the following change of variable
\[
\apice{u}{m-1}_{\!\!\!p}(r)= (\apice{r}{\emph m}_{m-1,p})^{\frac{2}{p-1}} \apice{u}{m}_{p}(\apice{r}{\emph m}_{m-1,p} r), \quad r\in[0,1].\]
As a consequence, it can be proved that, if one knows that Theorem \ref{theorem:mainDirichlet} holds for the solution $\apice{u}{m-1}_{\!\!p}$, then, in order to prove that it holds
for the solution $\apice{u}{m}_p$, one has to show only the \emph{last} $4$ relations
	\begin{eqnarray*}
&&(\apice{r}{m}_{m-1,p})^{\frac{2}{p-1}}=\apice{R}{m}_{m-1} +o(1),  
\\
&&p|(\apice{u}{m}_{p})'(1)|=\apice{D}{m}_m  +o(1),
\\
&&(\apice{s}{m}_{m-1,p})^{\frac{2}{p-1}}=\apice{S}{m}_{m-1}  +o(1),
\\
&&|\apice{u}{m}_p(\apice{s}{m}_{m-1,p})|=\apice{M}{m}_{m-1} +o(1), 
\end{eqnarray*}
as $p\rightarrow +\infty$. The proof of these $4$ sharp limits is at the core of Section \ref{section:stepInduttivo}. It is mainly based on integral estimates and  ODE techniques, and we exploit many arguments from \cite{GrossiGrumiauPacella2}, where the case $m=2$ was investigated. One of the ingredients of the proof is the 
description of the bubble behavior for the \emph{last} rescaled function $\apice{z}{m}_{m-1,p}$ of $\apice{u}{m}_{p}$ as $p\rightarrow +\infty$.
\begin{proposition}
	\label{theorem:ultimaBubble_PASSO_INDUTTIVO}
	Let $m\geq 2.$ If Theorem \ref{theorem:mainDirichlet} holds for radial solutions $\apice{u}{m-1}_{\!\!\!p}$ with $m-2$ interior zeros, then Theorem \ref{th:ultimaBubble} holds.
\end{proposition}

Observe that, once Propositions  \ref{theorem:analisiAsintoticaRiscalatem=1} and 
\ref{theorem:ultimaBubble_PASSO_INDUTTIVO} are proven, 
an iterative argument on $m$ gives also the proof of Theorem \ref{th:ultimaBubble}, using  the propositions as  \emph{inductive basis} and   \emph{inductive step} respectively.

\

In conclusion in order to show Theorems \ref{theorem:mainDirichlet} and  \ref{th:ultimaBubble} it remains to show
\begin{itemize}
	\item the \emph{inductive basis}: Propositions \ref{theorem:mainDirichletm=1} and \ref{theorem:analisiAsintoticaRiscalatem=1} (see Section \ref{section:base});
\item the \emph{inductive steps}: Propositions \ref{theorem:mainDirichlet_PASSO_INDUTTIVO} and \ref{theorem:ultimaBubble_PASSO_INDUTTIVO} (see Section \ref{section:stepInduttivo}).
\end{itemize}
\
\subsection{The case $m=1$} \label{section:base} Here we prove Propositions \ref{theorem:mainDirichletm=1} and \ref{theorem:analisiAsintoticaRiscalatem=1}.

\begin{proof}[Proof of Proposition \ref{theorem:analisiAsintoticaRiscalatem=1}] We have to show that
	\begin{eqnarray}
	\label{ahp}
	&\apice{\varepsilon}{1}_{0,p}=o(1),
	\\
	\label{ahahp}
	&\apice{z}{1}_{0,p}=Z_0 +o(1).
	\end{eqnarray}
as $p\rightarrow +\infty$ (indeed, since $\theta_0=2$, $\sigma_0=\sqrt{\frac{\theta_0^2-4}{2}}=0$). \eqref{ahp} was known from \cite{RenWeiTAMS1994}, we repeat the proof for completeness. By Poincare inequality,
\[\int_{\Omega} |\nabla u_p|^2\,dx=\int_{\Omega}|u_p|^{p+1}\,dx\leq \|u_p\|_{L^{\infty}(\Omega)}^{p-1}\int_{\Omega}|u_p|^2\,dx\leq \frac{\|u_p\|_{L^{\infty}(\Omega)}^{p-1}}{\lambda_1(\Omega)}\int_{\Omega}|\nabla u_p|^2\,dx,\]
where $\lambda_1(\Omega)$ is the first eigenvalue of $-\Delta$ in $H^1_0(\Omega)$, hence
\begin{equation}\label{Mnzero}
u_p(0)^{p-1}=\|u_p\|_{L^{\infty}(\Omega)}^{p-1}\geq \lambda_1(\Omega)>0\end{equation}
and, as a consequence, recalling that 
$
\apice{\varepsilon}{1}_{0,p}:=\left[p|\apice{u}{1}_p(0)|^{p-1}\right]^{-\frac{1}{2}}
$, we derive \eqref{ahp}.\\
The proof of \eqref{ahahp} can be found  in \cite{AdimurthiGrossi}.
Observe that the functions $\apice{z}{1}_{0,p}$ satisfy 
\[
\left\{
\begin{array}{lr}
\displaystyle
(\apice{z}{1}_{0,p})''+\frac{1}{r}(\apice{z}{1}_{0,p})'+\left|1+\frac{\apice{z}{1}_{0,p}}{p}\right|^{p-1}\left(1+\frac{\apice{z}{1}_{0,p}}{p}\right) =0, \quad r\in(0,\apice{b}{1}_{0,p}),
\\
\displaystyle
\apice{z}{1}_{0,p}(0)=(\apice{z}{1}_{0,p})'(0)=0,
\\
\displaystyle
\apice{z}{1}_{0,p}(\apice{b}{1}_{0,p})=-p,\\
\displaystyle
\apice{z}{1}_{0,p}< 0,
\end{array}\right..
\] 	
with $\left|1 + \frac{\apice{z}{1}_{0,p}}{p}\right|\leq1$ and  $\apice{b}{1}_{0,p}=(\apice{\varepsilon}{1}_{0,p})^{-1}\rightarrow +\infty$, as $p\rightarrow +\infty$. In \cite[Theorem 1.1]{AdimurthiGrossi} it is  shown that $\apice{z}{1}_{0,p}$ are locally uniformly bounded in $(0,+\infty)$. 
The rest of the proof is then standard:	by elliptic estimates, we have that  $\apice{z}{1}_{0,p}$ are uniformly
bounded in $C^2_{loc}(0, +\infty)$, so they  converge in $C^1_{loc}(0, +\infty)$ to a non-positive solution  of $-z''-\frac{z'}{r}=e^z$ with $z(0)=0$. Observe that the radial function $\mathbb R^2\ni x\mapsto z(|x|)$ has finite energy by Fatou's lemma  so, by the classification results for the Liouville equation, it must necessarily be the function $Z_0$.
\end{proof}

\begin{proof}[Proof of Proposition \ref{theorem:mainDirichletm=1}] 
	We have to show
	\begin{eqnarray}
	&& \lim_{p\rightarrow +\infty}|\apice{u}{1}_p(0)|=\apice{M}{1}_0,
	\label{limiteNormaLInfinitom=1}
	\\
	&& \lim_{p\rightarrow +\infty}p|(\apice{u}{1}_{p})'(1)|=\apice{D}{1}_1.
	\label{limiteenergiam=1altra}
	\end{eqnarray}
From \cite{RenWeiTAMS1994, RenWeiPAMS1996}  we know that  the least energy solutions  satisfy the following energy condition
	\begin{equation}
	\label{energylimit8pie}
p\int_0^1 |\apice{u}{1}_p|^{p+1} rdr= 4e +o(1),\quad\mbox{as $p\to+\infty$}.
	\end{equation}
Identity \eqref{limiteNormaLInfinitom=1} was obtained in \cite{AdimurthiGrossi} by a contradiction argument which uses \eqref{energylimit8pie} and Proposition \ref{theorem:analisiAsintoticaRiscalatem=1}. Here we write a more direct proof
which exploits also a pointwise estimate recently obtained in \cite{DIPpositive} for the rescaled function $\apice{z}{1}_{0,p}$ of the positive solution $\apice{u}{1}_p$. Indeed, in \cite[Lemma 4.4 and Proposition 4.3]{DIPpositive} it is proven that, for any $\delta>0$, there exist $R_{\delta}>1$, $C_{\delta}>0$ and $p_{\delta}>1$ such that
\begin{equation}\label{stimaIacopetti}
\apice{z}{1}_{0,p}(y)\leq \left(4-\delta\right)\log\frac{1}{|y|}+C_{\delta}, \quad 2R_\delta\leq |y|\leq \frac{r}{\apice{\varepsilon}{1}_{0,p}},
\end{equation}
provided  $p\geq p_{\delta}$. From \eqref{stimaIacopetti} one immediately derives  \eqref{limiteNormaLInfinitom=1} using the energy estimate \eqref{energylimit8pie}, changing variable in the integral and passing to the limit by Lebesgue's theorem (thanks to \eqref{stimaIacopetti}) and the results in Proposition \ref{theorem:analisiAsintoticaRiscalatem=1}:
\begin{eqnarray*}
	4e+o(1)&\overset{\eqref{energylimit8pie}}{=}& p \int_{0}^1|\apice{u}{\emph 1}_p(r)|^{p+1}rdr	\\
	&{=}&
	|\apice{u}{1}_{p}(0)|^2\int_0^{1/\apice{ \varepsilon}{1}_{0,p}}\left(1+\frac{\apice{z}{1}_{0,p}}{p} \right)^{p+1}\!\!\!\!\!\!r dr
	\\
	&\overset{\eqref{stimaIacopetti}+\text{Prop.}\ref{theorem:analisiAsintoticaRiscalatem=1} }{=}&
	|\apice{u}{1}_{p}(0)|^2 \left(\int_{0}^{\infty} e^{Z_0}rdr+o(1)\right)
	\\
	&=& |\apice{u}{1}_{p}(0)|^2 \left(2\theta_0+o(1)\right)
\end{eqnarray*}
as $p\rightarrow +\infty$. Identity \eqref{limiteNormaLInfinitom=1} follows recalling that $\theta_0=2$, and that $\apice{M}{1}_0=\sqrt{e}$.
\\
Similarly, recalling that $\apice{D}{1}_1=(2+\theta_0)\sqrt{e}=4\sqrt{e}$, one has that \eqref{limiteenergiam=1altra}:	
\begin{eqnarray*}
	p|(\apice{u}{1}_{p})'(1)|&\overset{\text{Lemma }\ref{lemmaN}}{=}& p \int_{0}^1|\apice{u}{\emph 1}_p(r)|^{p}rdr =
	\apice{u}{1}_{p}(0)\int_0^{1/\apice{ \varepsilon}{1}_{0,p}}\left(1+\frac{\apice{z}{1}_{0,p}}{p} \right)^{p}r dr
	\\
	&\overset{\eqref{limiteNormaLInfinitom=1}}{=}&
	\apice{M}{1}_0 \int_{0}^{\infty} e^{Z_0}rdr+o(1)
	= (2+\theta_0)\apice{M}{1}_0+o(1) = 4\sqrt{e}+o(1),
\end{eqnarray*}
as $p\rightarrow +\infty$.
\end{proof}
\
\subsection{The inductive step}\label{section:stepInduttivo}
In this section we prove Propositions \ref{theorem:mainDirichlet_PASSO_INDUTTIVO} and \ref{theorem:ultimaBubble_PASSO_INDUTTIVO}.
In order to shorten the notations let us set 
\[
\begin{array}{lr}
\varepsilon_p:=\apice{\varepsilon}{m}_{m-1,p},
\\
s_p:=\apice{s}{m}_{m-1,p},
\\
r_p:= \apice{r}{m}_{m-1,p},
\end{array}
\] 
	\begin{equation}\label{zpmenoNsimplify}  
	z_p:=\apice{z}{m}_{m-1,p},
	\end{equation} 
	\[\begin{array}{lr}
	a_p:=\apice{a}{m}_{m-1,p}\ (<0),\\
	b_p:=\apice{b}{m}_{m-1,p}\ (>0).
	\end{array}
	\]
Observe that  $z_{p}$ solves 
\begin{equation}\label{equazionizp}
\left\{
\begin{array}{lr}
\displaystyle
(z_{p})''+\frac{1}{r +\frac{s_p}{\varepsilon_{p}}}(z_{p})'+\left|1+\frac{z_{p}}{p}\right|^{p-1}\left(1+\frac{z_{p}}{p}\right) =0, \quad r\in(a_p,b_p)
\\
\displaystyle
z_{p}(0)=(z_{p})'(0)=0
\\
\displaystyle
z_{p}(a_p)=-p\\
z_{p}(b_p)=-p\\
\displaystyle
z_{p}< 0. 
\end{array}\right.
\end{equation}

\begin{lemma}[Relation between $ \apice{u}{m}_{p}$ and $ \apice{u}{m-1}_{\!\!\!p}$] \label{lemma:legami}
		Let $m\in\mathbb N$, $m\geq 2$, then 
		\begin{equation}\label{legame}
		\apice{u}{m-1}_{\!\!\!p}(r)= (r_p)^{\frac{2}{p-1}} \apice{u}{m}_{p}(r_p r).\end{equation}
		As a consequence,
		\begin{eqnarray}
		\label{SecondoLegame}
		&& \apice{r}{\emph m-1}_{\!\!\!j,p}=\frac{\apice{r}{\emph m}_{j,p}}{r_p}, \qquad 1\leq j\leq m-1,
		\\
		\label{PrimoLegame}
		&& \apice{s}{\emph m-1}_{\!\!\!j,p}=\frac{\apice{s}{\emph m}_{j,p}}{r_p} , \qquad 0\leq j< m-1,
		\\
		\label{TerzoLegame}
		&&  \frac{1}{\apice{\varepsilon}{\emph m-1}_{\!\!\!j,p}}=\frac{r_p}{\apice{\varepsilon}{\emph m}_{j,p}}  , \qquad 0\leq j< m-1.
		\end{eqnarray}
\end{lemma}
\begin{proof}
By a direct computation it is easy to verify that the function
$w(r):=(r_p)^{\frac{2}{p-1}} \apice{u}{m}_{p}(r_p r)$, $r\in (0,1)$, solves \eqref{problem} and \eqref{MaxInZero}. Moreover, it has $m-1$ nodal regions and so, by the uniqueness, $w=\apice{u}{m-1}_{\!\!\!p}$.
\end{proof}

	Hence, we have the following.
\begin{lemma} 	Let $m\in\mathbb N$, $ m\geq 2$, then 
		\begin{eqnarray}	\label{QuartoLegame}
		&& \frac{\apice{s}{\emph m-1}_{\!\!\!j,p} }{\apice{\varepsilon}{\emph m-1}_{\!\!\!j,p}}
		=\frac{\apice{s}{\emph m}_{j,p}}{
			\apice{\varepsilon}{\emph m}_{j,p}},\qquad 0\leq j< m-1,
		\\ 
		\label{QuintoLegame}
		&& \frac{\apice{r}{\emph m-1}_{\!\!\!j,p} }{\apice{\varepsilon}{\emph m-1}_{\!\!\!j,p}}
		=\frac{\apice{r}{\emph m}_{j,p}}{
			\apice{\varepsilon}{\emph m}_{j,p}},\qquad 1\leq j< m-1,
		\\ 
		\label{legamiRaggi}
		&& r_p=\frac{\apice{\varepsilon}{\emph m}_{j,p}}{\apice{\varepsilon}{\emph m-1}_{\!\!\!j,p}} \overset{(*)}{=}\frac{\apice{s}{\emph m}_{j,p}}{\apice{s}{\emph m-1}_{\!\!\!j,p}}\overset{(*)}{=} \frac{\apice{r}{\emph m}_{j,p}}{\apice{r}{\emph m-1}_{\!\!\!j,p}},\qquad 0\leq j< m-1, j\neq 0\mbox{ in $(*)$.}
		\end{eqnarray}
		As a consequence,
		\begin{equation}\label{scalingMj}
		|\apice{u}{m-1}_{\!\!\!p}(\apice{s}{m-1}_{\!\!\!j,p})|=  ( r_p)^{\frac{2}{p-1}}|\apice{u}{m}_{p}(\apice{s}{m}_{j,p})|<|\apice{u}{m}_{p}(\apice{s}{m}_{j,p})|   ,\qquad 0\leq j< m-1.\end{equation}
\end{lemma}

\begin{lemma}\label{lemma:legameCostantiPassiConsecutivi}
	Let $m\geq 2.$ If Theorem \ref{theorem:mainDirichlet} holds for radial solutions $\apice{u}{m-1}_{\!\!\!p}$ with $m-2$ interior zeros, then
\begin{eqnarray*}
	&\displaystyle\lim_{p\rightarrow +\infty}
(r_p)^{-\frac{2}{p-1}}(\apice{r}{m}_{i,p})^{\frac{2}{p-1}}=\apice{R}{m-1}_{\!\!\!i}, & \quad i=1,\ldots, m-2 \ (\mbox{if }m\geq 3),
	\\
	&\displaystyle
	\lim_{p\rightarrow +\infty}(r_p)^{\frac{2}{p-1}}p|(\apice{u}{m}_{p})'(\apice{r}{m}_{i,p})|(\apice{r}{m}_{i,p})
	=\apice{D}{m-1}_{\!\!i}, & \quad i=1,\ldots, m-1,
	\\
	&\displaystyle\lim_{p\rightarrow +\infty}
	(r_p)^{-\frac{2}{p-1}}(\apice{s}{m}_{i,p})^{\frac{2}{p-1}}=\apice{S}{m-1}_{\!\!\!i},
	 &\quad i=0,\ldots, m-2,
	\\
	&\displaystyle\
	\lim_{p\rightarrow +\infty} (r_p)^{\frac{2}{p-1}}|\apice{u}{m}_p(\apice{s}{m}_{i,p})|=\apice{M}{m-1}_{\!\!\!i}, &\quad i=0,\ldots, m-2,
\end{eqnarray*}
where $\apice{R}{m-1}_{\!\!\!i}$, $\apice{D}{m-1}_{\!\!i}$, $\apice{S}{m-1}_{\!\!\!i}$, $\apice{M}{m-1}_{\!\!\!i}$ are the constants in Theorem \ref{theorem:mainDirichlet}.
\end{lemma}
\begin{proof}
By \eqref{legame}, \eqref{SecondoLegame}, \eqref{PrimoLegame}, and the inductive assumption,
	\begin{eqnarray*}
	&&\displaystyle\lim_{p\rightarrow +\infty}
		(r_p)^{-\frac{2}{p-1}}(\apice{r}{m}_{i,p})^{\frac{2}{p-1}}\overset{\eqref{SecondoLegame}}{=}\lim_{p\rightarrow +\infty}(\apice{r}{m-1}_{\!\!\!i,p})^{\frac{2}{p-1}}=\apice{R}{m-1}_{\!\!\!i},
		\\
	&&\displaystyle
		\lim_{p\rightarrow +\infty}(r_p)^{\frac{2}{p-1}}p|(\apice{u}{m}_{p})'(\apice{r}{m}_{i,p})|(\apice{r}{m}_{i,p})
	\overset{\begin{subarray}{l}\eqref{legame}\\
			\eqref{SecondoLegame}	\end{subarray}}{=}
		\lim_{p\rightarrow +\infty}p|(\apice{u}{m-1}_{\!\!\!p})'(\apice{r}{m-1}_{\!\!\!i,p})|(\apice{r}{m-1}_{\!\!\!i,p})=\apice{D}{m-1}_{\!\!i},
		\\
	&&\displaystyle\lim_{p\rightarrow +\infty}
		(r_p)^{-\frac{2}{p-1}}(\apice{s}{m}_{i,p})^{\frac{2}{p-1}}\overset{\eqref{PrimoLegame}}{=}\lim_{p\rightarrow +\infty}(\apice{s}{m-1}_{\!\!\!i,p})^{\frac{2}{p-1}}=\apice{S}{m-1}_{\!\!\!i},
		\\
	&&\displaystyle\
		\lim_{p\rightarrow +\infty} (r_p)^{\frac{2}{p-1}}|\apice{u}{m}_p(\apice{s}{m}_{i,p})|
		\overset{\begin{subarray}{l}\eqref{legame}\\
				\eqref{PrimoLegame}	\end{subarray}}{=}		
		\lim_{p\rightarrow +\infty}|\apice{u}{m-1}_{\!\!\!p}(\apice{s}{m-1}_{\!\!\!i,p})|	=\apice{M}{m-1}_{\!\!\!i}. 
	\end{eqnarray*}
\end{proof}

The previous result implies that, in order to prove Proposition \ref{theorem:mainDirichlet_PASSO_INDUTTIVO}, it is enough to study the asymptotic behavior of the $4$ quantities
\[(r_p)^{\frac{2}{p-1}}, \quad p|(\apice{u}{m}_p)'(1) |,\quad(s_p)^{\frac{2}{p-1}}, \quad |\apice{u}{m}_p(s_p)|. \]
We start with a preliminary result.
	\begin{lemma} Let $m\in\mathbb N$, $m\geq 2$. For any $p>1$,
		\begin{equation}\label{MassimiNonNulli}
		|\apice{u}{m}_{p}(s_p)|^{p-1}\geq \lambda_1>0, 
		\end{equation}
	where $\lambda_1$ is the first eigenvalue of $-\Delta$ in $H^1_0(\Omega)$,	as a consequence 
		\begin{equation}\label{epsilonVaAZero}
		\varepsilon_{p} =\left[p|\apice{u}{m}_p(s_p)|^{p-1}\right]^{-\frac{1}{2}} =o(1)\qquad\mbox{ as }p\rightarrow +\infty.
		\end{equation}
	\end{lemma}
\begin{proof}\eqref{epsilonVaAZero} clearly follows from \eqref{MassimiNonNulli}. The proof of \eqref{MassimiNonNulli} is similar to the one of \eqref{Mnzero}. Indeed let $\Omega_p:=\{r_p<|x|<1\}$, then by the monotonicity property of the first eigenvalue $\lambda_1(\Omega)\leq \lambda_1(\Omega_p)$, moreover  $|u_p(s_p)|=\|u_p\|_{L^{\infty}(\Omega_p)}$ and so, by Poincare inequality:
	\[\int_{\Omega_p} |\nabla u_p|^2\,dx=\int_{\Omega_p}|u_p|^{p+1}\,dx\leq |u_p(s_p)|^{p-1}\int_{\Omega_p}|u_p|^2\,dx\leq \frac{|u_p(s_p)|^{p-1}}{\lambda_1(\Omega)}\int_{\Omega_p}|\nabla u_p|^2\,dx.\]
\end{proof}
	From this we easily deduce the following.
	\begin{lemma}[The nodal lines shrink to $(0,0)$]\label{Proposition:spVaAZero} Let $m\in\mathbb N^+$, $m\geq 2$. We have 
		\[s_p=o(1)   \mbox{\ (and so also $r_p=o(1)$) \qquad as }\ p\rightarrow +\infty.\]
	\end{lemma}
	\begin{proof} 
		If, by contradiction, $s_{p_n}\geq \alpha>0$  for a sequence $p_n\rightarrow +\infty$ as $n\rightarrow +\infty$, then, by  Lemma \ref{lemma:Strauss}, 
		\[\sqrt{p_n}|\apice{u}{m}_p(s_{p_n})|\leq \frac{C}{|\alpha|^{\frac{1}{2}}}\ \mbox{ for }n \mbox{ large}.\]
		So, the sequence $(\sqrt{p_n}|\apice{u}{m}_{p}(s_{p_n})|)_n$ would be bounded in contradiction with \eqref{MassimiNonNulli}.
	\end{proof}

\begin{lemma} Let $m\in\mathbb N$, $m\geq 2$. If Theorem \ref{theorem:mainDirichlet} holds for the solution $\apice{u}{m-1}_{\!\!p}$ with $m-2$ interior zeros, then there exist $p_m>1$, $K_m,C_m>0$ such that, for $p\geq p_m$,
	\begin{equation}\label{raggioLimitatoBasso}
	 0< K_m\leq  (r_p)^{\frac{2}{p-1}}(<1),\end{equation}
\begin{equation}\label{normaLInfinitoLimitataAlto}
(0<C\leq)\ |\apice{u}{m}_p(s_p)|\leq C_m.
\end{equation}
\end{lemma}
\begin{proof} The upper bound in \eqref{raggioLimitatoBasso} is immediate, since $r_p<1$, moreover
		\begin{eqnarray}
		\label{papo}
		p \int_{0}^1|\apice{u}{m-1}_{\!\!p}(r)|^{p+1}rdr &\overset{\text{Lemma \ref{lemmaP}}}{=} & \left(1+\frac{1}{p}\right)\frac{1}{4}\left[p (\apice{u}{m-1}_{\!\!p})'(1)\right]^2 \nonumber\\
		&\overset{\text{Thm \ref{theorem:mainDirichlet}  for $\apice{u}{m-1}_{\!\!\!\!p}$}}{=}& \frac{1}{4}\left[\apice{D}{m-1}_{\!\!m-1}\right]^2+o(1),
		\end{eqnarray}
as $p\rightarrow +\infty$, where for the last equality we have used the assumption that Theorem \ref{theorem:mainDirichlet} holds for  $\apice{u}{m-1}_{\!\!p}$. Moreover,
	\begin{eqnarray}
	\label{test}
		 p \int_{0}^1|\apice{u}{m-1}_{\!\!p}(r)|^{p+1}rdr &\overset{\eqref{legame}}{=}& p (r_p)^{\frac{2(p+1)}{p-1}}    \int_{0}^1|\apice{u}{\emph m}_p(r_p r)|^{p+1}rdr 
	\nonumber	\\
	& =& p (r_p)^{\frac{4}{p-1}}    \int_{0}^{r_p}|\apice{u}{\emph m}_p(s)|^{p+1}sds\nonumber \\
	&\leq&   (r_p)^{\frac{4}{p-1}}    p \int_{0}^{1}|\apice{u}{\emph m}_p(s)|^{p+1}sds
	\nonumber\\
	&\overset{\eqref{stimaenergiaupm}}{\leq}& (r_p)^{\frac{4}{p-1}}  E_m
		\end{eqnarray}
	for $p\geq p_m$. The lower bound in \eqref{raggioLimitatoBasso} then follows combining  \eqref{papo} with \eqref{test}. \\
	The lower bound in \eqref{normaLInfinitoLimitataAlto}   is due to \eqref{MassimiNonNulli}, moreover from \eqref{raggioLimitatoBasso} and from Lemma  \ref{lemma:legameCostantiPassiConsecutivi} (for which again we need the assumption on the validity of Theorem \ref{theorem:mainDirichlet} for  $\apice{u}{m-1}_{\!\!p}$) we also obtain the upper bound in \eqref{normaLInfinitoLimitataAlto}:
	\[\apice{M}{m-1}_{\!\!m-2}+o(1)\overset{\text{Lemma  \ref{lemma:legameCostantiPassiConsecutivi}}}{=} (r_p)^{\frac{2}{p-1}}|\apice{u}{m}_p(\apice{s}{m}_{m-2,p})|\overset{\eqref{raggioLimitatoBasso}}{\geq}K_m |\apice{u}{m}_p(\apice{s}{m}_{m-2,p})|\overset{\eqref{Mmonotoni}}{>}K_m |\apice{u}{m}_p(s_p)|.\]
\end{proof}

\begin{definition}
Let $m\in\mathbb N$, $m\geq 2$. If Theorem \ref{theorem:mainDirichlet} holds for the solution $\apice{u}{m-1}_{\!\!p}$ with $m-2$ interior zeros then, by \eqref{raggioLimitatoBasso} it is well defined, up to a subsequence,
\begin{equation}
\label{rmiinfty}
R:=\lim_{p\rightarrow +\infty} (r_p)^{\frac{2}{p-1}}\ (\in(0,1]).
\end{equation}
Moreover, by  \eqref{normaLInfinitoLimitataAlto} we can also define, up to a subsequence,
\begin{equation}
\label{Mmiinfty} 
M:=\lim_{p\rightarrow +\infty}|\apice{u}{m}_p(s_p)|\ (>0).
\end{equation}
\end{definition}

Next we prove that $R=\apice{R}{m}_{m-1}$ and $M=\apice{M}{m}_{m-1}$.

\begin{proposition}\label{prop:BmiValePerOgnim} Let $m\in\mathbb N^+$, $m\geq 2$. If Theorem \ref{theorem:mainDirichlet} holds for the solution $\apice{u}{m-1}_{\!\!\!p}$ with $m-2$ interior zeros, then there exists $\sigma>0$ such that
		\begin{equation}\label{tesibh}\lim_{p\rightarrow +\infty}\frac{s_p}{\varepsilon_p}= \sigma>0.
		\end{equation}
\end{proposition} 
\begin{proof}
By \eqref{P3estimate}, there exists $C>0$ independent of $p$ such that
		\begin{equation}
		\label{noinfi}
		\frac{s_p}{\varepsilon_p}\leq C.
		\end{equation}	
Assume, by contradiction, that, up to a subsequence, $\frac{s_p}{\varepsilon_p}
	\rightarrow 0$ as $p\rightarrow +\infty$. Then, since $0<r_p<s_p$, necessarily also  
	$a_p=\frac{r_p-s_p}{\varepsilon_p}\rightarrow 0$ and, by the change of variables, $r=s_p+\varepsilon_ps$, since $|\apice{u}{m}_p(s_p)|$ is bounded (see  \eqref{normaLInfinitoLimitataAlto}), we deduce that
	\begin{eqnarray*}
		p \int_{r_p}^{s_p}|\apice{u}{m}_p(r)|^{p}r\,dr &\leq&	p |\apice{u}{m}_p(s_p)|^p
		\int_{r_p}^{s_p}r\,dr
		\\
		&=&
		p |\apice{u}{m}_p(s_p)|^p (\varepsilon_p)^2\int_{a_p}^{0}\Big(\frac{s_p}{\varepsilon_p}+s\Big)\,ds
		\\
		&=&|\apice{u}{m}_p(s_p)|\left(\frac{s_p}{\varepsilon_p}|a_p|-\frac{a_p^2}{2}\right)\longrightarrow 0
	\end{eqnarray*}
	as $p\rightarrow +\infty$.
	But  choosing $s=r_{p}$ and $t=s_p$ into  	\eqref{N} one obtains, from Lemma \ref{lemma:legameCostantiPassiConsecutivi} (which holds true by the assumption on the validity of Theorem \ref{theorem:mainDirichlet} for  $\apice{u}{m-1}_{\!\!p}$) and \eqref{rmiinfty}, that	\[p\int_{r_p}^{s_p}  |\apice{u}{m}_p(r) |^p r \,dr	\overset{\eqref{N}}{=}p\left|(\apice{u}{m}_p)'(r_p) \right|r_p\longrightarrow \frac{\apice{D}{m-1}_{\!\!m-1}}{R}>0 \quad \text{as $p\rightarrow +\infty$,	}\]
	which gives a contradiction. 
\end{proof}

\begin{proposition}\label{prop:CmiValePerOgnim} Let $m\in\mathbb N^+$, $m\geq 2$. If  Theorem \ref{theorem:mainDirichlet} holds for the solution $\apice{u}{m-1}_p$ with $m-2$ interior zeros, then
\begin{equation}\label{tesiCh}
\frac{r_p}{\varepsilon_p} =o(1) \quad\mbox{ as }p\rightarrow +\infty.
\end{equation}
\end{proposition}
\begin{proof}
Let $\sigma>0$ be the constant in Proposition  \ref{prop:BmiValePerOgnim} and let  $\rho:=\lim_{p\rightarrow +\infty}\frac{r_p}{\varepsilon_p}$, then $\rho\in [0,\sigma]$, since $\frac{s_p}{\varepsilon_p}\rightarrow \sigma$ by Proposition \ref{prop:BmiValePerOgnim} and $0<r_p<s_p$. 
As a consequence $a_p=\frac{r_p-s_p}{\varepsilon_p}\rightarrow h:=\rho-\sigma\in [-\sigma, 0]$.
Proving $\rho=0$ is then equivalent with showing that $h=-\sigma$.\\
Assume by contradiction that $h>-\sigma$.
As $z_p(0)=0$ and $z_p(a_p)=-p$, by the mean value theorem in $(a_p,0)$ we have the existence of $t_p\in (a_p,0)$ such that
\begin{equation}\label{infinitolo}
|(z_p)'(t_p)|=\frac{|z_p(a_p)-z_p(0)|}{|a_p|}=\frac{p}{|a_p|}=\frac{p}{h+o(1)}\rightarrow +\infty,
\end{equation}
as $p\rightarrow +\infty$. Moreover, since $h>-\sigma$ by assumption, there exists $C_1>0$ such that  \begin{equation}\label{tippi}
t_p \geq -\sigma +C_1
\end{equation} for $p$ large.  On the other hand, from the equation \eqref{equazionizp},
\[\left[ (z_p)'(r)\left(r+\frac{s_p}{\varepsilon_p}\right) \right]'=-\left(r+\frac{s_p}{\varepsilon_p}\right)\left|1+\frac{z_{p}}{p}\right|^{p-1}\left(1+\frac{z_{p}}{p}\right),\mbox{ for }r\in (a_p,b_p).\]
Recall that $(z_p)'(0)=0$ and that, by definition, $\left|1+\frac{z_p}{p}\right|\leq 1$. Then, integrating on $(t_p,0)$, we have for $p$ large that
\[|(z_p)'(t_p)|\left|t_p+\frac{s_p}{\varepsilon_p}\right|\leq\int_{t_p}^0 \left|s+\frac{s_p}{\varepsilon_p}\right|\, ds,
\]
from which we obtain
\begin{equation}
\label{prec}
|(z_p)'(t_p)|\left| t_p+\sigma +o(1)\right|\leq \int_{-\sigma}^0 \left(|s|+|\sigma|+o(1)\right)\, ds \leq C_2.
\end{equation}
But \eqref{tippi} implies that  $\left| t_p+\sigma +o(1)\right|\geq C_3>0$ for $p$ large and so, by \eqref{prec},
\[|(z_p)'(t_p)|\leq C_3\]
uniformly in $p$, for $p$ large,
 reaching a contradiction with \eqref{infinitolo}.	
\end{proof}

\

Using  Proposition \ref{prop:BmiValePerOgnim} and Proposition \ref{prop:CmiValePerOgnim}  we can prove the convergence of the last rescaling $z_{p}$ of $\apice{u}{m}_p$.
\begin{proposition}
\label{prop:convergenzaUltimaRiscalata} 
Let $m\in\mathbb N^+$, $m\geq 2$. Assume that Theorem  \ref{theorem:mainDirichlet} holds for the solution $\apice{u}{m-1}_{\!\!p}$ with $m-2$ interior zeros. Let $\sigma>0$ be the constant in Proposition \ref{prop:BmiValePerOgnim}, then
\begin{equation}\label{convIntermediaz}z_{p} \longrightarrow Z(\cdot +\sigma)\quad\mbox{in }C^1_{loc}(-\sigma, +\infty), \quad\mbox{as $p\rightarrow +\infty$,}
\end{equation}
where
\begin{equation}\label{Z} Z(r):=\log\left(\frac{2\theta^2\beta^{\theta}|x|^{\theta-2}}{(\beta^{\theta}+|x|^{\theta})^2}\right),
\end{equation}
with
\begin{equation}\label{Zvari}
\beta=\frac{1}{\sqrt 2} (\theta+2)^{\frac{\theta+2}{2\theta}}(\theta-2)^{\frac{\theta -2}{2\theta}}
\end{equation}
and
\begin{equation}\label{Zvarii}\theta:=\sqrt{2(\sigma^2+2)} \ (>2)\end{equation}
is a singular radial solution of \[
\left\{
\begin{array}{lr}
-\Delta Z=e^Z+ 2\pi(2-\theta)\delta_{0}\quad\mbox{ in }\R^2,\\
\int_{\R^2}e^Zdx=4\pi\theta,
\end{array}
\right.
\]
where  $\delta_{0}$ is the Dirac measure centered at $0$.
\end{proposition}

\begin{proof}
The rescaled function $z_p$ satisfies \eqref{equazionizp} in $(a_p, b_p)$  	
with $\left|1 + \frac{z_{p}}{p}\right|\leq1$.
Observe that, by Propositions \ref{prop:BmiValePerOgnim} and \ref{prop:CmiValePerOgnim},
\begin{eqnarray*}
a_p\rightarrow -\sigma <0\quad \text{ and }\quad b_p\rightarrow +\infty\qquad \text{as $p\rightarrow +\infty$.}
\end{eqnarray*} 
Following the proof of \cite[Theorem 1.1]{AdimurthiGrossi} one can then show that $z_p$ is locally uniformly bounded in $(-\sigma,+\infty)$. 	By standard elliptic estimates, we have that $z_p$ is uniformly
	bounded in $C^2_{loc}(-\sigma, +\infty)$, so $z_p\to z$ in $C^1_{loc}(-\sigma,+\infty)$, as $p\rightarrow +\infty$, where $z$ solves
	\begin{equation}
	\left\{
	\begin{array}{lr}
	\displaystyle
	z''+\frac{1}{r +\sigma} z'+e^z=0, \qquad r\in (-\sigma,+\infty),
	\\
	\\
	\displaystyle
z(0)=z'(0)=0,
	\\
	\\
	\displaystyle
z\leq 0. 
	\end{array}\right.
	\end{equation}
Let us set $Z(r):=z(r-\sigma)$, so 	$Z$ satisfies
\begin{equation}
\left\{
\begin{array}{lr}
\displaystyle
Z''+\frac{1}{r } Z'+e^{Z}=0, \qquad r\in (0,+\infty),
\\
\\
\displaystyle
Z(\sigma)=Z'(\sigma)=0,
\\
\\
\displaystyle
Z\leq 0. 
\end{array}\right.
\end{equation}
Hence $Z$ must be the function in  \eqref{Z} with \eqref{Zvari} and \eqref{Zvarii}, which satisfies  
\[
-\Delta Z=e^Z+2\pi H\delta_{0}\quad\mbox{ in }\R^2,
\]
where  $\delta_{0}$ is the Dirac measure centered at $0$ and $H:=-\int_0^{\sigma}e^{Z(s)}sds$ (see  \cite[Proof of Proposition 3.1]{GrossiGrumiauPacella2}). \\
Using  \eqref{Z} with \eqref{Zvari} and \eqref{Zvarii} we have that
	\begin{equation}\label{intzerosigmaZ}
	-H=\int_0^{\sigma}e^{Z(s)}sds\overset{\eqref{Z}}{=}   
	2\theta^2\beta^{\theta}\int_0^{\sigma}
	\frac{s^{\theta-1}}{(\beta^{\theta}+s^{\theta})^2}ds
	=\frac{2\theta\sigma^{\theta}}{\beta^{\theta}+\sigma^{\theta}}\overset{ \eqref{Zvari}\&  \eqref{Zvarii}}{=}\theta-2
	\end{equation}
and similarly that 
\[\int_{\R^2}e^Zdx=4\pi\theta.\]
\end{proof}

Now we aim to show that $\sigma=\sigma_{m-1}$. Once this equality has been established, Proposition \ref{theorem:ultimaBubble_PASSO_INDUTTIVO} follows by combining \eqref{epsilonVaAZero} and Propositions \ref{prop:BmiValePerOgnim},  \ref{prop:CmiValePerOgnim}  and \ref{prop:convergenzaUltimaRiscalata}.  Observe that, by \eqref{Zvarii} and \eqref{sigmaEta}, the equality $\sigma=\sigma_{m-1}$ is equivalent to  $\theta=\theta_{m-1}$. We begin with some auxiliary lemmas.

\

\begin{lemma}\label{lemmaG} Let $m\in\mathbb N^+$, $m\geq 2$. Assume that Theorem  \ref{theorem:mainDirichlet} holds for the solution $\apice{u}{m-1}_{\!\!p}$ with $m-2$ interior zeros. Then
	\[p\int_{s_p}^1 	 |\apice{u}{m}_p(r)|^{p}r\,dr=\left(M+o(1)\right)I_p
	\]
	as $p\rightarrow +\infty$, where  $M>0$ is the constant in \eqref{Mmiinfty} and
	\begin{equation}\label{defIp}
	I_p:=
\int_0^{b_p}   \left|1+\frac{z_p (r )}{p}\right|^{p} \left(\frac{s_p}{\varepsilon_p} +  r \right)\,dr.
	\end{equation}
\end{lemma}
\begin{proof} The claim follows from the change of variable $s=s_p + \varepsilon_p r $ and the definitions of $z_p$ and $\varepsilon_p$, because 
	\begin{eqnarray}\label{comeQui} p\int_{s_p}^1 	 |\apice{u}{m}_p(s)|^{p}s\,ds
&=& p\varepsilon_{p}\int_0^{b_p}   |\apice{u}{m}_p (s_p + \varepsilon_p r )|^{p} (s_p+ \varepsilon_p r )\,dr  
\nonumber\\
&=& p \left(\varepsilon_p\right)^2\int_0^{b_p}   |\apice{u}{m}_p (s_p + \varepsilon_p r )|^{p} \left(\frac{s_p}{\varepsilon_p} +  r \right)\,dr  
\nonumber\\
&\overset{\eqref{defEpsilon}}{=}& 
|\apice{u}{m}_p(s_p)|\int_0^{b_p}   \left|1+\frac{z_p (r )}{p}\right|^{p} \left(\frac{s_p}{\varepsilon_p} +  r \right)\,dr 
\\
&\overset{\eqref{Mmiinfty}}{=}& 
\left(M+o(1)\right)\int_0^{b_p}    \left|1+\frac{z_p (r )}{p}\right|^{p} \left(\frac{s_p}{\varepsilon_p} +  r \right)\,dr. 
\nonumber
	\end{eqnarray} 
\end{proof}
Using Proposition \ref{prop:convergenzaUltimaRiscalata} we can now integrate $|\apice{u}{m}_p(r)|^{p+1}r$.
\begin{lemma}\label{lemma:j_PENULTIMO_integr_p+1}
Let $m\in\mathbb N^+$, $m\geq 2$. Assume that Theorem  \ref{theorem:mainDirichlet} holds for the solution $\apice{u}{m-1}_{\!\!p}$ with $m-2$ interior zeros. Then
\begin{align}\label{tint}
p\int_{r_{p}}^{s_p} |\apice{u}{m}_p|^{p+1}r\,dr= M^2( \theta-2) 	+ o(1),
\end{align}
as $p\rightarrow +\infty$, where $M>0$ is the constant in \eqref{Mmiinfty} and $\theta\ (>2)$ is the constant in \eqref{Zvarii}.
\end{lemma}
\begin{proof}
To compute \eqref{tint} we use the change of variable $r= s_p+\varepsilon_p s$,  exploit the definition of $z_p$  and $\varepsilon_p$, use that $\frac{s_p}{\varepsilon_p}=\sigma+o(1)$, $a_p=-\sigma +o(1)$ (by Propositions \ref{prop:CmiValePerOgnim} and \ref{prop:BmiValePerOgnim}), use the convergence of $z_p$ to $Z$ given by Proposition \ref{prop:convergenzaUltimaRiscalata} and the fact that $|\apice{u}{m}_p(s_p)|^2=M^2+o(1)$ by \eqref{Mmiinfty} (observe that, by assumption, Theorem \ref{theorem:mainDirichlet} holds for the solution $\apice{u}{m-1}_{\!\!\!p}$).  Namely, we have that
\begin{eqnarray*}p\int_{r_p}^{s_p} |\apice{u}{m}_p|^{p+1}r\,dr
&=&  p(\varepsilon_{p})^2\int_{a_p}^0  |\apice{u}{m}_p (s_p + \varepsilon_p s )|^{p+1} \left(\frac{s_p}{\varepsilon_p}+s\right)\,ds  
\\
&=&
|\apice{u}{m}_p(s_p)|^{2}\int_{a_p}^0\left|1+\frac{z_p}{p}  \right|^{p+1}\left(\frac{s_p}{\varepsilon_p}+s\right)\,ds
\\
&\overset{(\ast)}{=}& M^{2}\int_{-\sigma}^0  e^{Z(\sigma+s)}(\sigma+s)ds + o(1)
\\
	&=&
	M^{2}\int^{\sigma}_{0} e^{Z(s)}sds + o(1)
	\\
	&=&M^{2}( \theta-2)  + o(1),
\end{eqnarray*}
as $p\rightarrow +\infty$, where the passage to the limit in $(\ast)$ is due to Lebesgue's theorem ( $(1+\frac{z_p}{p})\leq 1$ and the integration is on a bounded interval), while  the last equality is a consequence of \eqref{Z} with \eqref{Zvari} and \eqref{Zvarii} in Proposition \ref{prop:convergenzaUltimaRiscalata}  (as already observed in \eqref{intzerosigmaZ}).
\end{proof}

\begin{lemma}\label{lemma:j_ULTIMO_integr_p+1}Let $m\in\mathbb N^+$, $m\geq 2$. Assume that Theorem  \ref{theorem:mainDirichlet} holds for the solution $\apice{u}{m-1}_{\!\!p}$ with $m-2$ interior zeros. Then
	\[p\int_{s_p}^{1} |\apice{u}{m}_p|^{p+1}r\,dr\leq  \left(M^2+o(1)\right)I_p,\]
as $p\rightarrow +\infty$, where $M>0$ is the constant in \eqref{Mmiinfty} and $I_p$ is defined in \eqref{defIp}.
\end{lemma}
\begin{proof}
We make the change of variable $r= s_p+\varepsilon_p s$ in the left hand side and use the definition of $z_p$  and of $\varepsilon_p$, then we observe that the inequality $(1+\frac{z_p}{p})\leq 1$ implies $(1+\frac{z_p}{p})^{p+1}\leq (1+\frac{z_p}{p})^p$, which allows to estimate the integral with $I_p$, that is,
\begin{eqnarray*}
	p\int_{s_p}^{1} |\apice{u}{m}_p(r)|^{p+1}r\,dr
&=&
|\apice{u}{m}_p(s_p)|^{2}\int_0^{b_p}\left|1+\frac{z_p}{p}  \right|^{p+1}\left(\frac{s_p}{\varepsilon_p}+s\right)\,ds
\\
&\leq & |\apice{u}{m}_p(s_p)|^{2} I_p
\\
&\overset{\eqref{Mmiinfty}}{=}& 
\left(M^2+o(1)\right)I_p.
\end{eqnarray*}
\end{proof}	
We are also able to compute the value of the derivative in the last zero of a solution.
\begin{lemma}\label{Lemma4.3} 
Let $m\in\mathbb N$, $m\geq 2$. If Theorem \ref{theorem:mainDirichlet} holds for radial solutions $\apice{u}{m-1}_{\!\!\!p}$ with $m-2$ interior zeros, then
\begin{equation}
	\label{corollarioLemma4.3A}
\apice{M}{m-1}_{\!\!m-2}\left(\theta_{m-2}+2 \right)=RM\left(\theta-2 \right),	
\end{equation}
where $R\in (0,1]$ is the constant in \eqref{rmiinfty}, $M>0$ is the constant in \eqref{Mmiinfty}, $\theta\ (>2)$ is the constant in \eqref{Zvarii} and  $\apice{M}{m-1}_{\!\!m-2}$ and $\theta_{m-2}$ are given in Theorem \ref{theorem:mainDirichlet}.
\end{lemma}
\begin{proof}
	Choosing $s=r_p$ and $t=s_p$ into \eqref{N} one has that
	\begin{equation}\label{inti1}
	p|(\apice{u}{m}_p)'(r_p)| (r_p)  = p\int^{s_p}_{r_p} |\apice{u}{m}_{p}(r)|^{p}r\,dr.
	\end{equation}
	We can now compute this integral similarly as we did in the proof of Lemma \ref{lemma:j_PENULTIMO_integr_p+1}, making the change of variable $r= s_p+\varepsilon_p s$ and passing to the limit as $p\rightarrow +\infty$. We use   that $\frac{s_p}{\varepsilon_p}=\sigma+o(1)$, $a_p=-\sigma +o(1)$ (by Proposition \ref{prop:CmiValePerOgnim} and Proposition \ref{prop:BmiValePerOgnim}), the convergence of $z_p$ to $Z$ in Proposition \ref{prop:convergenzaUltimaRiscalata} and also that $|\apice{u}{m}_p(s_p)|=M+o(1)$ by \eqref{Mmiinfty} (by assumption  Theorem \ref{theorem:mainDirichlet} holds for the solution $\apice{u}{m-1}_{\!\!\!p}$, so all these results hold), hence we obtain that
	\begin{eqnarray}
	\label{inti2}
	p\int_{r_p}^{s_p} |\apice{u}{m}_p(r)|^{p}r\,dr&=&
		p(\varepsilon_p )^2
		\int_{a_p}^0   |\apice{u}{m}_p(s_p+s\varepsilon_p  )|^{p}\left(\frac{s_p}{\varepsilon_p}+s\right)\,ds
		\nonumber\\
		&=&
	|\apice{u}{m}_p(s_p)|\int_{a_p}^0\left|1+\frac{z_p}{p}  \right|^{p}\left(\frac{s_p}{\varepsilon_p}+s\right)\,ds
		\nonumber\\
		&\overset{(\ast)}{=}& M\int_{-\sigma}^0  e^{Z(\sigma+s)}(\sigma+s)ds + o(1)
	\nonumber	\\
		&=&
		M\int^{\sigma}_{0} e^{Z(s)}sds + o(1)
\nonumber\\
		&\overset{(\ast\ast)}{=}&M( \theta-2)  + o(1),
	\end{eqnarray}
as $p\rightarrow +\infty$, where the passage to the limit in $(\ast)$ is again due to Lebesgue's theorem ( $(1+\frac{z_p}{p})\leq 1$ and the integration is on a bounded interval), while the equality in $(\ast\ast)$
is again a consequence of \eqref{Z} with \eqref{Zvari} and \eqref{Zvarii} in Proposition \ref{prop:convergenzaUltimaRiscalata}  (as already observed in \eqref{intzerosigmaZ}).
Putting together \eqref{inti1} with \eqref{inti2} one has
\begin{equation}\label{laderivata}
	p|(\apice{u}{m}_p)'(r_p)| (r_p)   =M( \theta-2)  + o(1)
\end{equation}
as $p\rightarrow +\infty$. On the other side, since by assumption Theorem \ref{theorem:mainDirichlet} holds for the solution $\apice{u}{m-1}_{\!\!\!p}$, one can use  Lemma \ref{lemma:legameCostantiPassiConsecutivi},   the convergence in \eqref{rmiinfty}  and exploit the expression of the known constant $\apice{D}{m-1}_{\!\!m-1}$ (see \eqref{Dk}) to obtain
\begin{eqnarray}
p|(\apice{u}{m}_{p})'(r_p)|(r_p)&\overset{\begin{subarray}{c}\text{Lemma } \ref{lemma:legameCostantiPassiConsecutivi}\\\eqref{rmiinfty}\end{subarray}}{=}&R^{-1}\apice{D}{m-1}_{\!\!m-1} +o(1),
\nonumber
\\
&\overset{\eqref{Dk}}{=}& R^{-1} \apice{M}{m-1}_{m-2}\left( \theta_{m-2}+2\right) +o(1),
\label{al}
\end{eqnarray} 
as $p\rightarrow +\infty.$  Identity \eqref{corollarioLemma4.3A} then follows from \eqref{al} and \eqref{laderivata}.
\end{proof}

\begin{lemma}
	\label{lemmaH}
Let $m\in\mathbb N$, $m\geq 2$. Assume that Theorem \ref{theorem:mainDirichlet} holds for radial solutions $\apice{u}{m-1}_{\!\!\!p}$ with $m-2$ interior zeros. 
Then
\begin{equation}\label{??}
I_p= \theta+2 + o(1),
\end{equation}
as $p\rightarrow +\infty$, where $I_p$  is defined in \eqref{defIp} and  $\theta\ (>2)$ is the constant in Proposition \ref{prop:convergenzaUltimaRiscalata}.
\end{lemma}

\begin{proof} Observe that, differently from the previous proofs, we cannot  pass to the limit into the integral that defines $I_p$, since now Lebesgue's Dominated Convergence Theorem does not apply. We follow similar ideas as in \cite{GrossiGrumiauPacella2}.\\\\
\emph{ STEP 1.  We show that $I_p$ is bounded.}
\\\\
Recalling the definition of $I_p$ and repeating the same computations as in \eqref{comeQui} one has that
\[(0\leq)\ I_p\overset{\eqref{comeQui}}{=}\frac{ p\int_{s_p}^1 	 |\apice{u}{m}_p(r)|^{p}r\,dr}{|\apice{u}{m}(s_p)|}\overset{\eqref{normaLInfinitoLimitataAlto}}{\leq} C p\int_0^1|\apice{u}{m}_p|^prdr\overset{\text{H\"older}}{\leq}C p\left(\int_0^1|\apice{u}{m}_p|^{p+1}rdr \right)^{\frac{p}{p+1}}\overset{\eqref{stimaenergiaupm}}{\leq} C.
\] 
	\\
\emph{STEP 2.  We show that \[\liminf_{p\rightarrow +\infty} I_p\geq  \theta+2.\]}

By the convergence of $z_p$ to $Z$ in Proposition \ref{prop:convergenzaUltimaRiscalata} and the one
 of $\frac{s_p}{\varepsilon_p}$ to $\sigma$ in  Proposition \ref{prop:BmiValePerOgnim}, using Fatou's lemma we obtain
	\begin{eqnarray*}\liminf_{p\rightarrow +\infty}I_p
		&=& \liminf_{p\rightarrow +\infty} \int_0^{b_p}\left|1+\frac{z_p}{p}  \right|^p\left( \frac{s_p}{\varepsilon_p}+r\right)dr
		\\
		&\overset{Fatou}{\geq}&
		\int_0^{+\infty} e^{Z(\sigma+r)}(\sigma+r)dr 
		\\
		&=&
		\int_{\sigma}^{+\infty} e^{Z(s)}sds
		\\
		&\overset{\eqref{Z}}{=}&2\theta^2\beta^{\theta}\int_{\sigma}^{+\infty}
		\frac{s^{\theta-1}}{(\beta^{\theta}+s^{\theta})^2}ds
		\\
		&=& \frac{2\theta\beta^{\theta}}{\beta^{\theta}+\sigma^{\theta}}
		\\
		&\overset{\eqref{Zvari},\eqref{Zvarii}}{=}& \theta+2.
	\end{eqnarray*}
\emph{STEP 3.  We prove that	
\begin{equation}
\label{UguaglianzaTESIStep2}
\frac{1}{4} \left[ I_p\right]^2 -  I_p \leq  \frac{(\theta-2)^2}{4}+( \theta-2)+o(1),
\end{equation}
as $p\rightarrow +\infty$.}\\
	
From  \eqref{LemmaD} and Lemma \ref{lemmaG}	
	\[
		M^2\left[ I_p\right]^2 +o(1)  \overset{\text{Lemma \ref{lemmaG}}}{=}  \left[p\int_{s_p}^1 	 |\apice{u}{m}_p(r)|^{p}r\,dr\right]^2\overset{\text{\eqref{LemmaD}}}{=} 
		p^2\left[(\apice{u}{m}_p)'(1)\right]^2,
	\]
as $p\rightarrow +\infty$, moreover, by Lemma \ref{lemmaP},
	\[\left[(\apice{u}{m}_p)'(1)\right]^2=\frac{4}{(p+1)}\int_0^1 |\apice{u}{m}_p(r)|^{p+1}r\, dr\]
	and so
	\begin{equation}\label{LegameIpConIntegralep+1}
\frac{1}{4}\left[ I_p\right]^2 =\frac{1}{M^2}\frac{p}{(p+1)}p\int_0^1 |\apice{u}{m}_p(r)|^{p+1}r\, dr+o(1)
	\end{equation}
as $p\rightarrow +\infty$.	We decompose the integral in \eqref{LegameIpConIntegralep+1} into the sum of three terms:	
\begin{eqnarray}\label{spezzoIntegralep+1}
p\int_0^1 |\apice{u}{m}_p(r)|^{p+1}r\, dr
&=& 
p\int_{0}^{r_p} |\apice{u}{m}_p|^{p+1}r\,dr
+
  p \int_{r_p}^{s_p}  |\apice{u}{m}_p|^{p+1}r\,dr  + p \int_{s_p}^{1}  |\apice{u}{m}_p|^{p+1}r\,dr
\end{eqnarray}		
For the first term  we make the change of variable  $r=s r_p$ and reduce it to the computation of the total energy for the  solution $\apice{u}{m-1}_{\!\!\!p}$, which can be done using Lemma \ref{lemmaP} and the inductive assumption that Theorem \ref{theorem:mainDirichlet} holds for the solution $\apice{u}{m-1}_{\!\!\!p}$. Finally, we have the dependence on $\theta$ and $M$ by exploiting Lemma \ref{Lemma4.3}, namely,
\begin{eqnarray}\label{parzialeEnergia1}
p\int_{0}^{r_p} |\apice{u}{m}_p|^{p+1}r\,dr
&=&(r_p)^2p\int_{0}^{1} |\apice{u}{m}_p(s r_p)|^{p+1}s\,ds
\nonumber\\
&\overset{\eqref{legame}}{=}& (r_p)^{-\frac{4}{p-1}}p\int_{0}^{1} |\apice{u}{m-1}_{\!\!\!p}(s)|^{p+1}s\,ds
\nonumber\\
&\overset{\eqref{rmiinfty} }{=}& R^{-2}p\int_{0}^{1}  |\apice{u}{m-1}_{\!\!\!p}(s)|^{p+1}s\,ds +o(1)
\nonumber\\
& \overset{\text{Lemma }\ref{lemmaP}}{=}&  R^{-2} \frac{1}{4}\left[p (\apice{u}{m-1}_{\!\!\! p})'(1)\right]^2+o(1)
\nonumber\\
& \overset{\text{Theorem \ref{theorem:mainDirichlet} for $\apice{u}{m-1}_{\!\!\!p}$}}{=}&\frac{1}{4}\frac{(\apice{D}{m-1}_{\!\! m-1})^2}{R^{2}}+o(1)
\nonumber\\
& \overset{\eqref{Dk}}{=}&\frac{1}{4}\frac{(\apice{M}{m-1}_{\!\! m-2})^2}{R^{2}}(\theta_{m-2}+2)^2+o(1)
\nonumber\\
&\overset{\text{Lemma }\ref{Lemma4.3}}{=}& M^2\frac{(\theta-2)^2}{4}+o(1),
\end{eqnarray}
as $p\rightarrow +\infty$.   We use  Lemma \ref{lemma:j_PENULTIMO_integr_p+1} and Lemma \ref{lemma:j_ULTIMO_integr_p+1} to estimate the last two terms in \eqref{spezzoIntegralep+1}, hence 
\begin{equation}\label{spezzoIntegralep+1CONTI_FATTI}
p\int_0^1 |\apice{u}{m}_p(r)|^{p+1}r\, dr
\leq M^2\left[\frac{(\theta-2)^2}{4}+( \theta-2)+I_p\right] +o(1),
\end{equation} 
as $p\rightarrow +\infty$. Substituting \eqref{spezzoIntegralep+1CONTI_FATTI} into  \eqref{LegameIpConIntegralep+1} we deduce \eqref{UguaglianzaTESIStep2}.
\\
\\
\emph{STEP 4. We show \eqref{??}.}  Note that
\begin{eqnarray}\label{functionIncreasing}
\frac{1}{4} \left[ I_p\right]^2 -  I_p 
&\overset{\text{ STEP 3}}{\leq} &\frac{(\theta-2)^2}{4}+( \theta-2)+o(1),
\nonumber\\
 &=& \frac{(\theta+2)^2}{4}-   ( \theta+2)    + o(1),
\end{eqnarray}
as $p\rightarrow +\infty$. Since $ I_p$ is bounded by \emph{STEP 1}, then, up to a subsequence, it converges to $I_{\infty}$, as $p\rightarrow +\infty$.  From \eqref{functionIncreasing} we then have that \[F( I_{\infty})\leq F(\theta+2),\]
where $F:\mathbb R\rightarrow \mathbb R$ is the function $F(x):=\frac{x^2}{4}-x$. Since $F$  is increasing for $x\geq 2$ and $I_{\infty}\geq \theta+2(\geq 2)$ by \emph{STEP 2}, it then follows that
\[I_{\infty}= \theta+2.\]
\end{proof}

\begin{definition} Let $m\in\mathbb N$, $m\geq 2$. If Theorem \ref{theorem:mainDirichlet} holds for the solution $\apice{u}{m-1}_{\!\!p}$ with $m-2$ interior zeros, then we can define 
\begin{equation}\label{deft}
t:=RM \ (>0),
\end{equation}
where $R\in (0,1]$ is the constant in \eqref{rmiinfty} and $M>0$ is the constant in \eqref{Mmiinfty}.\\ 
Then \eqref{corollarioLemma4.3A} can be written as
\begin{equation}
\label{remark4.5}
\theta=
\frac{\apice{M}{m-1}_{\!\!m-2} (\theta_{m-2}+2)}{t}+2,
\end{equation}
where  $\theta\ (>2)$ is the constant in \eqref{Zvarii}.
\end{definition}

\begin{proposition}\label{Proposition4.4}
Let $m\in\mathbb N$, $m\geq 2$. If Theorem \ref{theorem:mainDirichlet} holds for the solution $\apice{u}{m-1}_{\!\!p}$ with $m-2$ interior zeros.
Then the constant $t$  defined as in \eqref{deft} is the unique root of the equation
	\[\frac{\apice{M}{m-1}_{\!\!m-2}(\theta_{m-2}+2)}{2}\log x + x=0,\]
	namely
	\[t=  \frac{\apice{M}{m-1}_{\!\!m-2}\ (\theta_{m-2}+2)}{2} \mathcal L\left[ \frac{2}{\apice{M}{m-1}_{\!\!m-2}\ (\theta_{m-2}+2)}\right],
	\]
	where $\mathcal L$ is the Lambert function (i.e. the inverse function of $f(L)=Le^L$).\\
		
\end{proposition}
\begin{proof} It is enough to show the following equality
		\begin{equation}\label{lemma4.2}
		1=-\frac{ ( \theta-2)}{2} \log t,\end{equation}
the conclusion then follows by combining it with \eqref{remark4.5}.  To prove \eqref{lemma4.2}, observe that, for $s>\max\{-\frac{s_p}{\varepsilon_p},-\sigma\}$,
\begin{eqnarray}
		\label{logMenoLog}
		\log (s_p+\varepsilon_p s)-\log r_p&\overset{\text{Prop.} \ref{prop:BmiValePerOgnim}}{=}&
		\log (\sigma \varepsilon_p+\varepsilon_p s+o(1)\varepsilon_p)-\log r_p
		\nonumber
		\\
		&=& \log (\sigma +s)+\log \frac{\varepsilon_p}{r_p}+o(1)
		\nonumber
		\\
		&=& 
		\log (\sigma +s)-\frac{p-1}{2}\log \left(p^{\frac{1}{(p-1)}} |\apice{u}{m}_p(s_p)|(r_p)^{\frac{2}{(p-1)}}\right)+o(1)
		\nonumber
		\\
		&\overset{\begin{subarray}{c}\eqref{rmiinfty}\\\eqref{Mmiinfty}\end{subarray}}{=}& \log (\sigma +s)-\frac{p-1}{2}\left(\log t +o(1)\right) +o(1),
		\end{eqnarray}
as $p\rightarrow +\infty$. Moreover, by multiplying both sides of the equation $-\left[(\apice{u}{m}_p)'(r)r\right]'=|\apice{u}{m}_p(r)|^{p-1}\apice{u}{m}_p(r)r$ by  $(\log r-\log r_p)$ and integrating by parts, we have that
		\begin{eqnarray}
		\label{intmax}
		\int_{ r_p}^{ s_p}|\apice{u}{m}_p(r)|^{p}r (\log r-\log r_p) dr
		&=& (-1)^{m-1}
		\int_{ r_p}^{ s_p}|\apice{u}{m}_p(r)|^{p-1}\apice{u}{m}_p(r)r (\log r-\log r_p) dr
		\nonumber
		\\
		&=&(-1)^{m}\int_{ r_p}^{ s_p} 
		\left[(\apice{u}{m}_p)'(r)r\right]'(\log r-\log r_p)dr
		\nonumber
		\\
		&=& (-1)^{m}(\apice{u}{m}_p)'(s_p)s_p(\log s_p-\log r_p) + (-1)^{m-1}\int_{ r_p}^{ s_p} (\apice{u}{m}_p)'(r)r\frac{1}{r}dr
		\nonumber
		\\
		&=&  (-1)^{m-1}\left(\apice{u}{m}_p(s_p)-\apice{u}{m}_p(r_p)\right)= |\apice{u}{m}_p(s_p)|.
		\end{eqnarray}
In order to compute the integral on the left hand side of \eqref{intmax} we make the change of variable $r=s_p+\varepsilon_p t$ and recall the definition of $z_p$, to obtain that
	\begin{eqnarray*} 
	\int_{ r_p}^{s_p}|\apice{u}{m}_p(r)|^{p}r (\log r-\log r_p) dr
	&=&
	(\varepsilon_p)^2	\int_{a_p}^{0}
	|\apice{u}{m}_p(s_p+\varepsilon_p t)|^{p}\left(\frac{s_p}{\varepsilon_p}+t\right) \left(\log (s_p+\varepsilon_p t)-\log r_p\right) dt
	\\
	&=& \frac{|\apice{u}{m}_p(s_p)|}{p}    	\int_{a_p}^{0}    \left|1+\frac{z_p(t)}{p}  \right|^{p}\left(\frac{s_p}{\varepsilon_p}+t\right) \left(\log (s_p+\varepsilon_p t)-\log r_p\right) dt.
	\end{eqnarray*}
We then pass to the limit as $p\rightarrow +\infty$, using the convergence of $z_p$ to $Z$ (Proposition \ref{prop:convergenzaUltimaRiscalata}), the computation already obtained in \eqref{logMenoLog} and Lebesgue's convergence theorem:
	\begin{eqnarray*}	
		1&\overset{\eqref{intmax}}{=}&\frac{1}{p}    	\int_{a_p}^{0}    \left|1+\frac{z_p(t)}{p}  \right|^{p}\left(\frac{s_p}{\varepsilon_p}+t\right) \left(\log (s_p+\varepsilon_p t)-\log r_p\right) dt	\\
			&{=}&  
			\frac{1}{p} \int_{0}^{\sigma}  e^{Z(s)}s\log s\;ds
			-
			\frac{p-1}{2p} \int_{-\sigma}^0  e^{Z(\sigma+t)}(\sigma+t)
			\log\left( t +o(1)\right)dt + o(1)
			\\
			&=&
			-\frac{1}{2} \log t	\int^{\sigma}_{0} e^{Z(s)}sds + o(1) 
			\\
			&=&-\frac{( \theta-2)}{2} \log t   + o(1),
		\end{eqnarray*}
as $p\rightarrow +\infty$, where the last equality is a consequence of \eqref{Z} with \eqref{Zvari} and \eqref{Zvarii} in Proposition \ref{prop:convergenzaUltimaRiscalata}  (as already observed in \eqref{intzerosigmaZ}).
\end{proof}

\

\begin{proposition}\label{proposition:alphaok} Let $m\in\mathbb N^+$, $m\geq 2$. Assume that Theorem  \ref{theorem:mainDirichlet} holds for the solution $\apice{u}{m-1}_{\!\!p}$ with $m-2$ interior zeros and let $\theta\ (>2)$ be the constant in \eqref{Zvarii}. Then 
\begin{equation}\label{alphaok} 
\theta=\theta_{m-1},
\end{equation}
	where $\theta_{m-1}$ is the constant in  \eqref{iterateAlphaM}.
\end{proposition}
\begin{proof}
	Substituting the value of $t$ obtained in Proposition \ref{Proposition4.4} in \eqref{remark4.5}, we have that
	\begin{align*}
	 \theta&=
\frac{\apice{M}{m-1}_{\!\!m-2} (\theta_{m-2}+2)}{t}+2
=\frac{\apice{M}{m-1}_{\!\!m-2} (\theta_{m-2}+2)}{\frac{\apice{M}{m-1}_{\!\!m-2}\ (\theta_{m-2}+2)}{2} \mathcal L\left[ \frac{2}{\apice{M}{m-1}_{\!\!m-2}\ (\theta_{m-2}+2)}\right]}+2\\
&=\frac{2}{\mathcal L\left[ \frac{2}{(\theta_{m-2}+2)} e^{-{2}/(2+\theta_{m-2})} \right]}+2=\theta_{m-1}.
\end{align*}
\end{proof}

\begin{proof}[Proof of Proposition \ref{theorem:ultimaBubble_PASSO_INDUTTIVO}]
 Let $m\geq 2$ and we assume that Theorem \ref{theorem:mainDirichlet} holds for the radial solution $\apice{u}{m-1}_{\!\!\!p}$ with $m-2$ interior zeros. We want to show that Theorem \ref{th:ultimaBubble} holds, namely that \eqref{epsilon0Ultimo}, \eqref{repsilon0Ultimo}, \eqref{sepsilon0Ultimo} and \eqref{zpmenoNgUltimo} are satisfied.  Observe that
\eqref{epsilon0Ultimo} is \eqref{epsilonVaAZero}, while \eqref{repsilon0Ultimo} follows from Proposition \ref{prop:CmiValePerOgnim}.
By \eqref{alphaok} in Proposition \ref{proposition:alphaok} and \eqref{Zvarii}, $\sigma=\sigma_{m-1}$.  Then, by Proposition \ref{prop:BmiValePerOgnim}, \eqref{sepsilon0Ultimo} holds true and therefore, by \eqref{Z} and \eqref{Zvari}, we have that $Z=Z_{m-1}.$ But then the convergence in \eqref{convIntermediaz} implies \eqref{zpmenoNgUltimo}.
\end{proof}

\begin{proof}[Proof of Proposition \ref{theorem:mainDirichlet_PASSO_INDUTTIVO}]
By assumption, Theorem \ref{theorem:mainDirichlet} holds for the solution $\apice{u}{m-1}_{\!\!\!p}$ with $m-2$ interior zeros. We want to show that it holds for the solution $\apice{u}{m}_p$ with $m-1$ interior zeros, $m\geq 2$.
\[\]
\emph{STEP 1. We  prove the last convergence in  \eqref{raggio},  \eqref{derivataInRaggio},  \eqref{puntoCritico} and  \eqref{valoreInPuntoCritico}. Namely we show that:
	\begin{eqnarray}
	\label{raggioUltimo}
	&&\displaystyle\lim_{p\rightarrow +\infty}(r_p)^{\frac{2}{p-1}}=\apice{R}{m}_{m-1},  
	\\
	\label{derivataInRaggioUltimo}
	&&\displaystyle\lim_{p\rightarrow +\infty}p|(\apice{u}{m}_{p})'(1)|=\apice{D}{m}_m,
	\\
	\label{puntoCriticoUltimo}
	&&\displaystyle\lim_{p\rightarrow +\infty}(s_p)^{\frac{2}{p-1}}=\apice{S}{m}_{m-1}, 
	\\
	\label{valoreInPuntoCriticoUltimo}
	&&\displaystyle\lim_{p\rightarrow +\infty}|\apice{u}{m}_p(s_p)|=\apice{M}{m}_{m-1}, 
	\end{eqnarray}
	where $\apice{M}{m}_{m-1}$, $\apice{R}{m}_{m-1}$, $\apice{D}{m}_m$ and $\apice{S}{m}_{m-1}$ are the constants  in  Theorem \ref{theorem:mainDirichlet}.
}
\\\\
\emph{Proof of STEP 1.} Thanks to Lemma \ref{lemma:stimaDerivataInrp}, one can define
\begin{equation}
\label{Dmiinfty} 
D:=\lim_{p\rightarrow +\infty}p|(\apice{u}{m}_p)'(1)|\ (\geq 0);
\end{equation}	
moreover, since  by \eqref{raggioLimitatoBasso} $0< K_m\leq (s_p)^{\frac{2}{p-1}} (< 1)$, also
\begin{equation}
\label{smiinfty}
S:=\lim_{p\rightarrow +\infty} (s_p)^{\frac{2}{p-1}}\ (\in(0,1]).
\end{equation} Hence \eqref{raggioUltimo}, \eqref{derivataInRaggioUltimo}, \eqref{puntoCriticoUltimo}, \eqref{valoreInPuntoCriticoUltimo} are equivalent to prove that
		\begin{equation} \label{Mok} M=\apice{M}{m}_{m-1},
	\end{equation}
	\begin{equation} \label{Rok} R=\apice{R}{m}_{m-1},
	\end{equation}
	\begin{equation} \label{Dok} D=\apice{D}{m}_{m},
	\end{equation}
	\begin{equation} \label{Sok} S=\apice{S}{m}_{m-1},
	\end{equation}
where $M$ and $R$ have been defined in \eqref{Mmiinfty} and \eqref{rmiinfty}.
	Following similar ideas as in the proof of \eqref{logMenoLog} we obtain, for $s>0$,
	\begin{eqnarray}
	\label{logMenoLogSIM}
	\log (s_p+\varepsilon_p s)&\overset{\eqref{Bmi}}{=}&
	\log (\sigma_{m-1}\varepsilon_p+\varepsilon_ps+o(1)\varepsilon_p)
	\nonumber
	\\
	&=& \log (\sigma_{m-1}+s)+\frac{1}{2}\log (\varepsilon_p)^2+o(1)
	\nonumber
	\\
	&\overset{\eqref{defEpsilon}}{=}& 
	\log (\sigma_{m-1}+s)+\frac{1}{2}\log \frac{1}{p}-\frac{p-1}{2}\log |\apice{u}{m}_p(s_p)|+o(1).
		\end{eqnarray}
Then, as in the proof of \eqref{intmax}, we  multiply the equation
	$-\left[(\apice{u}{m}_p)'(r)r\right]'=|\apice{u}{m}_p(r)|^{p-1}\apice{u}{m}_p(r)r$ by $\log r$ and integrate by parts on the interval $(s_p,1)$ to obtain that
	\begin{eqnarray}\label{intmaxSIM}
	\int^{1}_{ s_p}|\apice{u}{m}_p(r)|^{p}r \log r \,dr
	&=& (-1)^{m-1}
	\int^{1}_{ s_p}|\apice{u}{m}_p(r)|^{p-1}\apice{u}{m}_p(r)r \log r \,dr
	\nonumber
	\\
	&=&(-1)^{m}\int^{1}_{ s_p} 
	\left[(\apice{u}{m}_p)'(r)r\right]'\log r\,dr
	\nonumber
	\\
	&=& (-1)^{m-1}(\apice{u}{m}_p)'(s_p)s_p\log s_{p} + (-1)^{m-1}\int^{1}_{s_p}(\apice{u}{m}_p)'(r)r\frac{1}{r}\,dr
	\nonumber
	\\
	&=&  (-1)^{m}\apice{u}{m}_p(s_p)=-|\apice{u}{m}_p(s_p)|.
	\end{eqnarray}
By the change of variable $r=s_p+\varepsilon_pt$ and recalling the definition of $z_p$ it follows that
	\begin{eqnarray}\label{st} 1&\overset{\eqref{intmaxSIM}}{=}&-\frac{1}{|\apice{u}{m}_p(s_p)|}
		\int^{1}_{ s_p}|\apice{u}{m}_p(r)|^{p}r \log r\, dr\nonumber\\
		&=&
		-\frac{(\varepsilon_p)^2}{|\apice{u}{m}_p(s_p)|}
			\int^{b_p}_{0}
		|\apice{u}{m}_p(s_p+\varepsilon_p t)|^{p}\left(\frac{s_p}{\varepsilon_p}+t\right) \log (s_p+\varepsilon_pt)\, dt
		\nonumber\\
		&=& -\frac{1}{p}    	\int^{b_p}_{0} \left|1+\frac{z_p(t)}{p}  \right|^{p}\left(\frac{s_p}{\varepsilon_p}+t\right) \log (s_p+\varepsilon_pt)\, dt
	\nonumber	\\
		&\overset{\eqref{logMenoLogSIM}}{=}& \frac{p-1}{2p}I_p\log |\apice{u}{m}_p(s_p)| + I_po(1) - \apice{J}{m}_p,
	\end{eqnarray}
where 
\[\apice{J}{m}_p:=\frac{1}{p}    	\int^{ b_p}_{0} \left|1+\frac{z_p(t)}{p}  \right|^{p}\left(\frac{s_p}{\varepsilon_p}+t\right)\log (\sigma_{m-1}+t)\, dt.\]
Following \cite[Lemma 4.8]{GrossiGrumiauPacella2} it is possible to show that $\apice{J}{m}_p=o(1)$ as $p\rightarrow +\infty$; moreover, $I_p=\theta_{m-1}+2+o(1)$ as $p\rightarrow +\infty$ by Lemma \ref{lemmaH} and \eqref{alphaok}.  Then, passing to the limit into \eqref{st},
\[1=\frac{\theta_{m-1}+2}{2}\log M, \]
namely, \[M=e^{\frac{2}{2+\theta_{m-1}}}\overset{\eqref{Mk}}{=}\apice{M}{m}_{m-1},\]
which ends the proof of \eqref{Mok}. Identity \eqref{Mok} implies \eqref{Rok}, because
\[R\overset{\begin{subarray}{l}\eqref{deft}\\\eqref{Mok}\end{subarray}}{=}\frac{t}{\apice{M}{m}_{m-1}}\overset{\begin{subarray}{l}\eqref{remark4.5}\\\eqref{alphaok}\end{subarray}}{=} \frac{\apice{M}{m-1}_{m-2} (\theta_{m-2}+2)}{\apice{M}{m}_{m-1}(\theta_{m-1} -2)}\overset{\eqref{Rk}}{=}\apice{R}{m}_{m-1}. \]
Next we prove \eqref{Dok}.  Observe that
\begin{eqnarray*}D&\overset{\eqref{Dmiinfty}}{=}&\lim_{p\rightarrow +\infty}p|(\apice{u}{m}_p)'(1)|
\\
	&\overset{\eqref{LemmaD}}{=}&
\lim_{p\rightarrow +\infty}p\int_{s_p}^1 	 |\apice{u}{m}_p(r)|^{p}r\,dr
\\
&\overset{\begin{subarray}{c}\text{Lemma } \ref{lemmaG}\\\eqref{Mok}\end{subarray}}{=}&
[\apice{M}{m}_{m-1}+o(1)]\lim_{p\rightarrow +\infty}I_p
\\
&\overset{\begin{subarray}{c}\text{Lemma }\ref{lemmaH}\\\eqref{alphaok}\end{subarray}}{=}& \apice{M}{m}_{m-1}(\theta_{m-1}+2)\overset{\eqref{Dk}}{=}\apice{D}{m}_{m}.
\end{eqnarray*}
The proof of \eqref{Sok}  also follows, since
\[S(\apice{M}{m}_{m-1})\overset{\eqref{Mok}}{=}\lim_{p\rightarrow +\infty} (s_p)^{\frac{2}{p-1}}|\apice{u}{m}_p(s_p)|
=\lim_{p\rightarrow +\infty} \left(\frac{s_p}{\varepsilon_p}\right)^{\frac{2}{p-1}}
\overset{\eqref{sepsilon0Ultimo} }{=}\lim_{p\rightarrow +\infty}\left(\sigma_{m-1}+o(1)\right)^{\frac{2}{p-1}}=1,
\]
so \[S=(\apice{M}{m}_{m-1})^{-1}\overset{\eqref{Sk}}{=}\apice{S}{m}_{m-1}.\]
\[\]
\emph{STEP 2. We prove all the other convergences in  \eqref{raggio}, \eqref{derivataInRaggio}, \eqref{puntoCritico}, \eqref{valoreInPuntoCritico}.}\\

\emph{Proof of STEP 2.}\\
The proof is obtained combining Lemma \ref{lemma:legameCostantiPassiConsecutivi} with the convergence of $(r_p)^{\frac{2}{p-1}}$ proved in \eqref{raggioUltimo} in STEP 1 and using also the explicit values  of the constants (in  \eqref{SoloUltimiReD} and \eqref{SoloUltimiSeM}) related to the solution $\apice{u}{m-1}_{\!\! p}$ with $m-2$ interior zeros, which are known since  Theorem \ref{theorem:mainDirichlet} holds for it  by assumption. We have that
\begin{eqnarray*}
\displaystyle\lim_{p\rightarrow +\infty} (\apice{r}{m}_{i,p})^{\frac{2}{p-1}}
&\overset{\text{Lemma } \ref{lemma:legameCostantiPassiConsecutivi}}{=}& \lim_{p\rightarrow +\infty }(r_p)^{\frac{2}{p-1}} \apice{R}{m-1}_{\!\!i}\overset{\eqref{raggioUltimo}}{=} \apice{R}{m}_{m-1}\apice{R}{m-1}_{\!\!i}\overset{\eqref{SoloUltimiReD}}{=}\apice{R}{m}_{m-1} \prod_{k=i+1}^{m-1}\!\!\apice{R}{k}_{k-1}
	\\
&=& \prod_{k=i+1}^{m}\!\!\apice{R}{k}_{k-1}\overset{\eqref{SoloUltimiReD}}{=} \apice{R}{m}_{i}, \qquad\quad   i=1,\ldots, m-2.  
\end{eqnarray*}
\begin{eqnarray*}
\displaystyle\lim_{p\rightarrow +\infty}p |(\apice{u}{m}_{p})'(\apice{r}{m}_{i,p})|(\apice{r}{m}_{i,p})&\overset{\text{Lemma } \ref{lemma:legameCostantiPassiConsecutivi}}{=}& \frac{ \apice{D}{m-1}_{\!\!i}}{\displaystyle\lim_{p\rightarrow +\infty }(r_p)^{\frac{2}{p-1}}}
\overset{\eqref{raggioUltimo}}{=} \displaystyle\frac{\apice{D}{m-1}_{\!\!i}}{\apice{R}{m}_{m-1}}\overset{\eqref{SoloUltimiReD}}{=}
 \frac{\apice{D}{i}_i}{	\apice{R}{m}_{m-1}\displaystyle\prod_{k=i+1}^{m-1}\!\!\apice{R}{k}_{k-1}}
 \\
 &=& \frac{\apice{D}{i}_i}{	\displaystyle\prod_{k=i+1}^{m}\!\!\apice{R}{k}_{k-1}}
\overset{\eqref{SoloUltimiReD}}{=}\apice{D}{m}_{i},\quad 	 \qquad i=1,\ldots, m-1.
\end{eqnarray*}
\begin{eqnarray*}
	\displaystyle\lim_{p\rightarrow +\infty} (\apice{s}{m}_{i,p})^{\frac{2}{p-1}}&\overset{\text{Lemma } \ref{lemma:legameCostantiPassiConsecutivi}}{=} &\lim_{p\rightarrow +\infty }(r_p)^{\frac{2}{p-1}} \apice{S}{m-1}_{\!\!i}\overset{\eqref{raggioUltimo}}{=}\apice{R}{m}_{m-1}\apice{S}{m-1}_{\!\!i}\overset{\eqref{SoloUltimiSeM}}{=}
	\apice{R}{m}_{m-1}\apice{S}{i+1}_{\!\!i} \prod_{k=i+2}^{m-1}\!\!\apice{R}{k}_{k-1}
	\\
	&=&\apice{S}{i+1}_{\!\!i} \prod_{k=i+2}^{m}\!\!\apice{R}{k}_{k-1}\overset{\eqref{SoloUltimiSeM}}{=}\apice{S}{m}_{i},\quad \qquad i=0,\ldots, m-2.
\end{eqnarray*}
\begin{eqnarray*}
	\displaystyle\lim_{p\rightarrow +\infty} |\apice{u}{m}_p(\apice{s}{m}_{i,p})|&\overset{\text{Lemma } \ref{lemma:legameCostantiPassiConsecutivi}}{=} &
\displaystyle\frac{\apice{M}{m-1}_{i}}{\displaystyle\lim_{p\rightarrow +\infty} (r_p)^{\frac{2}{p-1}}}\overset{\eqref{raggioUltimo}}{=}\frac{\apice{M}{m-1}_{i}}{\apice{R}{m}_{m-1}}
\overset{\eqref{SoloUltimiSeM}}{=}
	\frac{\apice{M}{i+1}_{i}}{\apice{R}{m}_{m-1}
	\displaystyle\prod_{k=i+2}^{m-1}\!\!\apice{R}{k}_{k-1}}\\
&=&
\frac{\apice{M}{i+1}_{i}}{
	\displaystyle\prod_{k=i+2}^{m}\!\!\apice{R}{k}_{k-1}}
\overset{\eqref{SoloUltimiSeM}}{=}
\apice{M}{m}_{i}, 
	\quad\qquad i=0,\ldots, m-2.
\end{eqnarray*}
\end{proof}

\subsection{Proof of Theorem \ref{theorem:mainDirichletG} -  case $\alpha=0$}
\label{section:conclusionDirichletLane_emden}
The proof is a consequence of  Theorem \ref{theorem:mainDirichlet}, Theorem \ref{th:ultimaBubble} and some of the  preliminary results in Section \ref{sectin:LabePreliminaries}.

\

We keep the notation of Section \ref{section:Lane}, namely $\apice{u}{m}_p=\apice{u}{m}_{0,p}$, $\apice{r}{m}_{i,p}=\apice{r}{m}_{i,0,p}$, $\apice{s}{m}_{i,p}=\apice{s}{m}_{i,0,p}$.\\
 \eqref{raggio},
\eqref{derivataInRaggio},
\eqref{puntoCritico} and
\eqref{valoreInPuntoCritico} in Theorem \ref{theorem:mainDirichlet} are exactly the case $\alpha=0$ of the sharp limits in \eqref{raggioG},
\eqref{derivataInRaggioG},
\eqref{puntoCriticoG} and
\eqref{valoreInPuntoCriticoG} in Theorem \ref{theorem:mainDirichletG}.
The proof of the case $\alpha=0$ in \eqref{EnergiaTotaleG}  follows instead from \eqref{derivataInRaggio}  in Theorem \ref{theorem:mainDirichlet}  combined with the preliminary Lemma \ref{lemmaP}, indeed
\[p\int_0^1 |\apice{u}{m}_p(r)|^{p+1}r\,dr
\overset{\text{Lemma \ref{lemmaP}}}{=} \frac{1}{4}\left[p(\apice{u}{m}_p)'(1)\right]^2 +o(1)\overset{\eqref{derivataInRaggio}}{=}\frac{(\apice{D}{m}_m)^2}{4}+o(1).\]
In order to conclude the proof of Theorem \ref{theorem:mainDirichletG}, we need to show the case $\alpha=0$ in the convergence in \eqref{pointwiseDirichletG}, namely that 
\begin{equation}
\label{pointwiseDirichletalpa0}
p\apice{u}{m}_{p}(x)=2\pi (-1)^{m-1}\apice{M}{m}_{m-1} (\theta_{m-1}+2) G(x,0) +o(1), \mbox{ in } C^1_{loc }(B\setminus\{0\}),
\end{equation}
as $p\rightarrow +\infty$. The proof of \eqref{pointwiseDirichletalpa0} exploits not only Theorem \ref{theorem:mainDirichlet} but also Theorem \ref{th:ultimaBubble} and some preliminary results in Section \ref{sectin:LabePreliminaries} (in particular the asymptotic  results in Lemma \ref{theorem: k=1EConvDeboleAZeroENEICOMPATTI}). We start from the representation formula and splitting the integral into $2$ terms:
	\begin{eqnarray*}p\apice{u}{m}_p(x)&=&p\int_B G(x,y)|\apice{u}{m}_p(y)|^{p-1}\apice{u}{m}_p(y)dy
	\\
	&=&p\int_{\{|y|<\apice{s}{m}_{m-1,p}\}} \!\!\!\!\!\!\!\!\!\!\!\!\!\!\!
	G(x,y)|\apice{u}{m}_p(y)|^{p-1}\apice{u}{m}_p(y)dy+p\int_{\{\apice{s}{m}_{m-1,p}<|y|<1\}} \!\!\!\!\!\!\!\!\!\!\!\!\!\!\!
	G(x,y)|\apice{u}{m}_p(y)|^{p-1}\apice{u}{m}_p(y)dy.
	\end{eqnarray*}	
Let $\delta\in(0,1)$. For the first term we use that for any $x,y$ such that $0<\delta\leq|x|\leq 1$ and $|y|<\apice{s}{m}_{m-1,p} $ ($<\frac{\delta}{2}$ since $\apice{s}{m}_{m-1,p}\rightarrow 0$  as $p\rightarrow +\infty$), 	
\begin{equation}\label{equaqui}\left| G(x,y)-G(x,0)\right|\leq |y|\sup_{\substack{ \delta\leq  |x|\leq 1\\ |y|\leq \frac{\delta}{2}}}|\nabla G(x,y)|=C|y|\leq C \apice{s}{m}_{m-1,p}   =o(1),\end{equation} (where the constant $C$ depends on $\delta$); moreover, exploiting  Lemma \ref{lemmaN} we also have that
\begin{equation}\label{enos}	p\int_{0}^{\apice{s}{m}_{m-1,p}}
|\apice{u}{m}_p(r)|^{p-1}\apice{u}{m}_p(r)rdr
\overset{\text{Lemma \ref{lemmaN}}}{=}	p(\apice{u}{m}_p)'(\apice{s}{m}_{m-1,p})\apice{s}{m}_{m-1,p}=0.\end{equation} 
Then	\begin{eqnarray}\label{pap}
	p\int_{\{|y|<\apice{s}{m}_{m-1,p}\}} \!\!\!\!\!\!\!\!\!\!\!\!\!\!\!
	G(x,y)|\apice{u}{m}_p(y)|^{p-1}\apice{u}{m}_p(y)dy&\overset{\eqref{equaqui}}{=}&
	(G(x,0)+o(1))p\int_{\{|y|<\apice{s}{m}_{m-1,p}\}} \!\!\!\!\!\!\!\!\!\!\!\!\!\!\!
	|\apice{u}{m}_p(y)|^{p-1}\apice{u}{m}_p(y)dy
	\nonumber\\
&=& 2\pi (G(x,0)+o(1))	p\int_{0}^{\apice{s}{m}_{m-1,p}}
|\apice{u}{m}_p(r)|^{p-1}\apice{u}{m}_p(r)rdr
\nonumber\\
&\overset{\eqref{enos}}{=}&0.
	\end{eqnarray}
	For the second term, we consider $\tau_p\in (\apice{s}{m}_{m-1,p},1)$ such that
	\[\tau_p=o(1)\qquad \text{ and }\qquad\frac{\tau_p-\apice{s}{m}_{m-1,p}}{\apice{\varepsilon}{m}_{m-1,p}}\rightarrow +\infty, \] 
 as $p\rightarrow +\infty$,  	and decompose the integral in the following way:
	\begin{eqnarray*}p\int_{\{\apice{s}{m}_{m-1,p}<|y|<1\}} \!\!\!\!\!\!\!\!\!\!\!\!\!\!\!
	G(x,y)|\apice{u}{m}_p(y)|^{p-1}\apice{u}{m}_p(y)dy&=& p\int_{\{\apice{s}{m}_{m-1,p}<|y|<\tau_p\}} \!\!\!\!\!\!\!\!\!\!\!\!\!\!\!
	G(x,y)|\apice{u}{m}_p(y)|^{p-1}\apice{u}{m}_p(y)dy\\&&+\,  p\int_{\{\tau_p<|y|<1\}} \!\!\!\!\!\!\!\!\!\!\!\!\!\!\!
	G(x,y)|\apice{u}{m}_p(y)|^{p-1}\apice{u}{m}_p(y)dy.
\end{eqnarray*}	
We show that 
\begin{equation}\label{rap}p\int_{\{\apice{s}{m}_{m-1,p}<|y|<\tau_p\}} \!\!\!\!\!\!\!\!\!\!\!\!\!\!\!
G(x,y)|\apice{u}{m}_p(y)|^{p-1}\apice{u}{m}_p(y)dy=(-1)^{m-1}2\pi G(x,0)\apice{M}{m}_{m-1} (\theta_{m-1}+2)+o(1)\end{equation}
and that
\begin{equation}\label{tap}p\int_{\{\tau_p<|y|<1\}} \!\!\!\!\!\!\!\!\!\!\!\!\!\!\!
G(x,y)|\apice{u}{m}_p(y)|^{p-1}\apice{u}{m}_p(y)dy=o(1),\end{equation}
uniformly in $\{\delta<|x|<1\}$, as $p\rightarrow +\infty$. The local uniform convergence of $pu_p$ in \eqref{pointwiseDirichletalpa0} then follows from \eqref{pap}, \eqref{rap} and \eqref{tap} (the uniform convergence of the derivative follows in a similar way, we omit it) and this will conclude the proof.\\	\\
Proof of \eqref{rap}:\\	
observe that for any $x,y$ such that $0<\delta\leq|x|\leq 1$ and $|y|<\tau_p $ ($<\frac{\delta}{2}$ since $\tau_p\rightarrow 0$,  as $p\rightarrow +\infty$), 	
\[\left| G(x,y)-G(x,0)\right|\leq |y|\sup_{\substack{ \delta\leq  |x|\leq 1\\ |y|\leq \frac{\delta}{2}}}|\nabla G(x,y)|=C|y|\leq C \tau_p   =o(1)\]
 as $p\rightarrow +\infty$, 
hence
	\begin{eqnarray*}
	p\int_{\{\apice{s}{m}_{m-1,p}<|y|<\tau_p\}} \!\!\!\!\!\!\!\!\!\!\!\!\!\!\!
	G(x,y)|\apice{u}{m}_p(y)|^{p-1}\apice{u}{m}_p(y)dy
		&{=}&
	(G(x,0)+o(1))p\int_{\{\apice{s}{m}_{m-1,p}<|y|<\tau_p\}} \!\!\!\!\!\!\!\!\!\!\!\!\!\!\!
	|\apice{u}{m}_p(y)|^{p-1}\apice{u}{m}_p(y)dy
		\\
	&=&2\pi (G(x,0)+o(1)) p\int_{\apice{s}{m}_{m-1,p}}^{\tau_p}
	|\apice{u}{m}_p(r)|^{p-1}\apice{u}{m}_p(r)rdr.
\end{eqnarray*}
We prove that 
\begin{equation}\label{enst}p\int_{\apice{s}{m}_{m-1,p}}^{\tau_p}
|\apice{u}{m}_p(r)|^{p-1}\apice{u}{m}_p(r)rdr=(-1)^{m-1}\apice{M}{m}_{m-1}(\theta_{m-1}+2) +o(1),\end{equation}  as $p\rightarrow +\infty$. 
On the one hand, since $\tau_p<1$,
\begin{eqnarray*}\limsup_p p\int_{\apice{s}{m}_{m-1,p}}^{\tau_p}
|\apice{u}{m}_p(r)|^{p}rdr&\leq& \limsup_pp\int_{\apice{s}{m}_{m-1,p}}^{1}
|\apice{u}{m}_p(r)|^{p}rdr
\\&\overset{\text{Lemma \ref{lemmaN}}}{=}&\limsup_p p|(\apice{u}{m}_p)'(1)|
\\
&\overset{\text{Thm \ref{theorem:mainDirichlet}}}{=}& \apice{D}{m}_m= \apice{M}{m}_{m-1}(\theta_{m-1}+2).
\end{eqnarray*}
On the other hand, since $\frac{\tau_p-\apice{s}{m}_{m-1,p}}{\apice{\varepsilon}{m}_{m-1,p}}\rightarrow +\infty$ as $p\rightarrow +\infty$, by making a change of variable $r=\apice{\varepsilon}{m}_{m-1,p} s+\apice{s}{m}_{m-1,p}$, recalling the definition of $z_p$, exploiting Theorem \ref{th:ultimaBubble} and using  Fatou's lemma (similarly as in the proof of STEP 2 in Lemma \ref{lemmaH}), we obtain that
\begin{eqnarray*}
\liminf_p p\int_{\apice{s}{m}_{m-1,p}}^{\tau_p}
|\apice{u}{m}_p(r)|^{p}rdr&=& \liminf_p|\apice{u}{m}_p(\apice{s}{m}_{m-1,p})|\int_{0}^{\frac{\tau_p-\apice{s}{m}_{m-1,p}}{\apice{\varepsilon}{m}_{m-1,p}}}  \left|1+\frac{z_p}{p}\right|^{p}\left(s+\frac{\apice{s}{m}_{m-1,p}}{\apice{\varepsilon}{m}_{m-1,p}}\right)ds 
\\
&\overset{\text{Thm \ref{th:ultimaBubble}}}{\geq}&\apice{M}{m}_{m-1}\int_0^{+\infty}
 e^{Z_{m-1}(s+\sigma_{m-1})}(s+\sigma_{m-1})ds  
\\
&=&\apice{M}{m}_{m-1}\int_{\sigma_{m-1}}^{+\infty}  e^{Z_{m-1}(s)}sds  
\\
&\overset{\eqref{espressioneZ}}{=}& \apice{M}{m}_{m-1} (\theta_{m-1})^2(\beta_{m-1})^{\theta_{m-1}}\int_{\sigma_{m-1}}^{+\infty}
\frac{s^{\theta_{m-1}-1}}{((\beta_{m-1})^{\theta_{m-1}}+s^{\theta_{m-1}})^2}ds
\\
&=& \apice{M}{m}_{m-1}\frac{2\theta_{m-1}(\beta_{m-1})^{\theta_{m-1}}}{(\beta_{m-1})^{\theta_{m-1}}+(\sigma_{m-1})^{\theta_{m-1}}}
\\
&=& \apice{M}{m}_{m-1} \left(\theta_{m-1}+2\right),
\end{eqnarray*}	
which ends the proof of \eqref{enst} and, therefore, of \eqref{rap}.\\
\\
Proof of  \eqref{tap}: 
\begin{eqnarray*}
p\int_{0}^1|\apice{u}{m}_p(r)|^{p-1}\apice{u}{m}_p(r) rdr&\overset{\text{Lemma } \ref{lemmaN}}{=}&
-p(\apice{u}{m}_p)'(1) \overset{\text{Thm \ref{theorem:mainDirichlet}}}{=} (-1)^{m-1}\apice{D}{m}_m  +o(1) \\
&=&(-1)^{m-1}\apice{M}{m}_{m-1}(\theta_{m-1}+2) +o(1)
\end{eqnarray*}
 as $p\rightarrow +\infty$ and so, by \eqref{enos} and \eqref{enst},
\begin{eqnarray*}p\int_{\tau_p}^1|\apice{u}{m}_p(r)|^{p-1}\apice{u}{m}_p(r) rdr&=&p\int_{0}^1|\apice{u}{m}_p(r)|^{p-1}\apice{u}{m}_p(r) rdr+\\
	&&-\, p\int_{0}^{\apice{s}{m}_{m-1,p}}|\apice{u}{m}_p(r)|^{p-1}\apice{u}{m}_p(r) rdr-p\int_{\apice{s}{m}_{m-1,p}}^{\tau_p}|\apice{u}{m}_p(r)|^{p-1}\apice{u}{m}_p(r) rdr
\\
&\overset{\begin{subarray}{l}\eqref{enos}\\\eqref{enst}\end{subarray}}{=}&o(1),
\end{eqnarray*}
 as $p\rightarrow +\infty$. Moreover, $|G(x,y)|\leq C$ for any $x,y$ such that $0<\delta\leq|x|\leq 1$ and $|y|<\frac{\delta}{2}$. As a consequence,
\begin{eqnarray}\label{ao}
\nonumber
	\left|p\int_{\{\tau_p<|y|<\frac{\delta}{2}\}} \!\!\!\!\!\!\!\!\!\!\!\!\!\!\!
	G(x,y)|\apice{u}{m}_p(y)|^{p-1}\apice{u}{m}_p(y)dy\right|
	&\leq&
	p\int_{\{\tau_p<|y|<\frac{\delta}{2}\}}\!\!\!\!\!\!\!\!\!\!\!\!\!\!\!
|G(x,y)||\apice{u}{m}_p(y)|^{p}dy\\
\nonumber&\leq &C p\int_{\{\tau_p<|y|<\frac{\delta}{2}\}}\!\!\!\!\!\!\!\!\!\!\!\!\!\!\!
|\apice{u}{m}_p(y)|^{p}dy\\
\nonumber&\leq& C p\int_{\{\tau_p<|y|<1\}}\!\!\!\!\!\!\!\!\!\!\!\!\!\!\!
|\apice{u}{m}_p(y)|^{p}dy
\\
&=& 2\pi(-1)^{m-1}C p\int_{\tau_p}^1
|\apice{u}{m}_{p}(r)|^{p-1}\apice{u}{m}_{p}(r)rdr\nonumber\\
&=&o(1).\end{eqnarray}
Finally, for the last term, we use \eqref{ZeroSuCompatti} in Lemma \ref{theorem: k=1EConvDeboleAZeroENEICOMPATTI}, namely,
\begin{eqnarray}\label{aoo}
\nonumber
\left|p\int_{\{\frac{\delta}{2}<|y|<1\}}	
 \!\!\!\!\!\!\!\!\!\!\!\!\!\!\!
G(x,y)|\apice{u}{m}_p(y)|^{p-1}\apice{u}{m}_p(y)dy\right|
&\leq& 
p\int_{\{\frac{\delta}{2}<|y|<1\}} \!\!\!\!\!\!\!\!\!\!\!\!\!\!\!
|G(x,y)||\apice{u}{m}_p(y)|^{p}dy
\\
\nonumber&\leq& \|p|\apice{u}{m}_p|^p\|_{L^{\infty}(\{\frac{\delta}{2}\leq|y|\leq 1\})}\int_{B}|G(x,y)|dy
\\
&= &C  \|p|\apice{u}{m}_p|^p\|_{L^{\infty}(\{\frac{\delta}{2}\leq|y|\leq 1\})}
\nonumber
\\
&\overset{\eqref{ZeroSuCompatti}}{=}&o(1)
\end{eqnarray}  as $p\rightarrow +\infty$. Identity \eqref{tap} follows from \eqref{ao} and \eqref{aoo}.

\subsection{Proof of  Theorem \ref{theorem:analisiAsintoticaRiscalateAlpha} - case $\alpha=0$}
\label{section:conclusionDirichletLane_emden2}
We keep the notation of Section \ref{section:Lane}, namely $\apice{u}{m}_p=\apice{u}{m}_{0,p}$, $\apice{r}{m}_{i,p}=\apice{r}{m}_{i,0,p}$, $\apice{s}{m}_{i,p}=\apice{s}{m}_{i,0,p}$; similarly we set  
\begin{equation}\label{defEpsilon}\apice{\varepsilon}{m}_{i,p}:=\apice{\varepsilon}{m}_{i,0,p},\quad i=0,\ldots, m-1,
\end{equation}
where $\apice{\varepsilon}{m}_{i,0,p}$ are the parameters defined in \eqref{scpar} in the case $\alpha=0$. Let us define the rescaled functions 
\begin{eqnarray}\label{zeta}
\apice{z}{m}_{i,p}(r):=
\frac{p}{|\apice{u}{m}_{p}(\apice{s}{m}_{i,p})|}\left[(-1)^i \apice{u}{m}_p(\apice{s}{m}_{i,p}+ \apice{\varepsilon}{m}_{i,p} r)  -    |\apice{u}{m}_{p}(\apice{s}{m}_{i,p})| \right], \quad r\in (\apice{a}{m}_{i,p},\apice{b}{m}_{i,p}),
\end{eqnarray}
for $i=0,\ldots, m-1,$ where  
\begin{equation}\label{cuore}\apice{a}{m}_{i,p}:=\left\{\begin{array}{lr}
0, &\qquad \mbox{ if }i=0, \\
\frac{\apice{r}{m}_{i,p}-\apice{s}{m}_{i,p}}{\apice{\varepsilon}{m}_{i,p}}\; (<0), &\qquad \mbox{ if }  i=1, \ldots, m-1,
\end{array}
\right.\end{equation}	and
\begin{equation}\label{doppiocuore}
\apice{b}{m}_{i,p}:=\frac{\apice{r}{m}_{i+1,p}-\apice{s}{m}_{i,p}}{ \apice{\varepsilon}{m}_{i,p}}\; (>0), \qquad  i=0, \ldots, m-1.\end{equation}

\begin{theorem}
	\label{theorem:analisiAsintoticaRiscalate}
	Let $m\in\mathbb N$, $m\geq 1$   then  $\apice{\varepsilon}{m}_{i,p}=o(1)$,  
	\begin{equation}\label{condizioniParam}
	\frac{\apice{r}{m}_{i,p}}{\apice{\varepsilon}{m}_{i,p}}=o(1)\ (i\neq 0),
	\qquad\qquad \frac{\apice{s}{m}_{i,p}}{\apice{\varepsilon}{m}_{i,p}}= \sigma_{i,0}+ o(1), \qquad\qquad \frac{\apice{\varepsilon}{m}_{i,p}}{\apice{r}{m}_{i+1,p}}=o(1),
\end{equation}
	and
	\begin{equation}	\label{zpmenoNg}
	\apice{z}{\emph m}_{i,p} = Z_{i,0}(\cdot +\sigma_{i,0})+o(1)\quad\mbox{in } C^1_{loc}(-\sigma_i, +\infty)\quad \text{as  $p\rightarrow +\infty$}
	\end{equation}
	 for all $i=0,\ldots, m-1$, where $\sigma_{i,0}$ and $Z_{i,0}$ are the constants and functions in the case $\alpha=0$ in Theorem \ref{theorem:analisiAsintoticaRiscalateAlpha}.
\end{theorem}
Before proving Theorem \ref{theorem:analisiAsintoticaRiscalate}, we show how the case $\alpha=0$ in  Theorem \ref{theorem:analisiAsintoticaRiscalateAlpha} is a consequence of it.

\begin{proof}[Proof of Theorem \ref{theorem:analisiAsintoticaRiscalateAlpha} - case $\alpha=0$]
If $i=0$ then $\sigma_{i,0}=0$ and $\apice{\xi}{ m}_{i,0,p}=\apice{z}{m}_{i,p}$ so there is nothing to prove. Let $i\geq 1$. We show the local uniform convergence, the convergence of the derivatives may be proved in a similar way.
Let $K\subset (0,+\infty)$ be a compact set, we want to show that
\begin{equation}
\label{tesisuK}
\sup_{r\in K} |\apice{\xi}{ m}_{i,0,p}(r)-Z_{i,0}(r)|=o(1),
\end{equation}
as $p\rightarrow +\infty$.
For any $c>0$ let 
\begin{align*}
K-c:=\{s\in\mathbb R\ :\ s+c\in K\}\qquad \text{ and }\qquad K_{c}:=\{s\in\mathbb R\ :\ \min_{y\in K}|s-y|\leq c\},  
\end{align*}
clearly $K-c$ and $K_c$ are compact sets. Observe that, by the second equality in \eqref{condizioniParam}, for $\delta>0$ there exists  $p_{\delta}>1$ such that 
\[K-\frac{\apice{s}{m}_{i,p}}{\apice{\varepsilon}{m}_{i,p}}\subseteq \widetilde K:=(K-\sigma_{i,0})_\delta, \qquad \mbox{ for }p\geq p_{\delta}.\]
Moreover, note that
	\[\apice{\xi}{ m}_{i,0,p}(r)=\apice{z}{m}_{i,p}(r-\frac{\apice{s}{m}_{i,p}}{\apice{\varepsilon}{m}_{i,p}}).\] 
Hence,
\begin{eqnarray*}
	\sup_{r\in K} |\apice{\xi}{ m}_{i,0,p}(r)-Z_{i,0}(r)|&=& \sup_{r\in K} \left|\apice{z}{m}_{i,p}(r-\frac{\apice{s}{m}_{i,p}}{\apice{\varepsilon}{m}_{i,p}})-Z_{i,0}(r)\right|
	\\
	&=& \sup_{t\in K-\frac{\apice{s}{m}_{i,p}}{\apice{\varepsilon}{m}_{i,p}}} \left|\apice{z}{m}_{i,p}(t)-Z_{i,0}(t+\frac{\apice{s}{m}_{i,p}}{\apice{\varepsilon}{m}_{i,p}})\right|
	\\
	&\leq& \sup_{t\in \widetilde K} \left|\apice{z}{m}_{i,p}(t)-Z_{i,0}(t+\frac{\apice{s}{m}_{i,p}}{\apice{\varepsilon}{m}_{i,p}})\right|
	\\
	&\leq & \sup_{t\in \widetilde K} \left|\apice{z}{m}_{i,p}(t)-Z_{i,0}(t+\sigma_{i,0})\right| + \sup_{t\in \widetilde K} \left|Z_{i,0}(t+\sigma_{i,0})-Z_{i,0}(t+\frac{\apice{s}{m}_{i,p}}{\apice{\varepsilon}{m}_{i,p}})\right|
	\\
	&\overset{
		\begin{subarray}{c}
	\eqref{zpmenoNg}	\\
	\eqref{condizioniParam}
\end{subarray}
}{=}& o(1)
\end{eqnarray*}
as $p\rightarrow +\infty$.
\end{proof}

\

\begin{proof}[Proof of Theorem \ref{theorem:analisiAsintoticaRiscalate}]

If $i=m-1$ then we directly apply Theorem \ref{th:ultimaBubble}. Hence, let us consider the case  $i<m-1$. Recall that the solutions $\apice{u}{m}_p$ and $\apice{u}{m-1}_{\!\!\!p}$ are linked by the following change of variable
\begin{equation}
\label{legamemucho}\apice{u}{m-1}_{\!\!\!p}(r)= (\apice{r}{\emph m}_{m-1,p})^{\frac{2}{p-1}} \apice{u}{m}_{p}(\apice{r}{\emph m}_{m-1,p} r), \qquad r\in [0,1].
\end{equation}
As a consequence, 
\[\frac{\apice{r}{m}_{i,p}}{\apice{\varepsilon}{m}_{i,p}} \overset{\eqref{QuintoLegame}}{=}
	\frac{\apice{r}{m-1}_{\!\!\!i,p}}{\apice{\varepsilon}{m-1}_{\!\!\!\!i,p}}, \qquad\quad
	\frac{\apice{s}{m}_{i,p}}{\apice{\varepsilon}{m}_{i,p}}
	\overset{\eqref{QuartoLegame}}{=}
	\frac{\apice{s}{m-1}_{\!\!\!i,p}}{\apice{\varepsilon}{m-1}_{\!\!\!i,p}}, \qquad \quad
	\frac{\apice{\varepsilon}{m}_{i,p}}{\apice{r}{m}_{i+1,p}}=
	\frac{\apice{\varepsilon}{m-1}_{\!\!\!i,p}}{\apice{r}{m-1}_{\!\!\!i+1,p}}.
	\]
So also	\[\apice{a}{m}_{i,p}= \apice{a}{m-1}_{\!\!\!i,p},\qquad\qquad\apice{b}{m}_{i,p}= \apice{b}{m-1}_{\!\!\!i,p},	\]
and the $i$-th rescaled function $\apice{z}{m}_{i,p}$ of the solution $\apice{u}{m}_p$ coincides with 
the $i$-th rescaled function $\apice{z}{m-1}_{\!\!\!i,p}$ of the solution $\apice{u}{m-1}_{\!\!\! p}$, because
\begin{eqnarray*}\apice{z}{m}_{i,p}(r)&=&
	p\left[ (-1)^i\frac{\apice{u}{m}_p( \apice{s}{m}_{i,p}+ \apice{\varepsilon}{m}_{i,p} r)}{|\apice{u}{m}_{p}(\apice{s}{m}_{i,p})|}  -    1 \right]\\
	&\overset{\eqref{legamemucho}}{=}&
	p\left[(-1)^i \frac{\apice{u}{m-1}_{\!\!\!p}\left(\frac{\apice{s}{m}_{i,p}}{\apice{r}{m}_{m-1,p}}+ \frac{\apice{\varepsilon}{m}_{i,p}}{\apice{r}{m}_{m-1,p}} r\right)}{(\apice{r}{m}_{m-1,p})^{2/(p-1)}|\apice{u}{m}_{p}(\apice{s}{m}_{i,p})|}  -    1 \right]
	\\ 
	&=& 	p\left[(-1)^i \frac{\apice{u}{m-1}_{\!\!\!p}(\apice{s}{m-1}_{\!\!\!i,p}+ \apice{\varepsilon}{m-1}_{\!\!\!i,p} r)}{|\apice{u}{m-1}_{\!\!p}(\apice{s}{m-1}_{\!\!i,p})|}  -    1 \right]=
	\apice{z}{m-1}_{\!\!\!i,p}(r).
\end{eqnarray*}
Iterating this procedure $m-i-1$ times, we have that the \emph{last} rescaled function of the solution $\apice{u}{i+1}_{\!\!\! p}$, namely,
\[\apice{z}{m}_{i,p}\underbrace{=\apice{z}{m-1}_{\!\!\!i,p}=\ldots=}_{m-i-1 \text{ times }}\apice{z}{i+1}_{\!\!\!i,p}\]
with
\[\apice{a}{m}_{i,p}= \apice{a}{m-1}_{\!\!\!i,p}=\ldots= \apice{a}{i+1}_{\!\!i,p};\qquad\qquad\apice{b}{m}_{i,p}= \apice{b}{m-1}_{\!\!\!i,p}=\ldots= \apice{b}{i+1}_{\!\!i,p};	\]
and similarly the relations
\begin{eqnarray*}
&& 
\frac{\apice{r}{m}_{i,p}}{\apice{\varepsilon}{m}_{i,p}} =
\frac{\apice{r}{m-1}_{\!\!\!i,p}}{\apice{\varepsilon}{m-1}_{\!\!\!\!i,p}}=\ldots=\frac{\apice{r}{i+1}_{\!\!i,p}}{\apice{\varepsilon}{i+1}_{\!\!i,p}},
\\
&&
\frac{\apice{s}{m}_{i,p}}{\apice{\varepsilon}{m}_{i,p}}
=
\frac{\apice{s}{m-1}_{\!\!\!i,p}}{\apice{\varepsilon}{m-1}_{\!\!\!i,p}}= \ldots=
\frac{\apice{s}{i+1}_{\!\!i,p}}{\apice{\varepsilon}{i+1}_{\!\!i,p}},
\\
&&
\frac{\apice{\varepsilon}{m}_{i,p}}{\apice{r}{m}_{i+1,p}}=
\frac{\apice{\varepsilon}{m-1}_{\!\!\!i,p}}{\apice{r}{m-1}_{\!\!\!i+1,p}}=\ldots=\apice{\varepsilon}{i+1}_{\!\!i,p}.
\end{eqnarray*}
The claim now follows from Theorem \ref{th:ultimaBubble} (with $m=i+1$).
\end{proof}

\

\section{The Dirichlet H\'enon problem in $B$}\label{section:Heno}
\addtocontents{toc}{\protect\setcounter{tocdepth}{1}}   
In this section we prove the case $\alpha>0$  in Theorems \ref{theorem:mainDirichletG} and \ref{theorem:analisiAsintoticaRiscalateAlpha}.\\
\\
The case $\alpha =0$ (Dirichlet Lane-Emden) has been proved in Section \ref{section:Lane}, the idea is to use a suitable \emph{change of variable} in order to derive the case $\alpha>0$ from the case $\alpha=0$.
\\
\\
To this aim let us observe that the Lane-Emden equation (case $\alpha=0$) and the H\'enon equation (case $\alpha>0$) are linked - \emph{in the radial setting and in dimension $2$} - by a change of variable first considered in \cite{Clem} and then also used in  \cite{CowGhou, GGN, SmetsSuWillem} (see also \cite{MoreiPacella, SmetsSuWillem} where it is exploited also in a non-radial setting).
	Indeed, setting
	\begin{equation}\label{changeOfVariableHenonDirichlet}v(r):=\left[\frac{2}{\alpha+2}\right]^{-\frac{2}{p-1}}u(r^\frac{\alpha +2 }{2}),\end{equation}
	then it is easy to see that $u$ is a solution of the Lane-Emden equation 
	\[
	-\Delta u= |u|^{p-1}u\qquad \mbox{ in }B
	\]
	if and only if $v$ is a solution of the H\'enon equation  
	\[
	-\Delta v=|x|^{\alpha} |v|^{p-1}v\qquad  \mbox{ or in }B ;
	\]
	moreover also the boundary conditions are preserved:
	\[
	u=0\quad\mbox{ on }\partial B \quad\mbox{ iff } \quad v=0\quad\mbox{ on }\partial B.
	\]
As a consequence, using the change of variable \eqref{changeOfVariableHenonDirichlet}, one immediately obtains the existence and uniqueness of the radial solutions with $m-1$ interior zeros for the Dirichlet H\'enon problems, since
\begin{equation}
\label{changeOfVariableHenonDirichletSection}
\apice{u}{m}_{\alpha,p}(r)=\left[\frac{2}{\alpha+2}\right]^{-\frac{2}{p-1}}\apice{u}{m}_{0,p}(r^\frac{\alpha +2 }{2}).
\end{equation}
We remark that \emph{this strategy cannot be used in higher dimension}.  From \eqref{changeOfVariableHenonDirichletSection} it follows that
\begin{equation}\label{rslink}\apice{r}{m}_{j,\alpha,p}= (\apice{r}{m}_{j,0,p})^{\frac{2}{\alpha +2}},\qquad\apice{s}{m}_{j,\alpha,p}= (\apice{s}{m}_{j,0,p})^{\frac{2}{\alpha +2}},\end{equation}
and then
\[|\apice{u}{m}_{\alpha,p}(\apice{s}{m}_{j,\alpha,p})|=\left[\frac{2}{\alpha+2}\right]^{-\frac{2}{p-1}}|\apice{u}{m}_{0,p}(\apice{s}{m}_{j,0,p})|.\]
Taking the derivative on both sides of \eqref{changeOfVariableHenonDirichletSection},
\[|(\apice{u}{m}_{\alpha,p})'(\apice{r}{m}_{j,\alpha,p})|\apice{r}{m}_{j,\alpha,p}=\frac{\alpha +2}{2}\left[\frac{2}{\alpha+2}\right]^{-\frac{2}{p-1}} |(\apice{u}{m}_{0,p})'(\apice{r}{m}_{j,0,p})| \apice{r}{m}_{j,0,p}.\]

\subsection{Proof of Theorem \ref{theorem:mainDirichletG} - case $\alpha>0$}

\eqref{raggioG}-\eqref{derivataInRaggioG}-\eqref{puntoCriticoG}-\eqref{valoreInPuntoCriticoG} follows passing to the limit as $p\rightarrow +\infty$ into the previous $4$ equalities  and using the corresponding convergences in the case $\alpha=0$.

\

Similarly one shows the convergence of the energy in \eqref{EnergiaTotaleG}, again using the change of variable \eqref{changeOfVariableHenonDirichletSection} and passing into the limit by exploiting the convergence already proved for the energy in the case $\alpha=0$, namely,
\begin{eqnarray*}
	p\int_0^1|\apice{u}{m}_{\alpha,p}(r)|^{p+1}r^{\alpha+1}dr&\overset{\eqref{changeOfVariableHenonDirichletSection}}{=}&   \left[\frac{2}{\alpha+2}\right]^{-\frac{2(p+1)}{p-1}}p\int_0^1|\apice{u}{m}_{0,p}(r^\frac{\alpha +2 }{2})|^{p+1}r^{\alpha +1} \ dr
\\
&=& \left[\frac{2}{\alpha +2} \right]^{-\frac{4}{p-1}}\frac{\alpha +2}{2} p\int_0^1|\apice{u}{m}_{0,p}(s)|^{p+1}    s\ ds
\\
&\overset{case\ \alpha=0}{=}&\frac{\alpha+2}{8}(\apice{M}{m}_{m-1})^2 (\theta_{m-1}+2)^2 +o(1).
\end{eqnarray*}

\

The convergence in \eqref{pointwiseDirichletG} also follows from \eqref{changeOfVariableHenonDirichletSection} and the corresponding $C^1_{loc}(B\setminus\{0\})$ convergence already proved for $\apice{u}{m}_{0,p}$ as $p\rightarrow +\infty$, because
\begin{eqnarray*}
p\apice{u}{m}_{\alpha,p}(|x|)&\overset{\eqref{changeOfVariableHenonDirichletSection}}{=}&  \left[\frac{2}{\alpha+2}\right]^{-\frac{2}{p-1}}p\apice{u}{m}_{0,p}(|x|^{\frac{\alpha +2 }{2}})
\\
&\overset{case\ \alpha=0}{=}& 2\pi\gamma_{0,m} G(|x|^{\frac{\alpha +2 }{2}},0) +o(1)
\\
&=&  \gamma_{0,m}\log |x|^{-\frac{\alpha +2 }{2}} +o(1)\\
&=& \frac{\alpha +2 }{2} \gamma_{0,m}\log |x|^{-1} +o(1)
\\
&=&2\pi\gamma_{\alpha,m} G(x,0) +o(1),\qquad \mbox{ in } C^1_{loc }(B\setminus\{0\}),
\end{eqnarray*}
with $\gamma_{\alpha,m}$ as in \eqref{gammaalm}.
\subsection{Proof of Theorem \ref{theorem:analisiAsintoticaRiscalateAlpha} - case $\alpha>0$}

An easy computation from \eqref{changeOfVariableHenonDirichletSection} gives
\begin{equation}\label{linkeps}\apice{\varepsilon}{m}_{i,\alpha,p}=(\apice{\varepsilon}{m}_{i,0,p})^{\frac{2}{\alpha+2}};\end{equation}
and hence,
\begin{equation}
\label{rescalHenonLane}
\apice{\xi}{m}_{i,\alpha,p}(r)=\apice{\xi}{m}_{i,0,p}(r^{\frac{\alpha+2}{2}}),\end{equation}
where $\apice{\varepsilon}{m}_{i,\alpha,p}$ and $\apice{\xi}{m}_{i,\alpha,p}$ are the scaling parameters and the rescaled solutions as defined in \eqref{scpar}, \eqref{scsol}. Using the result already proven in the case $\alpha=0$, from \eqref{linkeps} and \eqref{rslink} we then obtain \eqref{relaziParametri}, and from \eqref{rescalHenonLane} we obtain
\begin{align*}
\apice{\xi}{m}_{i,\alpha,p}(r)=
Z_{i,0}(r^{\frac{\alpha + 2}{2}})+o(1)\qquad \mbox{ in }C^1_{loc}(B\setminus\{0\})\quad \text{as $p\rightarrow +\infty$. } 
\end{align*}
The conclusion follows observing that $Z_{i,\alpha}(r)=Z_{i,0}(r^{\frac{\alpha + 2}{2}})$. Recalling that $Z_{i,0}$ solves
\[
\left\{
\begin{array}{lr}
-\Delta Z=e^Z+ 2\pi(2-\theta_i)\delta_{0}\quad\mbox{ in }\R^2,\\
Z(\sigma_{i,0})=0,\\
\int_{\R^2}e^Zdx=4\pi\theta_i,
\end{array}
\right.
\]
it is not difficult to check that $Z_{i,\alpha}$ solves \eqref{LiouvilleSingularEquationAlpha}.

\

\section{The Neumann problems and the equation in the whole $\mathbb R^2$} \label{section:Neumann&R2}
In this section we prove Theorems \ref{theoremNeumannMain} and \ref{theoremMainWhole}.
\subsection{The proof of Theorem \ref{theoremNeumannMain}}
Let us fix $m\in\mathbb N$, $m\geq 2$. The existence of a unique  (up to a sign)  radial solution $\apice{\bar u}{m}_{\alpha,p}$ for the Neumann problem  \eqref{equationHenon}-\eqref{NeumannBC} having $m$ nodal regions follows observing that the solutions $\apice{\bar u}{m}_{\alpha,p}$ of the Neumann problem  \eqref{equationHenon}-\eqref{NeumannBC} and the solution $\apice{u}{m}_{\alpha,p}$ of the  Dirichlet problem \eqref{equationHenon}-\eqref{DirichletBC} are linked via the following change of variable:
\begin{equation}\label{relazioneNeumann}\apice{\bar u}{m}_{\alpha,p}(r)=(s_{\alpha,p})^{\frac{\alpha+2}{p-1}}\apice{u}{m}_{\alpha,p}(s_{\alpha,p} r), \qquad r \in [0,1],
\end{equation}
where we have simplified the notation of  the last critical point of $\apice{u}{m}_{\alpha,p}$ setting \[s_{\alpha,p}:=\apice{s}{m}_{m-1,\alpha,p}.\]
The proof of  \eqref{raggioGN}-\eqref{derivataInRaggioGN}-\eqref{puntoCriticoGN}-\eqref{valoreInPuntoCriticoGN}, with the sharp constants in \eqref{raggioNeumann}, immediately follows  from \eqref{relazioneNeumann}. Indeed, we deduce the relations among the zeros and critical points of $\apice{\bar u}{m}_{\alpha,p}$ and $\apice{u}{m}_{\alpha,p}$ given by
\begin{eqnarray*}
\bar r_{i,\alpha,p}=\frac{r_{i,\alpha,p}}{s_{\alpha,p}}\quad \text{ and }\quad \bar s_{i,\alpha,p}=\frac{s_{i,\alpha,p}}{s_{\alpha,p}}.
\end{eqnarray*}
Here, in order to simplify the notation, we have dropped the dependence on $m$. Then 
\[|\apice{\bar u}{m}_{\alpha,p}(\bar s_{i,\alpha,p})|=(s_{\alpha,p})^{\frac{\alpha+2}{p-1}}|\apice{u}{m}_{\alpha,p}( s_{i,\alpha,p} )|.\]
Moreover, taking the derivative on both sides in \eqref{relazioneNeumann},
\[p|(\apice{\bar u}{m}_{\alpha,p})'(\bar r_{i,\alpha,p})|\bar r_{i,\alpha,p}=(s_{\alpha,p})^{\frac{\alpha+2}{p-1}}p
|(\apice{u}{m}_{\alpha,p})'( r_{i,\alpha,p} )|r_{i,\alpha,p}.\]
The claim now follows passing to the limit as $p\rightarrow +\infty$ into the previous equalities and applying \eqref{raggioG}-\eqref{derivataInRaggioG}-\eqref{puntoCriticoG}-\eqref{valoreInPuntoCriticoG}.
\\
In order to compute the energy and obtain \eqref{energiaNeumann} we use the relation \eqref{relazioneNeumann}, perform the change of variable $t= s_{\alpha,p} r$ and reduce it to the computation of  a partial energy of the Dirichlet solution $\apice{u}{m}_{\alpha, p}$, that is,
\begin{eqnarray}
\label{mare}
	p\int_0^1|(\apice{\bar u}{m}_{\alpha,p})'(r) |^2r\, dr&\overset{\eqref{equationHenon}-\eqref{NeumannBC}}{=}&
	p\int_0^1 |\apice{\bar u}{m}_{\alpha,p}(r)|^{p+1}r^{1+\alpha}\, dr
	\nonumber\\
	 &\overset{\eqref{relazioneNeumann}}{=} & (s_{\alpha,p})^{\frac{(\alpha+2)(p+1)}{p-1}}p\int_0^1|\apice{u}{m}_{\alpha,p}(s_{\alpha,p} r)|^{p+1}r^{1+\alpha}\, dr
	\nonumber\\
	&=&
	(s_{\alpha,p})^{\frac{2(\alpha +2)}{p-1}}p\int_0^{s_{\alpha,p}}|\apice{u}{m}_{\alpha,p}(t)|^{p+1}t^{1+\alpha}\, dt.
\end{eqnarray}
Using the change of variable \eqref{changeOfVariableHenonDirichletSection} we can then reduce this last integral for the solution $\apice{u}{m}_{\alpha,p}$ of the Dirichlet H\'enon problem ($\alpha>0$) to an integral for the solution $\apice{u}{m}_{0,p}$ of the Dirichlet Lane-Emden problem ($\alpha=0$), namely,
\begin{eqnarray*}p\int_0^{s_{\alpha,p}}|\apice{u}{m}_{\alpha,p}(t)|^{p+1}t^{1+\alpha}\, dt
	&\overset{\eqref{changeOfVariableHenonDirichletSection}}{=}& 
	\left(\frac{\alpha +2}{2}\right)^{\frac{2(p+1)}{p-1}}
		p\int_0^{s_{\alpha,p}}|\apice{u}{m}_{0,p}(t^{\frac{\alpha+2}{2}})|^{p+1}t^{1+\alpha}\,dt
		\\
		&=&
\left(\frac{\alpha +2}{2}\right)^{\frac{p+3}{p-1}}
		p\int_0^{s_{p}}|\apice{u}{m}_{0,p}(r)|^{p+1}r\,dr,
\end{eqnarray*}
where $s_p:=s_{0,p}$.
Finally, we compute this integral using the results of Section \ref{section:Lane}. More precisely, by \eqref{parzialeEnergia1}, \eqref{alphaok} and \eqref{Mok},
\[
p\int_{0}^{r_{p}} |\apice{u}{m}_{0,p}|^{p+1}r\,dr= (\apice{M}{m}_{m-1})^2\frac{(\theta_{m-1}-2)^2}{4}+o(1),
\]
and by Lemma \ref{lemma:j_PENULTIMO_integr_p+1}, \eqref{alphaok} and \eqref{Mok} we obtain that
\[p\int_{r_{p}}^{s_{p}} |\apice{u}{m}_{0,p}|^{p+1}r\,dr=  (\apice{M}{m}_{m-1})^2( \theta_{m-1}-2) 	+ o(1),\]
so that
\begin{eqnarray}\label{sole}
p\int_{0}^{s_{\alpha,p}} |\apice{u}{m}_{\alpha,p}|^{p+1}t^{1+\alpha}\,dt= \frac{\alpha +2}{8}  (\apice{M}{m}_{m-1})^2( \theta_{m-1}-2) (\theta_{m-1}+2)	+ o(1).
\end{eqnarray}
Collecting \eqref{mare} and \eqref{sole}, and recalling that, by \eqref{puntoCriticoG},
$ (s_{\alpha,p})^{\frac{2(\alpha+2)}{p-1}}=(\apice{S}{m}_{m-1})^{2} +o(1)$
as $p\rightarrow +\infty$, we obtain \eqref{energiaNeumann}.

\subsection{The proof of Theorem \ref{theoremMainWhole}}

A radial solution $w_{\alpha,p}$ of \eqref{problemHenonWhole} such that $w_{\alpha,p}(0)=1$ exists by ODE considerations. Moreover it oscillates infinitely many times and has a unique local maximum or minimum $\delta_{m,\alpha,p}$ between any two consecutive zeros $\rho_{m,\alpha,p}$ and $\rho_{m+1,\alpha,p}$ (see \cite[p. 294]{N} for the case $\alpha=0$ and use the change of variable \eqref{changeOfVariableHenonDirichlet} to obtain the same result for any $\alpha>0$).

It is easy to see that  \begin{equation}\label{legameconR2}\apice{u}{m}_{\alpha,p}(r)=(\rho_{m,\alpha,p})^{\frac{\alpha+2}{p-1}}w_{\alpha,p}(\rho_{m,\alpha,p}r), \qquad r\in[0,1],\end{equation}
is a radial solution of the Dirichlet problem \eqref{equationHenon}-\eqref{DirichletBC} with $m-1$ interior zeros.
The conclusion then follows as a corollary of \eqref{raggioG}-\eqref{derivataInRaggioG}-\eqref{puntoCriticoG}-\eqref{valoreInPuntoCriticoG} in  Theorem \ref{theorem:mainDirichletG}, observing that, from \eqref{legameconR2},
\[(\rho_{m,\alpha,p})^{\frac{\alpha+2}{p-1}}=   \apice{u}{m}_{\alpha,p}(0)
\quad \text{ and }\quad 
\delta_{m-1,\alpha,p}=\rho_{m,\alpha,p}\apice{s}{m}_{m-1,\alpha,p},\]
and then 
\[w_{\alpha,p}(\delta_{m-1,\alpha,p})=(\rho_{m,\alpha,p})^{-\frac{\alpha+2}{p-1}}\apice{u}{m}_{\alpha,p}(\apice{s}{m}_{m-1,\alpha,p}).\]
Taking the derivative on both sides of \eqref{legameconR2}, it follows that
\[p|(w_{\alpha,p})'(\rho_{m,\alpha,p})|\rho_{m,\alpha,p}=(\rho_{m,\alpha,p})^{-\frac{\alpha+2}{p-1}}p |(\apice{u}{m}_{\alpha,p})'(1)|.\]

\section{Analysis of the constants}\label{Constants:sec}
\addtocontents{toc}{\protect\setcounter{tocdepth}{2}} 
\newcommand{\Pl}{{\mathcal L}}

 In this section we do a careful study of the constants in Definition \ref{definition:Constants2} and show Theorem \ref{theorem:NObounds}.  The first result in this section characterizes the growth of $\theta_k$.
 \begin{theorem}[Linear growth of $\theta_k$]\label{theta:t} Let $\theta_k$ be the numbers in Definition \ref{definition:Constants}. It holds that 
 	\begin{align*}
 	2+8k < \theta_k < 2+\frac{2}{\Pl\left( \frac{1}{4k} \right)} < 4+8k\qquad \text{ for all }k\in\N.
 	\end{align*}
 	In particular, the sequence  $(\theta_k)_k$ is strictly increasing and 
 	\begin{align}\label{theta:ip}
 	\left[ \frac{\theta_k}{2} \right] = 4k+1\qquad \text{ for all }k\in\N,
 	\end{align}
 	where $[x]$ denotes the integer part of $x$.
 \end{theorem}
Using Theorem \ref{theta:t} we can deduce information on all the constants in Definition \ref{definition:Constants2}. 
 \begin{theorem}[Sublinear growth of $\apice{M}{m}_{0}$]\label{a:t}
 	Let $c_1:=\sqrt{\pi}\approx 1.77245$ and $c_2:= 6\frac{\Gamma(\frac{3}{4})}{\Gamma(\frac{1}{4})}\approx 2.028$. Then, for every $m\in \N$, 
 	\begin{align}\label{stiM0}
 	c_1\ \frac{\Gamma(m+1)}{\Gamma(m+\frac{1}{2})}e^{\frac{1}{3+4m}} <\apice{M}{m+1}_{\!\!0}<
 	c_2 \ \frac{\Gamma(m+\frac{5}{4})}{\Gamma(m+\frac{3}{4})}e^{\frac{1}{2+4m}},
 	\end{align}
 	where $\Gamma$ is the usual Gamma function. Furthermore, 
 		\begin{align}\label{asym}
 		c_1\leq \liminf_{m\to\infty}\frac{\apice{M}{m+1}_{\!\!0}}{\sqrt{m}}\leq\limsup_{m\to\infty}\frac{\apice{M}{m+1}_{\!\!0}}{\sqrt{m}}\leq c_2.
 		\end{align}
 \end{theorem}
\begin{corollary} \label{thm:growthOtherConstants} For every $i\in\mathbb N$,
\begin{align}\label{allOtherMax}
&0<\liminf_{m\rightarrow\infty} \frac{\apice{M}{m}_i}{\sqrt m}\leq\limsup_{m\rightarrow\infty} \frac{\apice{M}{m}_i}{\sqrt m}<\infty,
\qquad 0<\liminf_{m\rightarrow\infty} \frac{\apice{D}{m}_i}{\sqrt m}\leq\limsup_{m\rightarrow\infty} \frac{\apice{D}{m}_i}{\sqrt m}<\infty,
\\
&0<\liminf_{m\rightarrow\infty} \apice{S}{m}_i \sqrt m\leq \limsup_{m\rightarrow\infty} \apice{S}{m}_i \sqrt m<\infty,
\qquad
0<\liminf_{m\rightarrow\infty} \apice{R}{m}_i \sqrt m\leq \limsup_{m\rightarrow\infty}\apice{R}{m}_i \sqrt m<\infty.
\end{align}
\end{corollary}

 To show Theorem \ref{theta:t}, we begin with an alternative formula for $\theta_k.$ Let
 \begin{align*}
 \phi:(0,\infty)\to (0,\infty)\quad \text{ be given by }\quad \phi(s):=\frac{s}{1+2s}e^{-\frac{s}{1+2s}}
 \end{align*}
 and let
 \begin{align*}
 A_1:=\Pl(\frac{1}{2}e^{-\frac{1}{2}}),\qquad A_{k+1}:=\Pl(\phi(A_{k}))\quad \text{ for }k\in\N.	
 \end{align*}
 
 \begin{lemma}\label{Ak:l} For every $m\in \N$,
 	\begin{align}\label{eq}
 	\theta_m=2+\frac{2}{A_m}\qquad .
 	\end{align}
 \end{lemma}
 \begin{proof} We argue by induction. The inductive base is clear since 
 	$\theta_1=2+\frac{2}{A_1}$ by definition.  Let $k\in\N$ and assume that $\theta_k=2+\frac{2}{A_k}$.  Then
 	\begin{align*}
 	\frac{2}{ e^{{2}/(2+\theta_{k})}\ (\theta_{k}+2)}
 	=\frac{2}{ e^{{2}/(4+\frac{2}{A_k})}\ (4+\frac{2}{A_k})}
 	=\frac{A_k}{2A_k+1}e^{-\frac{A_k}{2A_k+1}},
 	\end{align*}
 	and therefore
 	\begin{align*}
 	\theta_{k+1}=2+\displaystyle\frac{2}{\mathcal L\left[\displaystyle \frac{2}{ e^{{2}/(2+\theta_{k})}\ (\theta_{k}+2)}\right]}
 	=2+\displaystyle\frac{2}{\mathcal L\left[
 		\frac{A_k}{2A_k+1}e^{-\frac{A_k}{2A_k+1}}
 		\right]}
 	=2+\displaystyle\frac{2}{A_{k+1}}.
 	\end{align*}
 \end{proof}
 
 Next we deduce some bounds for $A_k$.  We need first an auxiliary lemma. Recall the following well-known properties of the $\Pl$ function:  For every $s>0$, $s\mapsto\Pl(s)$ is monotone increasing, $\Pl(0)=0$, $\Pl'(s)=\frac{\Pl(s)}{s (\Pl(s)+1)}$, and 
 \begin{align}\label{Pl:bds}
 s-s^2<\Pl(s)<s\qquad \text{ for $s>0$.}
 \end{align}
 
 \begin{lemma}\label{aux:l}
 	For every $k\in\N$,
 	\begin{align}
 	\phi\Big(\frac{1}{4k}\Big)=\frac{1}{4k+2}e^{-\frac{1}{4k+2}}&< \frac{1}{4k+4}e^{ \frac{1}{4k+4}}=\Pl^{-1}\Big(\frac{1}{4(k+1)}\Big),\label{cc1}\\
 	\phi\Big(\Pl\Big(\frac{1}{4k}\Big)\Big)&> \frac{1}{4(k+1)}.\label{cc2}
 	\end{align}
 \end{lemma}
 \begin{proof}
 	Let $f:(0,\infty)\to (0,\infty)$ be given by $f(s)=\frac{s+\frac{1}{2}}{s+1}-\ln\left( \Big(1+\frac{1}{s}\Big)^s\right).$  Observe that $f$ is convex (because $f''(s)=\frac{1}{s(1+s)^3}>0$) and that $\lim_{s\to\infty} f(s)=0$; as a consequence, $f>0$ in $(0,\infty)$, and therefore $\ln\left(1+\frac{1}{s} \right)<\frac{s+\frac{1}{2}}{s(s+1)}$, or, equivalently,
 	\begin{align*}
 	1+\frac{1}{s}< e^{\frac{s+\frac{1}{2}}{s(s+1)}}
 	=e^{\frac{1}{2}(\frac{1}{1+s}+\frac{1}{s})}.
 	\end{align*}
 	If $k\in\N$ and $s=2k+1$, then
 	\begin{align*}
 	\frac{4k+4}{4k+2}=1+\frac{1}{2k+1}<e^{\frac{1}{2}(\frac{1}{2k+2}+\frac{1}{2k+1})},
 	\end{align*}
 	and \eqref{cc1} follows. Next we show inequality \eqref{cc2}. For $s>0$, let 
 	\begin{align*}
 	h(s)=\frac{\phi(\Pl(s))}{\frac{s}{1+4s}},\quad \text{ then }\quad h'(s)=\frac{4 e^{-\frac{\Pl(s)}{1+2\Pl(s)}} \Pl(s) \left(s-\Pl(s)-\Pl(s)^2\right)}{s^2 (2 \Pl(s)+1)^3}.
 	\end{align*} 
 	From \eqref{Pl:bds} we deduce that $h(0)=1$ and $h'(s)>0$ for $s\in(0,1)$, therefore $h>1$ in $(0,1)$. Indeed, let $G(s):=\frac{s}{\Pl(s)+\Pl(s)^2}$, then $G(0)=1$ (by \eqref{Pl:bds}) and $G'(s)=\frac{\Pl(s)}{(\Pl(s)+1)^3}>0$ for $s\in(0,1)$, as a consequence $G>1$ in $(0,1)$ and thus $s-\Pl(s)-\Pl(s)^2>0$ for $s\in(0,1)$.   Then, $h(\frac{1}{4k})>1$ for $k\in\N$ and thus $\phi(\Pl(\frac{1}{4k}))> \frac{1}{4(k+1)}$ as claimed in \eqref{cc2}.
 \end{proof}
 
 \begin{proposition}\label{b:l}
 	It holds that
 	\begin{align*}
 	\frac{1}{4k+1}<\Pl\Big( \frac{1}{4k} \Big) < A_k < \frac{1}{4k}\qquad \text{ for all }k\in\N.
 	\end{align*}
 \end{proposition}
 \begin{proof}
 	We show first that $A_k < \frac{1}{4k}$ for all $k\in\N.$ The inductive base is immediate, since  $0.2388\approx A_1 < \frac{1}{4}$.  Let $k\in\N$ and assume that $A_k<\frac{1}{4k}$. Using that $t\mapsto \Pl(\phi(t))$ is a monotone increasing function, it follows, by Lemma \ref{aux:l}, that
 	\begin{align*}
 	A_{k+1}=\Pl(\phi(A_k))< \Pl\left(\phi\Big(\frac{1}{4k}\Big)\right)< \frac{1}{4(k+1)}.
 	\end{align*}
 	
 	Next, we show that $A_k>\Pl(\frac{1}{4k})$ for all $k\in\N$. Note that $A_1>\Pl(\frac{1}{4})\approx 0.2038$ and if $A_k>\Pl(\frac{1}{4k})$ then, using that $t\mapsto \Pl(\phi(t))$ and $t\mapsto \Pl(t)$ are monotone increasing functions and Lemma \ref{aux:l},
 	\begin{align*}
 	A_{k+1}=\Pl(\phi(A_k))>\Pl\left(\phi\Big(\Pl\Big(\frac{1}{4k}\Big)\Big)\right)>\Pl\Big(\frac{1}{4(k+1)}\Big).
 	\end{align*}
 	Finally, for $s\in(0,1)$, let $g(s):=\frac{\Pl(s)}{\frac{s}{s+1}}$; then, by \eqref{Pl:bds},
 	\begin{align*}
 	g'(s)=\frac{(s-\Pl(s)) \Pl(s)}{s^2 (\Pl(s)+1)}>0\quad \text{ in }(0,1)\qquad \text{ and }\qquad g(0)=1.
 	\end{align*}
 	Therefore, $g>1$ in $(0,1)$ and, using $s=\frac{1}{4k}$ with $k\in\N$, we deduce that
 	$\Pl(\frac{1}{4k})>\frac{1}{4k+1}.$
 \end{proof}
	
\medskip
We are now ready to prove Theorem \ref{theta:t}.
 \begin{proof}[Proof of Theorem \ref{theta:t}] 
	Let $k\in\N$. By Proposition \ref{b:l} we know that $4k<\frac{1}{A_k}<4k+1$, then, by Lemma \ref{Ak:l}, $2+8k < \theta_k < 2+\frac{2}{\Pl\left( \frac{1}{4k} \right)} < 4+8k$ and $[\frac{\theta_k}{2}]=1+[\frac{1}{A_k}]=1+4k$.
\end{proof}

The proof of Theorem \ref{a:t} relies on a formula for $\apice{M}{m+1}_{\!\!0}$ in terms of $\theta_k$'s (see Lemma \ref{lemma:M0Theta} below), on the estimates for $\theta_k$ (see Theorem \ref{theta:t}), and on some well-known properties of the $\Gamma$-function (see \eqref{G:p:1}--\eqref{p:l}).

\begin{lemma} \label{lemma:M0Theta} Let $m\in\N$, then
 \[\apice{M}{m+1}_{\!\!0}=\frac{(\theta_{m}-2)}{4}e^{\frac{2}{2+\theta_{m}}}
 \displaystyle\prod_{k=1}^{m-1}\frac
 {\theta_{k}-2}
 {\theta_{k}+2}
 	\]
 \end{lemma}
\begin{proof} By \eqref{SoloUltimiSeM},
 	\begin{align*}
 	\apice{M}{m+1}_{\!\!0} &= \frac{\apice{M}{1}_{\! 0}}{\displaystyle\prod_{k=2}^{m+1}\!\apice{R}{k}_{k-1}}
 	=\frac{e^{{2}/(2+\theta_{0})}}{\displaystyle\prod_{k=2}^{m+1}\!\frac{\apice{M}{k-1}_{\!\!k-2}\ (\theta_{k-2}+2)}{\apice{M}{k}_{k-1}(\theta_{k-1}-2)}}
 	=e^{{2}/(2+\theta_{0})}\displaystyle\prod_{k=2}^{m+1}\frac
 	{\apice{M}{k}_{k-1}(\theta_{k-1}-2)}
 	{\apice{M}{k-1}_{\!\!k-2}\ (\theta_{k-2}+2)}\\
 	&=e^{\frac{1}{2}}\displaystyle\prod_{k=2}^{m+1}\frac
 	{\theta_{k-1}-2}
 	{\theta_{k-2}+2}
 	\displaystyle\prod_{k=2}^{m+1}\frac
 	{\apice{M}{k}_{k-1}}
 	{\apice{M}{k-1}_{\!\!k-2}}
 	=e^{\frac{1}{2}}\displaystyle\prod_{k=2}^{m+1}\frac
 	{\theta_{k-1}-2}
 	{\theta_{k-2}+2}
 	\displaystyle\prod_{k=2}^{m+1}\frac
 	{e^{{2}/(2+\theta_{k-1})}}
 	{e^{{2}/(2+\theta_{k-2})}}\\
 	&=e^{\frac{1}{2}+\sum\limits_{k=2}^{m+1} \frac{2}{2+\theta_{k-1}}-\frac{2}{2+\theta_{k-2}}}\frac
 	{\displaystyle\prod_{k=2}^{m+1}(\theta_{k-1}-2)}
 	{\displaystyle\prod_{k=2}^{m+1}(\theta_{k-2}+2)}
 	=
 \frac{(\theta_{m}-2)}{4}e^{\frac{2}{2+\theta_{m}}}
 	\displaystyle\prod_{k=1}^{m-1}\frac
 	{\theta_{k}-2}
 	{\theta_{k}+2}.
 	\end{align*}
 \end{proof}
 Recall that the Gamma function is the unique function that simultaneously satisfies that $\Gamma(1)=1$,
 \begin{align}
 z\Gamma(z)&=\Gamma(z+1),\label{G:p:1}\\
 \lim_{n\to\infty}\frac{\Gamma(n+z)}{\Gamma(n)n^z}&=1\qquad \text{ for every }z\in\mathbb C.\label{G:p:2}
 \end{align}
 In particular, by  \eqref{G:p:1}, we have, for every $\alpha\geq 0$ and $\beta\geq 0$, that
 	\begin{align}
 	\prod_{k=1}^{m-1} 
 	\frac{k+\beta}{k+\alpha} = \frac{\Gamma(\alpha+1)}{\Gamma(\beta+1)}
 	\frac{\Gamma(m+\beta)}{\Gamma(m+\alpha)}.\label{p:l}
 	\end{align}
 
\medskip
 
 We are now ready to show Theorem \ref{a:t}.

 \begin{proof}[Proof of Theorem \ref{a:t}] 
By Theorem \ref{theta:t},
 	\[e^{\frac{1}{3+4m}}<e^{\frac{2}{2+\theta_{m}}}<e^{\frac{1}{2+4m}} \quad\text{ and }\quad  2m<\frac{\theta_m-2}{4}<2m+\frac{1}{2}\quad \text{ for $m\in\mathbb N$}\]
 	and, using that $s\mapsto \frac{s-2}{s+2}$ is a monotone increasing function, 
 	\[\frac{k}{k+\frac{1}{2}}<\frac{\theta_k-2}{\theta_k+2}<\frac{k+\frac{1}{4}}{k+\frac{3}{4}}\quad\text{ for $k\in\N$}.\]
 	Then, by Lemma \ref{lemma:M0Theta},
\[2m\,e^{\frac{1}{3+4m}}\displaystyle\prod_{k=1}^{m-1}\frac{k}{k+\frac{1}{2}}
<\apice{M}{m+1}_{\!\!0}
=\frac{(\theta_{m}-2)}{4}e^{\frac{2}{2+\theta_{m}}}
 \displaystyle\prod_{k=1}^{m-1}\frac
 {\theta_{k}-2}
 {\theta_{k}+2}
<(2m+\frac{1}{2})\,e^{\frac{1}{2+4m}}\displaystyle\prod_{k=1}^{m-1}\frac{k+\frac{1}{4}}{k+\frac{3}{4}},\]
and, by \eqref{p:l},
\begin{equation}\label{comp}2\frac{m\Gamma(m)}{\Gamma(m+\frac{1}{2})}\Gamma(1+\frac{1}{2})
e^{\frac{1}{3+4m}} <\apice{M}{m+1}_{\!\!0}<(2m+\frac{1}{2})\frac{\Gamma(m+\frac{1}{4})}{\Gamma(m+\frac{3}{4})}\frac{\Gamma(1+\frac{3}{4})}{\Gamma(1+\frac{1}{4})}e^{\frac{1}{2+4m}}.\end{equation}
Using \eqref{G:p:1} and the fact that $\Gamma(\frac{1}{2})=\sqrt{\pi}$,
 	\begin{align}
& 2\frac{m\Gamma(m)}{\Gamma(m+\frac{1}{2})}\Gamma(1+\frac{1}{2}) =\sqrt{\pi}\ \frac{m\Gamma(m)}{\Gamma(m+\frac{1}{2})} =
 	\sqrt{\pi}\ \frac{\Gamma(m+1)}{\Gamma(m+\frac{1}{2})},
\label{bcost1} 	\\
 	&(2m+\frac{1}{2})\frac{\Gamma(m+\frac{1}{4})}{\Gamma(m+\frac{3}{4})}\frac{\Gamma(1+\frac{3}{4})}{\Gamma(1+\frac{1}{4})}= 	3\frac{\Gamma(\frac{3}{4})}{\Gamma(\frac{1}{4})}\ (2m+\frac{1}{2})\frac{\Gamma(m+\frac{1}{4})}{\Gamma(m+\frac{3}{4})}
 	=6\frac{\Gamma(\frac{3}{4})}{\Gamma(\frac{1}{4})}\frac{\Gamma(m+\frac{5}{4})}{\Gamma(m+\frac{3}{4})}.\label{bcost2}
 	\end{align}
Substituting \eqref{bcost1} and \eqref{bcost2} into  \eqref{comp}, we obtain \eqref{stiM0}.  Finally, from \eqref{stiM0} and \eqref{G:p:2} one has \eqref{asym}, because
 	\begin{align*}
  \liminf_{m\to \infty}\frac{\apice{M}{m+1}_{\!\!0}}{m^{\frac{1}{2}}}&\geq \lim_{m\to\infty}\sqrt{\pi}\ \frac{\Gamma(m+1)}{m^{\frac{1}{2}}\Gamma(m+\frac{1}{2})}e^{\frac{1}{3+4m}} 
 	= \sqrt{\pi},\\
 	\limsup_{m\to \infty}\frac{\apice{M}{m+1}_{\!\!0}}{m^{\frac{1}{2}}} &\leq \lim_{m\to\infty}6\frac{\Gamma(\frac{3}{4})}{\Gamma(\frac{1}{4})}\frac{\Gamma(m+\frac{5}{4})}{\Gamma(m+\frac{3}{4})}e^{\frac{1}{2+4m}}
 	=6\frac{\Gamma(\frac{3}{4})}{\Gamma(\frac{1}{4})}.
 	\end{align*}
 \end{proof}

 \begin{proof} [Proof of Corollary \ref{thm:growthOtherConstants}]
 This is a direct consequence of Theorem \ref{a:t} and the fact that, for every $i\in \N$, \begin{align*}
 \frac{\apice{M}{m+1}_{\!\!0}}{\apice{M}{m+1}_{\!\!i}}=\frac{\apice{M}{1}_0}{\apice{M}{i+1}_{\!i}}\prod_{k=2}^{i+1}\apice{R}{k}_{k-1},\qquad 
 \frac{\apice{M}{m+1}_{\!\!0}}{\apice{D}{m+1}_{\!\!i}}&=
 \frac{\apice{M}{1}_0}{\apice{D}{i}_{i}}\prod_{k=2}^{i}\apice{R}{k}_{k-1}, \qquad
 \apice{M}{m+1}_{\!\!0} \apice{S}{m+1}_{\!\!i}= \apice{M}{1}_0 \apice{S}{i+1}_{\! i} (\prod_{k=2}^{i+1}\apice{R}{k}_{k-1})^{-1},\\
   \apice{M}{m+1}_{\!\!0} \apice{R}{m+1}_{\!\!i}&= \apice{M}{1}_0 (\prod_{k=2}^{i}\apice{R}{k}_{k-1})^{-1},
 \end{align*} where all the right-hand sides are positive  constants independent of $m$.
 \end{proof}

 \medskip
 
 We conclude the section with the proofs of Theorem \ref{theorem:NObounds} and  Corollary \ref{coro}.
 \begin{proof}[Proof of Theorem \ref{theorem:NObounds}]
 	By Theorem \ref{theorem:mainDirichletG}, we have that $|\apice{u}{m}_{\alpha,p}(\apice{s}{m}_{0,\alpha,p})|=\apice{M}{m}_{0}+o(1)$ as $p\to\infty$, and then \eqref{D:bd}, \eqref{D:asym} follow from Theorem \ref{a:t}. Furthermore, by Theorem \ref{theoremNeumannMain}, $|\apice{\bar u}{m}_{\alpha,p}(\apice{\bar s}{m}_{0,\alpha,p})|=\apice{\bar M}{m}_{0} +o(1)=\apice{S}{m}_{m-1} \apice{ M}{m}_0+o(1)$. Then \eqref{N:bd} follows from Theorem \ref{a:t} and the fact that
 	\begin{align*}
 	e^{-\frac{1}{4m-2}}<\apice{S}{m}_{m-1}=e^{-\frac{A_{m-1}}{2A_{m-1}+1}}<e^{-\frac{1}{4m-1}}.
 	\end{align*}
 	Finally, \eqref{N:asym} holds by Theorem \ref{a:t} and because \begin{equation}\label{smZero}\lim_{m\to\infty} \apice{S}{m}_{m-1}=\lim_{m\to\infty} e^{-\frac{2}{2+\theta_{m-1}}} =1.\end{equation} 
 \end{proof}
 
 \begin{proof}[Proof of Corollary \ref{coro}] By Theorems \ref{theorem:mainDirichletG} and \ref{theoremNeumannMain},  $|\apice{u}{m}_{\alpha,p}(\apice{s}{m}_{i,\alpha,p})|=\apice{M}{m}_{i}+o(1)$
 	and  $|\apice{\bar u}{m}_{\alpha,p}(\apice{\bar s}{m}_{i,\alpha,p})|=\apice{\bar M}{m}_{i} +o(1)=\apice{S}{m}_{m-1} \apice{ M}{m}_i+o(1)$ as $p\to\infty$ and the claim follows from \eqref{smZero} and  \eqref{allOtherMax}. 
 \end{proof}
 
   \begin{remark}
 	An efficient way of computing numerically the values of the $\theta_k$'s is via a for-cycle; for instance, in the computing software Mathematica, for a given number $N1$, the code 
 	\begin{verbatim}
 	 T[0] = N[2]
 	 psi[t_] := 2/(2 + t)*Exp[-2/(2 + t)]
 	 For[k = 1, k < N1, k++, T[k] = 2 + 2 / ProductLog[psi[T[k - 1]]]]
 	\end{verbatim}
computes N1 values of $\theta_k=T[k]$ (c.f. \eqref{definition:Constants}). Using $N1=10^6$ and  Lemma \ref{lemma:M0Theta}, we see that $\lim\limits_{m\to\infty}\frac{\apice{M}{m+1}_{0}}{\sqrt{m}}\approx 1.82774.$ Indeed, numerically it seems that the sequence $(\frac{\apice{M}{m+1}_{0}}{\sqrt{m}})_{m\in\N}$ is monotone, and in particular it has a limit as $m\to \infty$.  This would imply that Theorem \ref{theorem:NObounds} and Corollary \ref{coro} can be written with $\lim$ instead of $\limsup$ and $\liminf$.  However, a rigorous proof of this monotonicity seems to require sharper bounds that the ones given by Theorem \ref{a:t}.
 \end{remark}

\begin{center}
\begin{minipage}{.9\textwidth}
\begin{multicols}{2}	

	{\color{white}.}

	\begin{tabular}{||c|c|c|c||}
		\hline 
		$m$ & $\theta_m$ & $\apice{M}{m}_{0}$ & $\apice{M}{m}_{0}/\sqrt{m}$ \\
		 \hline\hline 
		1 & 10.374 & 1.6487 & 1.6487 \\
		2 & 18.4277 & 2.46075 & 1.74001 \\
		3 & 26.4493 & 3.06521 & 1.7697 \\
		4 & 34.4609 & 3.56876 & 1.78438 \\
		5 & 42.4682 & 4.00957 & 1.79313 \\
		6 & 50.4731 & 4.40651 & 1.79895 \\
		7 & 58.4767 & 4.77053 & 1.80309 \\
		8 & 66.4795 & 5.10867 & 1.80619 \\
		9 & 74.4816 & 5.42579 & 1.8086 \\
		10 & 82.4833 & 5.72537 & 1.81052 \\
		11 & 90.4848 & 6.01003 & 1.81209 \\
		12 & 98.486 & 6.28181 & 1.8134 \\
		13 & 106.487 & 6.5423 & 1.81451 \\
		14 & 114.488 & 6.79282 & 1.81546 \\
		15 & 122.489 & 7.03442 & 1.81628 \\
		16 & 130.489 & 7.26799 & 1.817 \\
		17 & 138.49 & 7.49429 & 1.81763 \\
		18 & 146.49 & 7.71395 & 1.81819 \\
		19 & 154.491 & 7.92752 & 1.8187 \\
		20 & 162.491 & 8.13549 & 1.81915 \\
		21 & 170.492 & 8.33828 & 1.81956 \\
		22 & 178.492 & 8.53625 & 1.81993 \\
		23 & 186.492 & 8.72973 & 1.82027 \\
		24 & 194.493 & 8.91901 & 1.82059 \\
		25 & 202.493 & 9.10436 & 1.82087 \\
		\hline
\end{tabular}

\setlength{\unitlength}{1cm}
\thicklines
\begin{picture}(5,4)
\put(5.8,.4){$m$}
\put(0.5,3.5){$\theta_m$}
\includegraphics[width=.4\textwidth]{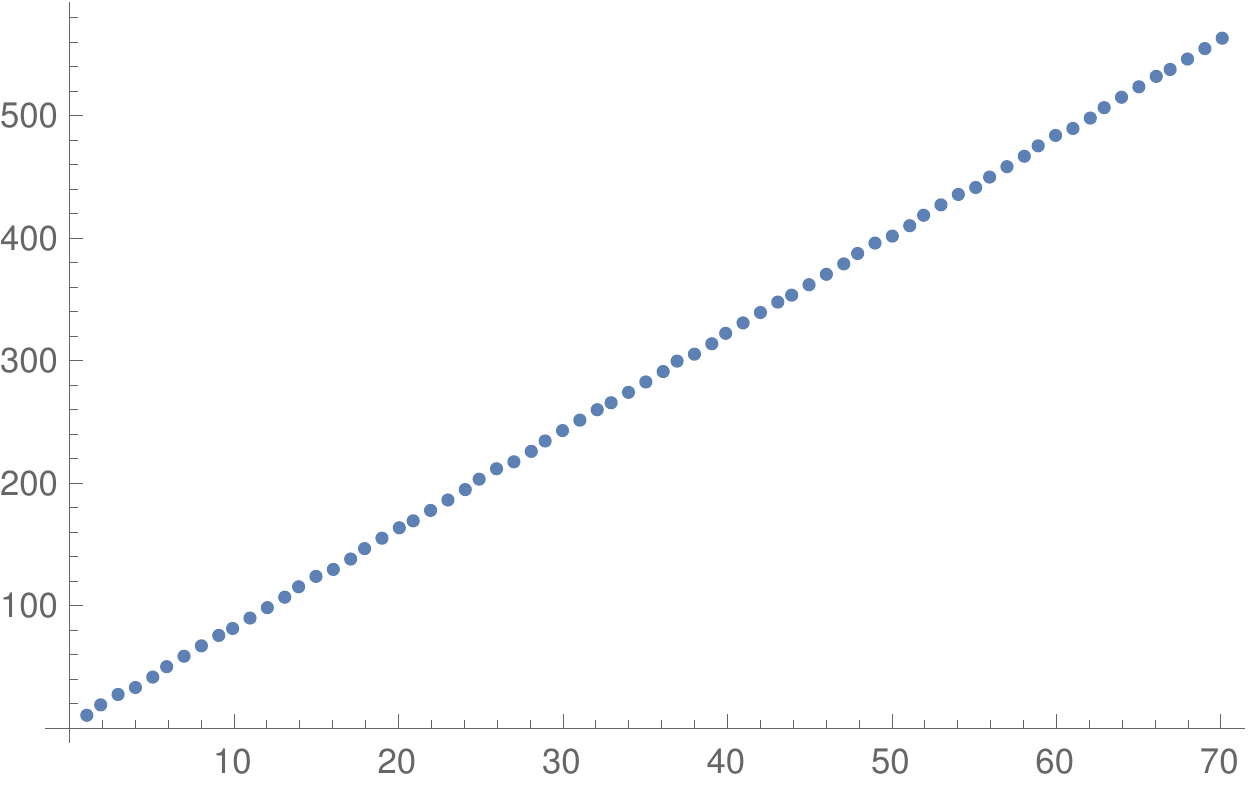}
\end{picture}

\setlength{\unitlength}{1cm}
\thicklines
\begin{picture}(5,4)
	\put(5.8,.4){$m$}
	\put(0.5,3.5){$\apice{M}{m}_{0}$}
\includegraphics[width=.4\textwidth]{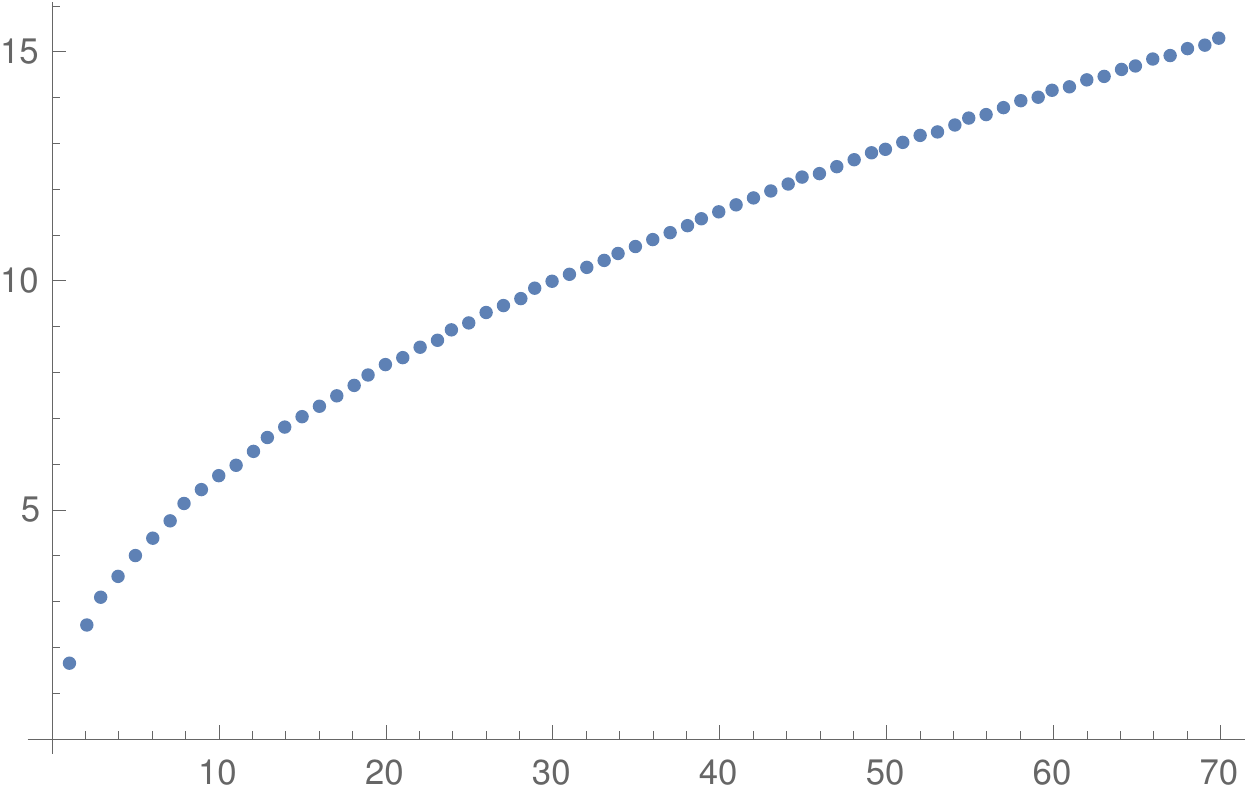}
\end{picture}

\setlength{\unitlength}{1cm}
\thicklines
\begin{picture}(5,4)
	\put(5.8,.4){$m$}
	\put(0.6,3.5){$\apice{M}{m}_{0}/\sqrt{m}$}
\includegraphics[width=.4\textwidth]{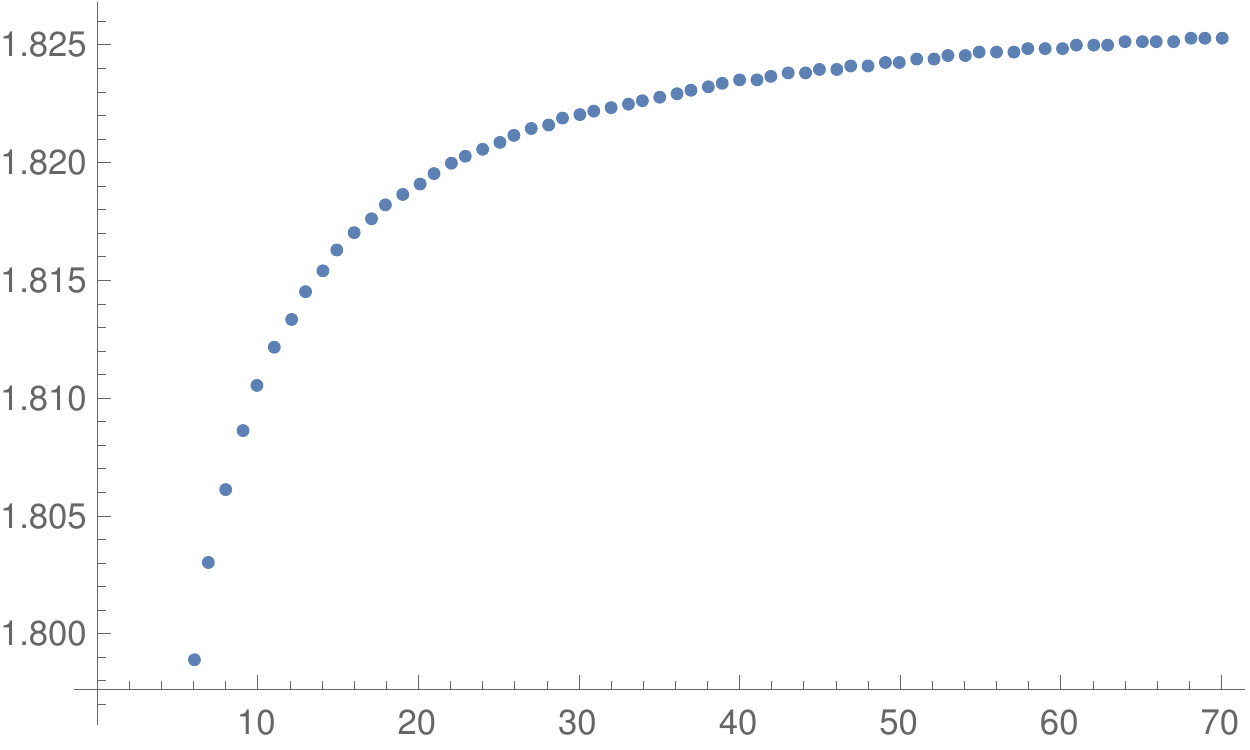}
\end{picture}

\end{multicols}
\end{minipage}
\captionof{figure}{Some numerical values and plots of $\theta_m$, $\apice{M}{m}_{0}$, and $\apice{M}{m}_{0}/\sqrt{m}$.}
	\label{f3}
\end{center}

We close this section with the proof of Lemma \ref{lemma:PropertiesConstants}. 
\begin{proof}[Proof of Lemma \ref{lemma:PropertiesConstants}]
	To see \eqref{catenateointro} and \eqref{catenaMintro}, observe that $\apice{M}{m}_{m-1}= e^{\frac{2}{2+\theta_{m-1}}} >1,$ and as a consequence, $\apice{S}{m}_{m-1}= (\apice{M}{m}_{m-1})^{-1}<1.$ Then, $\frac{\apice{S}{m}_{i}}{\apice{R}{m}_{i+1}}=\apice{S}{i+1}_{i}<1$. Hence, $\apice{S}{m}_{i}<\apice{R}{m}_{i+1}$ for $i=1,\ldots, m-2.$	Furthermore, since $\mathcal L\left(\frac{y}{e^y}\right)<\frac{y}{e^y}$ we have (setting $y:=\frac{2}{\theta_{m-2}+2}$) that $e^{\frac{2}{\theta_{m-2}+2}}\frac{(\theta_{m-2}+2)}{2}<\frac{\theta_{m-1}-2}{2}$, and then, by \eqref{Rk} and \eqref{Sk}, we deduce that $\apice{R}{m}_{m-1}<\apice{S}{m}_{m-1}.$  As a consequence, $\apice{R}{m}_{i}<\apice{S}{m}_{i}$ for $i=1,\ldots, m-1$ and \eqref{catenateointro}, \eqref{catenaMintro} follow. The strict monotonicity of the sequence  $(\apice{M}{m}_j)_m$ follows from \eqref{SoloUltimiSeM}, since $\apice{M}{m-1}_{\!\!j}=\apice{R}{m}_{m-1}\apice{M}{m}_j<\apice{M}{m}_j$.  Similarly, one can prove the monotonicity for the sequences  $(\apice{S}{m}_j)_m$, $(\apice{R}{m}_j)_m$ and the monotonicity of  $(\apice{D}{m}_j)_m$ can be deduced from \eqref{SoloUltimiReD} and the monotonicity of  $(\apice{R}{m}_j)_m$.
\end{proof}	

\appendix
\section{Morse index conjecture for the Dirichlet Lane-Emden case}\label{App:Morse}

Let $\apice{u}{m}_{p}:=\apice{u}{m}_{0,p}$ be the radial  solution of the Lane-Emden equation \eqref{equationHenon} with $\alpha=0$ and with Dirichlet boundary conditions \eqref{DirichletBC}, having $m-1$ interior zeros. We conjecture that for $p$ sufficiently large the Morse index $\mathfrak m$ satisfies
\begin{equation}\label{conjecture}\mathfrak{m}(\apice{u}{m}_{p})=\sum_{k=0}^{m-1}\mathfrak{m}(Z_k),\end{equation}
where $Z_k:=Z_{k,0}$ are the limit profiles in \eqref{bubblecomplete} with $\alpha=0$, whose Morse indexes are known (see \cite{ChenLin}), namely,
\begin{align*}
\mathfrak{m}(Z_0)=1,\qquad \mathfrak{m}(Z_k)=1 +2\left[ \frac{\theta_k}{2}\right]\qquad \mbox{ for }k\geq 1,
\end{align*}
 where $[c]$ is the integer part of $c\in\mathbb R$. Then, by \eqref{theta:ip},
\begin{align}\label{c2}
\mathfrak{m}(Z_k)=1 +2\left[ \frac{\theta_k}{2}\right]=8k+3\qquad \text{ for all }k\in\N,\ k\geq 1.
\end{align}

As a consequence from conjecture \eqref{conjecture} one would obtain the following \emph{Morse index formula.}
\begin{conjecture}
\label{MorseFormula} 
	For any $m\geq 1$ there exists $p_m>1$ such that
	\[\mathfrak{m}(\apice{u}{m}_{p})=4m^2-m-2, \qquad \forall p\geq p_m.\]
\end{conjecture}
Indeed, simple computations would show that
 \begin{eqnarray*}\mathfrak{m}(\apice{u}{m}_{p})&\overset{\eqref{conjecture}}{=}&\sum_{k=0}^{m-1}\mathfrak{m}(Z_k) \overset{\eqref{c2}}{=} 1 + 3(m-1)+8\sum_{k=1}^{m-1} k 
=3m-2+4m(m-1)=4m^2-m-2.\end{eqnarray*}

Conjecture \ref{MorseFormula} holds in  the case $m=1$, since it can be proved that the least energy solution has Morse index $1$ (see \cite{Lin}). Moreover it has been proved in the case $m=2$ in \cite{DeMarchisIanniPacellaMathAnn}, where it is shown that $\mathfrak{m}(\apice{u}{2}_{p})=12.$ We remark that in higher dimension $N\geq 3$, the Morse index is smaller, since, in this case, the Morse index grows linearly; to be more precise, in \cite{DeMarchisIanniPacellaAdvMath} it is shown that for $N\geq 3$ and $m\geq 1$ there exists $\delta_m>0$ such that
\[\mathfrak m(\apice{u}{m}_p) = m+ N(m- 1) \qquad\mbox{ for }\  \frac{N+2}{N-2}-\delta_m\leq p<\frac{N+2}{N-2}. \]

\end{document}